\documentclass[10pt,a4paper]{amsart}
\usepackage[english]{babel}
\usepackage{amsmath}
\usepackage{url}
\usepackage{amsfonts}
\usepackage{amssymb}
\usepackage{xcolor}
\usepackage[left=2cm,right=2cm,top=2cm,bottom=2cm]{geometry}
\usepackage{enumitem}
\usepackage{nicefrac}

\usepackage[mathlines,pagewise]{lineno}
\title[Analysis of heterogeneous catalysis models with fast sorption and surface chemistry]{Analysis of some heterogeneous catalysis models with fast sorption and fast surface chemistry}
\author{Bj\"orn Augner, Dieter Bothe}
\dedicatory{Dedicated to Matthias Hieber on the occasion of his $60^\text{th}$ birthday}

 \newtheorem{theorem}{Theorem}[section]
 \newtheorem{remark}[theorem]{Remark}
 \newtheorem{proposition}[theorem]{Proposition}
 \newtheorem{corollary}[theorem]{Corollary}
 \newtheorem{lemma}[theorem]{Lemma}

 \newtheorem{assumption}[theorem]{Assumption}

 \renewcommand{\vec}{\bd}
 
 \DeclareMathOperator{\R}{\mathbb{R}}
 \DeclareMathOperator{\N}{\mathbb{N}}
 \DeclareMathOperator{\C}{\mathbb{C}}
 
 \DeclareMathOperator{\Z}{\mathbb{Z}}
 \renewcommand{\S}{\mathbb{S}}
 \DeclareMathOperator{\dv}{\operatorname{div}}
 \DeclareMathOperator{\F}{\mathbb{F}}

 \DeclareMathOperator{\B}{\mathcal{B}}
 \DeclareMathOperator{\dom}{\mathrm{D}}
 \DeclareMathOperator{\ran}{\mathrm{R}}

 \renewcommand{\ker}{\mathrm{N}}
 
 \newcommand{\cip}[2]{\left( #1 \mid #2 \right)}
 \newcommand{\ip}[2]{#1 \cdot #2}
 
 \newcommand{\bd}[1]{\boldsymbol{#1}}
 \newcommand{\bb}[1]{\boldsymbol{#1}}
 
 \renewcommand{\Re}{\operatorname{Re}\,}
 \renewcommand{\Im}{\operatorname{Im}\,}
 \newcommand{\dd}{\, \mathrm{d}}
 \newcommand{\norm}[1]{\left\| #1 \right\|} 
  
 \newcommand{\abs}[1]{\left| #1 \right|}
 
 \newcommand{\diag}{\operatorname{diag}}

 \newcommand{\ii}{\mathrm{i}}
 \newcommand{\ee}{\mathrm{e}}
 \newcommand{\BB}{\mathrm{B}}
 \newcommand{\LL}{\mathrm{L}}
 \newcommand{\WW}{\mathrm{W}}
 \newcommand{\CC}{\mathrm{C}}
 \newcommand{\DD}{\mathrm{D}}
 \newcommand{\II}{\mathrm{I}}

\begin{document}
 \allowdisplaybreaks[1]

 \begin{abstract}
  We investigate limit models resulting from a dimensional analysis of quite general heterogeneous catalysis models with fast sorption (i.e.\ exchange of mass between the bulk phase and the catalytic surface of a reactor) and fast surface chemistry for a prototypical chemical reactor.
  For the resulting reaction-diffusion systems with linear boundary conditions on the normal mass fluxes, but at the same time nonlinear boundary conditions on the concentrations itself, we provide analytic properties such as local-in time well-posedness, positivity and global existence of strong solutions in the class $\WW^{(1,2)}_p(J \times \Omega; \R^N)$, and of classical H\"older solutions in the class $\CC^{(1+\alpha, 2 + 2\alpha)}(\overline J \times \overline{\Omega})$.
  Exploiting that the model is based on thermodynamic principles, we further show a priori bounds related to mass conservation and the entropy principle.
 \end{abstract}
 
 \keywords{Heterogeneous catalysis, dimension analysis, reaction diffusion systems, surface chemistry, surface diffusion, sorption, quasi-linear PDE, $\LL_p$-maximal regularity, positivity, blow-up.}

 \maketitle
 
 Version of \today.
 
\section{Introduction}
\label{sec:introduction}
In chemical engineering catalytic processes often play an important, if not predominant role: Certain chemical reactions taking place within a chemical reactor are supposed to be accelerated, whereas other, usually undesirable, side reactions should be suppressed.
This aim can be accomplished by adding substances which catalytically act in the fluid mixture (homogeneous catalysis), or e.g.\ by using suitable structures for the reactor surface (active surface) which may then act as a catalytic surface to accelerate the desirable reactions on the surface. In many cases such heterogeneous catalysis mechanism are actually more efficient by several orders of magnitude than homogeneous catalysts, and, moreover, one may often avoid the need for filtration technology to separate the desired product from the catalyst.
Heterogeneous catalysis mechanisms and sorption processes may be modelled starting from a continuum thermodynamic viewpoint by reaction-diffusion systems in the chemical reactor and on the active surface which are coupled via sorption processes, i.e.\ the exchange of mass between the boundary layer of the bulk phase and the active surface, cf.\ \cite{SoOrMaBo19}. In accord with their purpose, catalytic accelerated chemical reactions on the surface are very fast, i.e.\ both the surface chemistry (at least for the desired reactions) as well as sorption processes take place on very small time scales.
Hence, it is natural to consider limit models, for which the surface chemistry and sorption are taken to be infinitely fast, i.e.\ surface chemistry and sorption processes are modelled as if they would attain an equilibrium configuration instantaneously.
Using a dimensionless formulation of such coupled reaction-diffusion-sorption bulk-surface systems, several of such limit models have been proposed in \cite{AugBot19+}, including a general formulation of such a fast sorption and fast surface chemistry model.
The mathematical analysis of such systems there has been accomplished for the case of a three-component system with chemical reactions of type $A_1^\Sigma + A_2^\Sigma \rightleftharpoons A_3^\Sigma$ on the surface, neglecting any bulk chemistry (the latter being no strong obstacle, and for the construction of (uniquely determined) strong solutions not a highly relevant assumption). In the present manuscript, the mathematical analysis is continued for limit systems of the same structure, but for general bulk and surface chemistry.
In particular, the results on local-in-time existence of strong solutions, positivity, first blow-up criteria as well as a-priori estimates for the solutions will be extended to the  generic case.
\newline
The paper is organised as follows: In Section \ref{sec:notation} some basic notation is introduced and some preliminary results are recalled. Thereafter, in Section \ref{sec:model}, the class of heterogeneous catalysis models considered in this manuscript is introduced and the underlying modelling assumptions recalled from the article \cite{AugBot19+}. Section \ref{sec:lit-wellposedness} constitutes the core of this article, and is splitted into subsections on $\LL_p$-maximal regularity of a linearised version of the fast-sorption-fast-surface-chemistry model, on local-in-time existence of strong $\WW_p^{(1,2)}$-solutions, and on an abstract blow-up criterion as well as a-priori bounds, e.g.\ entropy estimates, on the strong solution. The appendix then serves to sketch how to generalise the main results of \cite{DeHiPr03} and \cite{DeHiPr07} to the slightly more general setup considered here, where the system splits into parts on which Dirichlet type conditions, i.e.\ boundary conditions prescribed by differential operators of order zero, are imposed, and parts on which a no-flux boundary condition, i.e.\ given via a differential operator of order one, is used.
There is a vast amount of literature on reaction-diffusion-systems or general parabolic systems in the bulk phase with homogeneous or inhomogeneous, linear or nonlinear boundary conditions, e.g.\ \cite{LaSoUr68}, \cite{Pie10} for a start, and quite recently \emph{thermodynamic principles} have become a resourceful driving force for \emph{entropy methods}, e.g.\ \cite{DesFel06}, \cite{DesFelTan17}, \cite{DruJun19+}.
Astonishingly, however, up to now (at least to our knowledge) combined type boundary conditions, i.e.\ systems where at a fixed boundary point $\vec z \in \partial \Omega$ Dirichlet type boundary conditions are imposed on some of the variables (or, a linear combination thereof), whereas the remaining variables satisfy Neumann type boundary conditions, have not been considered in the literature yet.
Instead, other types of generalisations have been considered so far:
Several authors, e.g.\ \cite{AcqTer89}, \cite{Lu95}, \cite{Ter89}, did indeed consider parabolic systems with nonlinear boundary conditions, but these are always assumed to be of a common order, cf., e.g., the \emph{non-tangentiality condition} in \cite{Lu95}.
In, e.g., \cite{DePrZa08}, \cite{BoKoMaSa17} and \cite{Sn17} general parabolic systems or reaction-diffusion-systems with dynamic boundary conditions have been considered, i.e.\ typically two parabolic systems in the bulk phase and on the surface are coupled.
In \cite{DeKa13}, on the other hand, the authors consider more general structures leading to parabolic systems based on the notion of a \emph{Newton polygon}.

 \section{Notation and preliminaries}
 \label{sec:notation}
 Throughout, all Banach spaces appearing are Banach spaces over $\F$, the field of real numbers $\R$ or complex numbers $\C$, and $\abs{z}$ denotes the modulus of a real or complex number $z$, $\Re z$ its real and $\Im z$ its imaginary part.
 Real or complex vectors (or, vector fields) $\vec v = (v_1, \ldots, v_N)^\mathsf{T} \in \F^N$ will be typically denoted by small, Roman letters in boldface and have Euclidean norm $\abs{\vec v} = \sqrt{\sum_{i=1}^N \abs{v_i}^2}$, whereas matrices $\bb M = [m_{ij}]_{i,j} \in \R^{n \times m}$ (or $\C^{n \times m}$) most of the time are written in Roman capitals and boldface.
 The set of natural or entire numbers $k$ are denoted by $\N = \{1, 2, \ldots, \}$ (or $\N_0 = \{0, 1, 2, \ldots \}$) and $\Z = \{\ldots, -1, 0, 1, \ldots\}$, respectively, and vectors of entire numbers $\vec \alpha = (\alpha_1, \ldots, \alpha_N)^\mathsf{T} \in \Z^N$ by small Greek letters in boldface, and with length $\abs{\vec \alpha} = \sum_{i=1}^N \abs{\alpha_i}$, but $\bb \nu = [\nu_{i,j}]_{i,j} \in \Z^{N \times M}$ may also denote entire numbers-valued matrices.
 \newline
 $\Omega \subseteq \R^n$ typically denotes an open and nonempty subset of $\R^n$, $\overline{\Omega}$ its closure and $\partial \Omega$ its boundary, and $J \subseteq \R$ any interval.
 Function spaces that are frequently used are $\CC(\Omega)$ and $\CC(\overline{\Omega})$ (continuous functions on $\Omega$ and $\overline{\Omega}$, resp.), $\CC^k(\Omega)$ and $\CC^k(\overline{\Omega})$ ($k \in \N_0$ times continuous differentiable functions on $\Omega$ and $\overline{\Omega}$, resp.), $\CC^{k+\gamma}(\overline{\Omega})$ ($k \in \N_0$ times continuously differentiable functions with bounded and $\gamma \in (0,1]$ H\"older continuous derivatives of order $k$), $\LL_p(\Omega)$ (Lebesgue spaces of order $p \in [1, \infty]$, where as usual function classes are identified with its representatives), $\WW_p^k(\Omega)$ (Sobolev spaces of differentiability order $k \in \N_0$ and integrability order $p \in [1, \infty)$), $\WW_p^s(\Omega)$ (Sobolev--Slobodetskii spaces, $s \in \R_+$, $p \in [1, \infty]$), $\BB_{pq}^s$ (Besov spaces for parameters $s \in \R_+$, $p, q \in [1, \infty]$).
 Similarly, one also has their corresponding boundary (for sufficiently regular boundary), Banach space $E$-valued and anisotropic versions, e.g.\ $\LL_p(\Omega; E)$ ($E$-valued Lebesgue spaces), $\LL_p(\partial \Omega)$ (Lebesgue spaces w.r.t.\ surface measure) and
  \begin{align*}
   \CC^{(1,2m) \cdot \alpha}(\overline{J} \times \overline{\Omega})
    &= \CC^\alpha(\overline{J}; \CC^0(\overline{\Omega})) \cap \LL^\infty(\overline{J}; \CC^{2m\alpha}(\overline{\Omega})),
    &&m \in \N, \, \alpha \in \R_+,
    \\
   \WW_p^{(1,2m) \cdot s}(J \times \Omega)
    &= \WW_p^s(J; L_p(\Omega)) \cap \LL_p(J; W_p^{2m\alpha}(\Omega)),
    &&m \in \N, \, s \in \R_+
  \end{align*}
 etc.
 \begin{remark}[Sobolev-Slobodetskii spaces and Besov spaces]
  Recall that for sufficiently regular domains $\Omega \subseteq \R^n$, one has $\BB_{pp}^s(\Omega) = \WW_p^s(\Omega)$ for $s \in \R_+ \setminus \N_0$, but $\BB_p^k(\Omega) \neq \WW_p^k(\Omega)$ for $k \in \N$.
 \end{remark}
 For the definitions, basic properties and more information on these spaces, the interested reader is referred to the vast literature, e.g.\ \cite{LiLo01}, \cite{AdaFou03} and \cite{Ama83}.
 \newline
 Most importantly for our purposes, the following embedding results hold true.
  
  \begin{theorem}[Sobolev--Morrey embedding theorems]
   Let $\Omega \subseteq \R^n$ be a bounded domain, and parameters $s \geq r \in \R_+$, $p, q \in [1, \infty]$, $\alpha \in [0,1)$ and $m \in \N$ be given. Then, for sufficiently regular boundary $\partial \Omega$ of $\Omega$, e.g.\ $\partial \Omega \in \CC^\infty$, the following embeddings are continuous
    \begin{align*}
     \WW_p^s(\Omega)
      &\hookrightarrow \WW_q^r(\Omega),
      &&\text{if}
      &&s - \frac{n}{p} \geq r - \frac{n}{q} \not\in \Z,
      \text{ or } k - \frac{n}{p} > l - \frac{n}{q},
      \\
     \WW_p^s(\Omega)
      &\hookrightarrow \CC^{l+\alpha}(\overline{\Omega}),
      &&\text{if}
      &&s - \frac{n}{p} \geq r + \alpha \not\in \Z,
      \text{ or } s - \frac{n}{p} > r + \alpha,
      \\
     \WW_p^{(1,2m) \cdot s}(J \times \Omega)
      &\hookrightarrow \WW_q^{(1,2m) \cdot r}(J \times \Omega),
      &&\text{if}
      &&2m - \frac{s}{n+2} \geq 2mr - \frac{r}{n+2} \neq \Z,
      \text{ or } 2m - \frac{s}{n+2} > 2mr - \frac{r}{n+2},
      \\
     \WW_p^{(1,2m) \cdot s}(J \times \Omega)
      &\hookrightarrow \CC^{(1,2m) \cdot \frac{l + \alpha}{2m}}(J \times \Omega),
      &&\text{if}
      &&2m - \frac{s}{n+2} \geq l + \alpha \neq \Z,
      \text{ or } 2m - \frac{s}{n+2} > l + \alpha,
      \\
     \WW_p^s(\Sigma)
      &\hookrightarrow \WW_q^r(\Sigma),
      &&\text{if}
      &&s - \frac{n-1}{p} > r - \frac{n-1}{q},
      \\
     \WW_p^s(\Sigma)
      &\hookrightarrow \CC^{l+\alpha}(\Sigma),
      &&\text{if}
      &&s - \frac{n-1}{p} > l + \alpha,
      \\
     \WW_p^{(1,2m) \cdot s}(J \times \Sigma)
      &\hookrightarrow \WW_q^{(1,2m) \cdot r}(J \times \Sigma),
      &&\text{if}
      &&2m - \frac{s}{n+1} > 2mr - \frac{r}{n+1},
      \\
     \WW_p^{(1,2m) \cdot s}(J \times \Sigma)
      &\hookrightarrow \CC^{(1,2m) \cdot \frac{l + \alpha}{2m}}(J \times \Sigma),
      &&\text{if}
      &&2m - \frac{s}{n+1} > l + \alpha.
    \end{align*}
  \end{theorem}

 These embeddings will prove quite useful later on.
 
 \section{The model}
 \label{sec:model}
 In this paper, the following, rather general fast sorption, fast surface chemistry limit model will be considered:
  \begin{equation}
   \begin{cases}
    \partial_t \vec c + \dv \bb J
     = \vec r(\vec c)
     &\text{in } (0, \infty) \times \Omega,
     \\
    \vec e^k \cdot (\bb J \cdot \vec n)
     = 0
     &\text{on } (0, \infty) \times \Sigma, \, k = 1, \ldots, N - m^\Sigma =: n^\Sigma
     \\
    \exp(\vec \nu^{\Sigma,a} \cdot \vec \mu^{\vec \nu^{\Sigma,a}})
     = 1
     &\text{on } (0, \infty) \times \Sigma, \, a = 1, \ldots, m^\Sigma
   \end{cases}
   \tag{GFLM}
   \label{GFLM}
  \end{equation}
 where the appearing variables and coefficients have the following physical interpretation and relations with each other. \newline
 \emph{Thermodynamic and geometric variables and vectors:}
  \begin{itemize}
   \item
    $\vec c = (c_1, \ldots, c_N)^\mathsf{T}: \R \times \overline{\Omega} \rightarrow \R^N$ denotes the vector field of \emph{molar concentrations}, i.e.\ $c_i(t,\vec z) \in \R$ is the molar concentration of the chemical substance $A_i$ at time $t \in \R$ in position $\vec z \in \overline{\Omega}$, for $i = 1, \ldots, N$;
   \item
    $\bb J = [\vec j_1 \cdots \vec j_N]: \R \times \overline{\Omega} \rightarrow \R^{n \times N}$ for $\vec j_i: \R \times \overline{\Omega} \rightarrow \R^n$ the vector field of \emph{individual mass fluxes} of species $A_i$, $i = 1, \ldots, N$;
   \item
    $\vec n: \partial \Omega \rightarrow \R^n$, the \emph{outer normal vector field} to $\Omega$ on $\partial \Omega$;
   \item
    $\vec r(\vec c) = \sum_{a=1}^m R_a(\vec c) \vec \nu^a$, the vector field of \emph{total molar reaction rates} in the bulk phase, modelling chemical reactions given by the formal (reversible) chemical reaction equations
     \[
      \sum_{i=1}^N \alpha_i^a A_i
       \rightleftharpoons \sum_{i=1}^N \beta_i^a A_i,
       \quad a = 1, \ldots, m,
     \]
    where $\vec \alpha^a = (\alpha_1^a, \ldots, \alpha_N^a)^\mathsf{T}$, $\vec \beta^a = (\beta_1^a, \ldots, \beta_N^a)^\mathsf{T} \in \N_0^N$ and the \emph{stoichiometric vector} of the $a$-th reaction is given by $\vec \nu^a = \vec \beta^a - \vec \alpha^a \in \Z^N$.
    Moreover, $R_a(\vec c)$ denotes the effective reaction rate for the $a$-th chemical reaction in the bulk phase;
   \item
    $\vec \nu^{\Sigma,a} = \vec \beta^{\Sigma,a} - \vec \alpha^{\Sigma,a} \in \Z^N$, $a = 1, \ldots, m^\Sigma$, are the stoichiometric vectors of the \emph{surface chemical reactions}
     \[
      \sum_{i=1}^N \alpha_i^{\Sigma,a} A_i^\Sigma
       \leftrightharpoons \sum_{i=1}^N \beta_i^{\Sigma,a} A_i^\Sigma,
       \quad a = 1, \ldots, m^\Sigma,
     \] 
    where $\vec \alpha^{\Sigma,a} = (\alpha_1^{\Sigma,a}, \ldots, \alpha_N^{\Sigma,a})^\mathsf{T}$, $\vec \beta^{\Sigma,a} = (\beta_1^{\Sigma,a}, \ldots, \beta_N^{\Sigma,a})^\mathsf{T} \in \N_0^N$ for the adsorbed versions $A_i^\Sigma$ of species $A_i$.
  \end{itemize}
 \emph{Modelling assumptions:}
  \begin{itemize}
   \item
    The concentrations are assumed to be very small (cp.\ to a reference concentration $c_\mathrm{ref}$ of some solute which is not included in the model, $0 \leq \nicefrac{c_i(t,\vec z)}{c_\mathrm{ref}} \ll 1$, \emph{dilute mixture}), and the fluid in the bulk is at rest, so that Fickian diffusion,
     \[
      \vec j_i
       = - d_i \nabla c_i,
       \quad
       \text{for some \emph{diffusion coefficients} } d_i > 0, 
       \, i = 1, \ldots, N
     \]
    is a reasonable (though, not thermodynamically consistent) model for the diffusive fluxes;
   \item
    the chemical potentials $\mu_i$ in the bulk phase are modelled as an ideal mixture with
     \[
      \mu_i(\vec c, \vartheta)
       = \mu_i^0(\vartheta) + \ln x_i,
       \quad
       i = 1, \ldots, N,
     \]
    where $\mu_i^0(\vartheta)$ corresponds to some temperature-dependent equilibrium configuration and $x_i = \frac{c_i}{c}$ the scalar field of \emph{molar fractions}, where $c = \sum_{i=1}^{N+1} c_i$ is the total concentration density in the bulk phase, including the concentration of some solvent $A_{N+1}$.
    Instead of including the solvent $A_{N+1}$ in the model, we replace $c$ by some constant approximation $c_\mathrm{ref}$ to the actual total concentration $c$, so that we may consider the vector $\vec x = (x_1, \ldots, x_N)^\mathsf{T} = \nicefrac{\vec c}{c_\mathrm{ref}}$ and its dynamics instead of $\vec c$.
    Formally assuming $c_\mathrm{ref} = 1$, we then have $\mu_i(\vec c, \vartheta) = \mu_i^0(\vartheta) + \ln c_i$.
    Additionally, an \emph{isothermal} system is assumed, hence $\mu_i^0(\vartheta) = \mu_i^0 \in \R$ is simply a constant;
   \item
    the \emph{reaction velocity} $R_a(\vec c)$ of the $a$-th reaction is modelled (consistent with the entropy law) as $R_a(\vec c) = k_a^f \vec c^{\vec \alpha^a} - k_a^b \vec c^{\vec \beta^a}$ with $k_a^f, k_a^b > 0$ satisfying the relation
     \[
      \frac{k_a^f}{k_a^b}
       = \exp (\vec \nu^a \cdot \vec \mu^0),
     \]
    for $\vec \mu_0 = (\mu_i^0)_i$.
      \item
    Throughout, we assume that all equilibria of the surface chemistry are \emph{detailed-balanced}, i.e.\
    \[
     \vec \nu^{\Sigma,1}, \ldots, \vec \nu^{\Sigma,m^\Sigma} \quad \text{are linearly independent.}
    \]
   Then
    $\vec e^k \in \R^N$, $k = 1, \ldots, n^\Sigma := N - m^\Sigma$, denotes a set of linearly independent \emph{conserved quantities} under the surface chemistry, spanning the orthogonal complement of the surface chemistry stoichiometric vectors $\{\vec \nu^{\Sigma,a}: a = 1, \ldots, m^\Sigma\}$.
  \end{itemize}
  
 Under these assumptions, and the additional assumption that the sorption processes and surface chemistry take place very fast, i.e.\ on much smaller time scales than the bulk diffusion and the bulk chemistry, it has been demonstrated in \cite{AugBot19+} that \eqref{GFLM} is a reasonable limit model for the limiting case of infinitely fast surface chemistry and sorption processes (actually, independent of whether one of these two fast thermodynamic mechanism is even ultra-fast), and the condensed form of the limit model (including initial data) reads
  \begin{align}
   \partial_t c_i - d_i \Delta c_i
    &= r_i(\vec c),
    &&i =1, \ldots, N, \, \text{in } (0, \infty) \times \Omega,
    \label{eqn:condensed_form_bulk-phase}
    \\
  \vec e^k \cdot \bb D \partial_{\vec n} \vec c
    &= 0,
    &&k = 1, \ldots, n^\Sigma, \, \text{on } (0, \infty) \times \Sigma,
    \label{eqn:condensed_flux_bc}
    \\
   k_a^f \prod_{i=1} c_i^{\alpha_i^{\Sigma,a}}
    &= k_a^b \prod_{i=1}^N c_i^{\beta_i^{\Sigma,a}},
    &&a = 1, \ldots, m^\Sigma, \, \text{on } (0, \infty) \times \Sigma,
    \label{eqn:condensed_form_chemistry_bc}
    \\
   \vec c(0,\cdot)
    &= \vec c^0,
    &&\text{in } (0, \infty) \times \Omega
    \label{eqn:condensed_form_initial_data}
  \end{align}
 where $\Sigma = \partial \Omega$ denotes boundary of $\Omega$, acting as an active surface, $\bb D = \diag (d_1, \ldots, d_N)$ is the diagonal matrix of (Fickian) diffusion coefficients, and some initial data $\vec c^0: \overline{\Omega} \rightarrow \R^N$ are given.
 
 \section{Local-in time well-posedness for arbitrary bulk and surface chemistry}
 \label{sec:lit-wellposedness}
 
 This section is devoted to the local-in time well-posedness analysis for generic fast-sorption--fast-surface-chemistry limit models of the form \eqref{eqn:condensed_form_bulk-phase}--\eqref{eqn:condensed_form_initial_data}.
 
 Since, by the detailed-balance assumption on the surface chemistry, the stoichiometric vectors $\nu^\Sigma_1, \ldots, \vec \nu^\Sigma_{m^\Sigma}$ are linearly independent, this system is equivalent to the system 
  \begin{align*}
   \partial_t \vec c - \bb D \Delta \vec c
    &= \sum_a k_a \left( \exp(\vec \alpha^a \cdot \vec \mu) - \exp(\vec \beta^a \cdot \vec \mu) \right) \vec \nu^a
    &&\text{in } (0,\infty) \times \Omega,
    \\
   \ip{\vec e^k}{\bb D \partial_{\vec n} \vec c}
    &= 0,
    &&k = 1, \ldots, n^\Sigma, \, \text{on } (0,\infty) \times \Sigma,
    \\
   \exp(\vec \alpha^{\Sigma,a} \vec \mu) - \exp(\vec \beta^{\Sigma,a} \cdot \vec \mu)
    &= 0,
    &&a = 1, \ldots, m^\Sigma, \, \text{on } (0,\infty) \times \Sigma.
  \intertext{ Multiplying the latter equation by $\exp(- \vec \alpha^{\Sigma,a} \cdot \vec \mu|_\Sigma)$ gives the formulation}
   \partial_t \vec c - \bb D \Delta \vec c
    &= \sum_a k_a \left( \exp(\vec \alpha^a \cdot \vec \mu) - \exp(\vec \beta^a \cdot \vec \mu) \right) \vec \nu^a
    &&\text{in } (0,\infty) \times \Omega,
    \\
   \ip{\vec e^k}{\bb D \partial_{\vec n} \vec c}
    &= 0,
    &&k = 1, \ldots, n^\Sigma, \, \text{on } (0,\infty) \times \Sigma,
    \\
   1 - \exp(\vec \nu^{\Sigma,a} \cdot \vec \mu)
    &= 0,
    &&a = 1, \ldots, m^\Sigma, \, \text{on } (0,\infty) \times \Sigma.
  \end{align*}
 Here, we insert the special choice $\mu_i = \mu_i^0 + \ln c_i$ for the bulk chemical potentials to obtain the system
  \begin{align}
   \partial_t \vec c - \bb D \Delta \vec c
    &= \sum_a k_a \left( \exp(\vec \alpha^a \cdot \vec \mu) - \exp(\vec \beta^a \cdot \vec \mu) \right) \vec \nu^a
    &&\text{in } (0,\infty) \times \Omega,
    \nonumber \\
   \ip{\vec e^k}{\bb D \partial_{\vec n} \vec c}
    &= 0,
    &&k = 1, \ldots, n^\Sigma, \, \text{on } (0,\infty) \times \Sigma,
    \label{eqn:FSFSC-limit}
    \\
   \vec c^{\vec \nu^{\Sigma,a}}
    &= \kappa_a,
    &&a = 1, \ldots, m^\Sigma, \, \text{on } (0,\infty) \times \Sigma,
    \nonumber
  \end{align}
  for the equilibrium constant $\kappa_a = \exp(- \vec \nu^{\Sigma,a} \cdot \vec \mu^0)$.
 A possible linearised (around some $\vec c^\ast: \overline{\Omega} \rightarrow (0, \infty)$) version of this nonlinear system is obtained by taking the partial derivatives $\frac{\partial}{\partial c_i^\Sigma} \vec c|_\Sigma^{\vec \nu^{\Sigma,a}} = \nu_i^{\Sigma} \frac{1}{c_i|_\Sigma} \vec c|_\Sigma^{\vec \nu^{\Sigma,a}}$ for $c_i|_\Sigma \neq 0$, and reads
  \begin{align*}
   \partial_t \vec v - \bb D \Delta \vec v
    &= \vec f
    &&\text{in } (0,\infty) \times \Omega,
    \\
   \ip{\vec e^k}{\bb D \partial_{\vec n} \vec v}
    &= g_k,
    &&k = 1, \ldots, n^\Sigma, \, \text{on } (0,\infty) \times \Sigma,
    \\
   \sum_{i=1}^N \nu_i^{\Sigma,a} \frac{v_i}{c_i^\ast} (\vec c^\ast)^{\vec \nu^{\Sigma,a}}
    &= h_a,
    &&a = 1, \ldots, m^\Sigma, \, \text{on } (0,\infty) \times \Sigma.
  \end{align*}
 or, for short
  \begin{align*}
   \partial_t \vec v - \bb D \Delta \vec v
    &= \vec f,
    &&\text{in } (0,\infty) \times \Omega,
    \\
   \ip{\vec e^k}{\bb D \partial_{\vec n} \vec v|_\Sigma}
    &= g_k,
    &&k = 1, \ldots, n^\Sigma, \, \text{on } (0,\infty) \times \Sigma,
    \label{eqn:LP}
    \tag{LP} \\
   \bb C_a \vec \nu^{\Sigma,a} \cdot \vec v|_\Sigma
    &= h_a,
    &&a = 1, \ldots, m^\Sigma, \, \text{on } (0,\infty) \times \Sigma.
  \end{align*}
 where $\bb C_a = (\vec c^\ast|_\Sigma)^{\vec \nu^{\Sigma,a}} \diag(c_i^\ast|_\Sigma^{-1})_{i=1}^N: \, \Sigma \rightarrow \R^{N \times N}$.
  \begin{remark}
   Since only concentrations $c_i, c_i^\Sigma \geq 0$ have physical significance, only linearisations around states for which all components are (uniformly) strictly positive, i.e.\ only uniformly strict positive initial data, will be feasible by this approach towards linearisation. Regularisation effects for reaction-diffusion systems (cf.\ the strict parabolic maximum principle), however, suggest that this is no severe restriction as (under slight structural assumptions on the reaction network) for any initial data $c_i^0 \geq 0$, but not identically zero, the solution immediately becomes strictly positive, cf.\ the strict maximum principle for reaction-diffusion equations.
  \end{remark}
 The program of the remainder of this section is the following:
  \begin{itemize}
   \item
    Show $\LL_p$-maximal regularity of the linearised problem, provided sufficient regularity of the reference function $\vec c^\ast$.
    This can be done based on abstract theory in a slightly extended version of the results in \cite{DeHiPr03} and \cite{DeHiPr07}, so that mainly the validity of the Lopatinskii--Shapiro condition and regularity properties have to be checked.
   \item
    Use $\LL_p$-maximal regularity of the linearised problem and the contraction mapping principle to establish local-in-time existence for the fast-sorption--fast-surface-chemistry limit, provided the initial data are regular enough, have uniformly strictly positive entries and satisfy suitable compatibility conditions.
   \item
    Moreover, this procedure will give a "natural" blow-up criterion for global-in time existence, where by "natural" it is meant that this norm appears in the contraction mapping argument for the local-in time existence.
%   \item
%    Lastly, one copes with non-negative, but not necessarily strictly positive initial data.
  \end{itemize}
  
 \subsection{$\LL_p$-maximal regularity for the linearised problem}
 \label{subsec:maximal_regularity}
 To show that the linearised problem \eqref{eqn:LP} possesses the property of $\LL_p$-maximal regularity, let us first consider the case of constant coefficients (i.e.\ a constant reference function $\vec c^\ast \in (0, \infty)^N$) and a flat boundary (i.e.\ consider the special case of a boundary $\Sigma = \R^{n-1} \times \{0\}$ for the half-space domain $\Omega = \R^{n-1} \times (0, \infty)$).
 The corresponding linear initial-boundary value problem to be investigated,  on the half-space then takes the general form of a parabolic reaction-diffusion system with boundary conditions of inhomogeneous type.
 For technical reasons (unboundedness of the domain $\Omega$) shifted by a damping constant $\mu \geq 0$, hence, it reads as
  \begin{align}
   \partial_t \vec v - \bb D \Delta \vec v + \mu \vec v
    &= \vec f,
    &&(t, \vec z) \in J \times \Omega = (0, \infty) \times \R^{n-1} \times (0,\infty)
    \nonumber \\
   \ip{\vec e^k}{\bb D \frac{\partial}{\partial z_n} \vec v}
    &= g_k,
    &&k = 1, \ldots, n^\Sigma, \, (t, \vec z') \in J \times \Sigma = (0, \infty) \times \R^{n-1} \times \{0\}
    \label{eqn:half-space_constant_coeff}
    \\
   \bb C \vec \nu^{\Sigma,a} \cdot \vec v
    &= h_a,
    &&a = 1, \ldots, m^\Sigma, \, (t,\vec z') \in J \times \Sigma
    \nonumber \\
   \vec v(0,\cdot)
    &= \vec v^0,
    &&\vec z \in \R^{n-1} \times (0, \infty)
    \nonumber
  \end{align}
 where one writes $\vec z = (\vec z', z_n) \in \R^{n-1} \times \R_+$ for the spatial variables,  and the right hand sides -- as the analysis of the linearised problem will reveal -- have to satisfy the following regularity conditions:
  \begin{enumerate}
   \item
    $\vec f \in \LL_p(J \times \Omega;\R^N)$;
   \item
    $\vec g \in \WW_p^{(1,2) \cdot (\nicefrac{1}{2}-\nicefrac{1}{2p})} (J \times \Sigma; \R^{n^\Sigma})$;
   \item
    $\vec h \in \WW_p^{(1,2) \cdot (1 - \nicefrac{1}{2p})}(J \times \Sigma; \R^{m^\Sigma})$;
   \item
    $\vec v^0 \in \BB_{pp}^{2-2/p}(\Omega; \R^N)$.
  \end{enumerate}
  
  The corresponding maximal regularity result on the half-space reads as follows:
  
  \begin{proposition}[$\LL_p$-maximal regularity on the half-space for constant coefficients]
   Assume that $\vec c^\ast \in (0, \infty)^N$ is a constant and let $p \in (1, \infty)$.
   Then there is $\mu_0 \geq 0$ such that for all $\mu \geq \mu_0$, the half-space problem \eqref{eqn:half-space_constant_coeff} admits a unique solution $\vec v \in \WW_p^{(1,2)}(J \times \Omega; \R^N)$ if and only if $\vec f \in \LL_p(J \times \Omega; \R^N)$, $\vec g \in \WW_p^{(1,2) \cdot (\nicefrac{1}{2}-\nicefrac{1}{2p})}(J \times \Sigma; \R^{n^\Sigma})$, $\vec h \in \WW_p^{(1,2) \times (1 - \nicefrac{1}{2p})}(J \times \Sigma; \R^{m^\Sigma})$ and $\vec v^0 \in \BB_{pp}^{2-2/p}(\Omega; \R^N)$, and, moreover, the following compatibility conditions are satisfied, if the respective time traces exist:
    \begin{align*}
   \ip{\vec e^k}{\bb D \frac{\partial}{\partial_{z_n}} \vec v^0}
    &= g_k|_{t=0},
    &&k = 1, \ldots, n^\Sigma, \, \vec z' \in \R^{n-1}
    &&(\text{if } p > n + 2),
    \nonumber \\
   \bb C_a \vec \nu^{\Sigma,a} \cdot \vec v^0
    &= h_a|_{t=0},
    &&a = 1, \ldots, m^\Sigma, \, \vec z' \in \R^{n-1}
    &&(\text{if } p > \frac{n + 2}{2}).
    \end{align*}
    In this case, there is $C = C(p, \mu) > 0$, independent of the boundary and initial data, such that
     \[
      \norm{\vec v}_{\WW_p^{(1,2)}}
       \leq C(p,\mu) \left( \norm{\vec f}_{\LL_p} + \norm{\vec g}_{\WW_p^{(1,2) \cdot (\nicefrac{1}{2} - \nicefrac{1}{2p})}} + \norm{\vec h}_{\WW_p^{1 - \nicefrac{1}{2p}}} + \norm{\vec v^0}_{\BB_{pp}^{2-2/p}} \right).
     \]
  \end{proposition}
  \begin{proof}
   Starting with the system of PDEs \eqref{eqn:half-space_constant_coeff}, taking the partial Laplace-Fourier transform $\mathcal{F}$ for $(\lambda, \vec \xi') \in \overline{\C_0^+} \times \R^{n-1}$, and setting $y = z_n$, leads (formally) to the following parameter-dependent initial value problems
  \begin{align}
   (\lambda + d_i \abs{\vec \xi'}^2) \hat v_i - d_i \frac{\partial^2}{\partial y^2} \hat v_i + \mu \hat v_i(\lambda, \vec \xi', y)
    &= \hat f_i(\lambda, \vec \xi', y),
    &&i = 1, \ldots, N, \, (\lambda, \vec \xi', y) \in \overline{\C_0^+} \times \R^{n-1} \times (0, \infty)
    \nonumber \\
   \ip{\vec e^k}{\bb D \frac{\partial}{\partial_{y}} \hat {\vec v}(\lambda, \vec \xi',0)}
    &= \hat g_k(\lambda, \vec \xi'),
    &&k = 1, \ldots, n^\Sigma, \, (\lambda, \vec \xi') \in \overline{\C_0^+} \times \R^{n-1},
    \label{eqn:half-space_fourier}
    \\
   \bb C_a \vec \nu^{\Sigma,a} \cdot \hat{\vec v}(\lambda, \vec \xi')
    &= \hat h_a(\lambda, \vec \xi'),
    &&a = 1, \ldots, m^\Sigma, \, (\lambda, \vec \xi') \in \overline{\C_0^+} \times \R^{n-1},
    \nonumber \\
   \hat{\vec v}(0,\vec \xi, y)
    &=\hat {\vec v}^0(\vec \xi, y),
    &&(\vec \xi', y) \in \R^{n-1} \times (0, \infty),
    \nonumber
  \end{align}
 where $\hat v_i = \mathcal{F} v_i$, $\hat f_i = \mathcal{F} f_i$ etc.
 For fixed $(\lambda, \vec \xi) \in \overline{\C_0^+} \times \R^{n-1} \setminus \{(0,\vec 0)\}$, the general solution to the ODE in the first line of this system is
   \[
    v_i(\lambda, \vec \xi', y)
     = \exp \left( (- (\tfrac{\lambda + \mu}{d_i} + \abs{\vec \xi'}^2)^{\nicefrac{1}{2}}) y \right) \hat v_{i,+}(\lambda, \vec \xi')
      +  \exp \left( (- (\tfrac{\lambda + \mu}{d_i} + \abs{\vec \xi'}^2)^{\nicefrac{1}{2}} y \right)) \hat v_{i,-}(\lambda, \vec \xi').
   \]
 As we look for a solution in the class $\vec v_i = \mathcal{F}^{-1}(\hat v_i) \in \LL_p((0,T); \LL_p(\R^{n-1} \times (0,\infty))$, one necessarily needs to have $\hat v_{i,+}(\lambda, \vec \xi') = 0$ for a.e.\ $(\lambda, \vec \xi') \in \overline{\C_0^+} \times \R^{n-1}$, and in that case
  \[
   d_i \frac{\partial}{\partial y} \hat v_i(\lambda, \vec \xi',0)
    = - d_i \left( (\tfrac{\lambda + \mu}{d_i} + \abs{\vec \xi'}^2)^{\nicefrac{1}{2}} \right) \hat v_i(\lambda, \vec \xi', 0).
  \]
 To match the solution with the boundary conditions at $\Sigma = \R^{n-1} \times \{0\}$, we thus need to solve the linear systems
  \begin{align*}
   - \ip{\vec e^k}{\bb D \bb R \vec v(\lambda, \vec \xi', 0)}
    &= \hat g_k(\lambda, \vec \xi),
    &&k = 1, \ldots, n^\Sigma, \, (\lambda, \vec \xi') \in \overline{\C_0^+} \times \R^{n-1},
    \\
   \bb C_a \vec \nu^{\Sigma,a} \cdot \vec v(\lambda, \vec \xi',0)
    &= \hat h_a(\lambda, \vec \xi'),
    &&a = 1, \ldots, m^\Sigma, \, (\lambda, \vec \xi') \in \overline{\C_0^+} \times \R^{n-1},
  \end{align*}
 where $\bb R = \diag \left( (\tfrac{\lambda + \mu}{d_i} + \abs{\vec \xi}^2)^{\nicefrac{1}{2}} \right)_{i=1}^N$. Writing $\vec w = \bb D^{-1} \bb R^{-\ast} \vec v$, this system is uniquely solvable for all $(\lambda, \vec \xi') \in \overline{\C_0^+} \times \R^{n-1}$ if and only if the matrix
  \[
   \left[ \begin{array}{c}
    (\vec e^1)^\mathsf{T} \\
    \vdots \\
    (\vec e^{n^\Sigma})^\mathsf{T} \\
    (\bb D^{-1} \bb R^{-\ast} \bb C_1 \vec \nu^{\Sigma,1})^\mathsf{T} \\
    \vdots \\
    (\bb D^{-1} \bb R^{-\ast} \bb C_m \vec \nu^{\Sigma,m^\Sigma})^\mathsf{T}
   \end{array} \right]
   \in \C^{N \times N}
   \quad
   \text{is invertible}.
  \]
 Since all matrices $\bb D$, $\bb R$ and $\bb C_a$ are diagonal and the matrices
  \[
   \bb C_a
    = (\vec c^\ast|_\Sigma)^{\vec \nu^{\Sigma,a}} \tilde {\bb C},
    \quad
    a = 1, \ldots, m
  \]
 coincide with the $N \times N$ diagonal matrix $\tilde C = \diag (c_i^\ast|_\Sigma)_{i=1, \ldots, N}^{-1}$ up to a non-zero factor $(\vec c^\ast|_\Sigma)^{\vec \nu^{\Sigma,a}}$ this matrix is invertible if and only if
  \[
   \left[ \begin{array}{c}
    (\tilde {\bb D} \vec \nu^{\Sigma,1})^\mathsf{T} \\
    \vdots \\
    (\tilde {\bb D} \vec \nu^{\Sigma,m^\Sigma})^\mathsf{T} \\
    (\vec e^1)^\mathsf{T} \\
    \vdots \\
    (\vec e^{n^\Sigma})^\mathsf{T} \\
   \end{array} \right]
   \in \C^{N \times N}
   \quad
   \text{is invertible},
  \]
 where $\tilde {\bb D} = \bb D^{-1} \bb R^{-\ast} \tilde {\bb C}$ is an $N \times N$ diagonal matrix with entries $[\tilde {\bb D}]_{ii} \in \C_0^+$ on the diagonal.
 Due to Lemma \ref{lem:invertibility} below, this is the case.
 Using standard properties of the partial Fourier-Laplace transform and suitable multiplier theorems, it now follows that there is $\mu_0 \geq 0$ such that for all $\mu \geq \mu_0$, the system \eqref{eqn:half-space_constant_coeff} has a unique solution in the class $\vec v \in \WW_p^{(1,2)}(J \times \Omega; \R^N)$ if and only if the initial data are subject to the regularity and compatibility conditions listed in the proposition.
 This solves the problem of $\LL_p$-maximal regularity for the constant coefficient case on the half-space.
\end{proof}

 \begin{lemma}
 \label{lem:invertibility}
  Let $s,m \in \N$ and $N = s + m$.
  Let $\vec v^1, \ldots, \vec v^m \in \R^N$ and $\vec w^1, \ldots, \vec w^s \in \R^N$ be linearly independent, real vectors such that
   \[
    \ip{\vec v^i}{\vec w^j}
     = 0,
     \quad
     i = 1, \ldots, m, \,
     j = 1, \ldots, s.
   \]
  Let $\delta_j \in \C$, $j = 1, \ldots, N$, be such that $0 \not \in \operatorname{conv} \{\delta_j: j = 1, \ldots, N\}$ and the matrix $\bb M \in \C^{N \times N}$ be defined as
   \[
    \bb M
     = \left[ \begin {array}{c}
     (\bb D \vec v^1)^\mathsf{T} \\ \vdots \\ (\bb D \vec v^m)^\mathsf{T} \\ (\vec w^1)^\mathsf{T} \\ \vdots \\ (\vec w^s)^\mathsf{T}
     \end{array} \right]
     \quad
     \text{where}
     \quad
     \bb D = \diag (\delta_1, \ldots, \delta_N) \in \C^{N \times N}.
   \]
  Then $\bb M$ is invertible, i.e.\ $\bb D \vec v^1$, \ldots, $\bb D \vec v^m$, $\vec w^1$, \ldots, $\vec w^s$ form a basis of $\C^N$.
 \end{lemma}
 \begin{proof}
  As $\bb M \in \C^{n \times n}$ is a square matrix, it suffices to demonstrate injectivity of $\bb M$.
  Let $\vec u \in \ker (\bb M)$. Then, in particular,
   \[
    0
     = [\bb M \vec u]_{m+j}
     = \sum_{i=1}^N w_i^j u_i,
     \quad
     \text{i.e.\ } \vec u \bot \vec w^j,
     \quad
     j = 1, \ldots, s.
   \]
  Therefore, there are $\vec \gamma_i \in \C$, $i = 1, \ldots, m$, such that
   \[
    \vec u
     = \sum_{i=1}^m \gamma_i \vec v^i.
   \] 
  Writing $\bb V = \left[ \begin{array}{ccc} \vec v^1 & \cdots & \vec v^m \end{array} \right] \in \bb \R^{N \times m}$, from $\bb M \vec u = \vec 0$ it follows $\bb V^\mathsf{T} \bb D \vec u = \vec 0$, thus
   \[
    \vec 0
     = \sum_{i=1}^m \gamma_i \bb V^\mathsf{T} \bb D \vec v^i
     = \bb V^\mathsf{T} \bb D \sum_{i=1}^m \gamma_i \vec v^i
     = \bb V^\mathsf{T} \bb D \bb V \vec \gamma,
     \quad
     \text{for }
     \vec \gamma = (\gamma_1, \ldots, \gamma_m)^\mathsf{T} \in \C^m.
   \]
  In particular, since $\bb V^\mathsf{T} = \bb V^\ast$ (as $\bb V$ has real entries), for the inner product on $\C^m$ one finds
   \[
    0
     = \cip{\bb V^\mathsf{T} \bb D \bb V \vec \gamma}{\vec \gamma}_{\C^m}
     = \cip{\bb D \bb V \vec \gamma}{\bb V \vec \gamma}_{\C^m}
     = \sum_{i=1}^m \delta_i \abs{(\bb V \vec \gamma)_i}^2.
   \]
  As $0 \not\in \operatorname{conv} \{\delta_i: i = 1, \ldots, m\}$, and $\abs{(\bb V \vec \gamma)_i}^2 \geq 0$, this can only hold true if $\vec u = \bb V \vec \gamma = \vec 0$, and $\bb M$ must be injective.
 \end{proof}
 
 The general $\LL_p$-maximal regularity theorem (for bounded $\CC^2$-domains $\Omega$) can then be derived via the standard technique, i.e.\ first a generalisation to the bend space problem and, thereafter, a localisation procedure. For these techniques to work properly, one needs additional conditions on the (then non-constant) reference function $\vec c^\ast: \overline{\Omega} \rightarrow \R^N$.
 In this particular case, the bulk diffusion operator $- \bb D \Delta$ does not depend on the spatial position $\vec z \in \Omega$. Therefore, there is no need to consider perturbations of it, i.e.\ $\mathcal{A}^{sm} = 0$ in the language of \cite{DeHiPr03}, \cite{DeHiPr07}. Neither do the conserved quantities $\vec e^k$, $k = 1, \ldots, n^\Sigma$, but only the matrix $\bb C(\vec z) = \diag (\vec c^\ast|_\Sigma)^{-1}(\vec z)$ (which is $\bb C_a$ up to a $a$-dependent factor $(\vec c^\ast|_\Sigma)^{\vec \nu^{\Sigma,a}}$) depend on the spatial position $\vec z \in \overline{\Omega}$.
 Using the same strategy as in \cite{DeHiPr03}, one may write
  \[
   \bb C(\vec z) \vec \nu^{\Sigma,a} \cdot \vec v|_\Sigma
    = \bb C(\vec z_0) \vec \nu^{\Sigma,a} \cdot \vec v + (\bb C(\vec z) - \bb C(\vec z_0)) \cdot \vec v
    =: \bb C(\vec z_0) \vec \nu^{\Sigma,a} \cdot \vec v + \bb C^\mathrm{sm}(\vec z) \vec \nu^{\Sigma,a} \cdot \vec v,
  \]
  where $\bb C^\mathrm{sm}(\vec z)$ corresponds to a \emph{small} perturbation of $\bb C(\vec z_0)$. As in \cite{DeHiPr03}, one then considers the problem
   \begin{align*}
    (\lambda - \bb D \Delta) \vec v(\lambda, \vec z)
     &= f(\lambda,\vec z),
     &&\lambda \in \overline{\C_0^+}, \, \vec z \in \R^{n-1} \times (0, \infty)
     \\
    \ip{\vec e^k}{\bb D \frac{\partial}{\partial y} \hat {\vec v}(\lambda, \vec z',0)}
     &= 0,
     &&k = 1, \ldots, n^\Sigma, \, \lambda \in \overline{\C_0^+}, \, \vec z' \in \R^{n-1}
     \\
    \bb C \vec \nu^{\Sigma,a} \cdot \vec v(\lambda, \vec z',0)
     &= - \bb C^\mathrm{sm} \vec \nu^{\Sigma,a} \cdot \vec v(\lambda, \vec z',0),
     && \lambda \in \overline{\C_0^+}, \, \vec z' \in \R^{n-1}.
   \end{align*}
 Letting $A_0 = - \bb D \Delta$ on the domain \[\dom(A_0) = \{ \vec v \in \WW_p^2(\Omega; \R^N): \, \ip{\vec e^k}{\bb D \frac{\partial}{\partial y} \hat {\vec v}|_\Sigma}= 0 \quad (k = 1, \ldots, n^\Sigma), \, \bb C \vec \nu^{\Sigma,a} \cdot \vec v|_\Sigma = 0 \quad (a = 1, \ldots, m^\Sigma) \} \]
 one then needs to derive a fixed point equation of the form
  \[
   \vec v
    = (\lambda + A_0)^{-1} f - \sum_{j=1}^m S_\lambda^j (\bb C^\mathrm{sm} \vec \nu^{\Sigma,a}) \cdot \vec v|_\Sigma.
  \]
 The regularity assumptions in \cite{DeHiPr07} now suggest that one should demand the following regularity of $\bb C$, hence of $\vec c^\ast$:
  \begin{quote}
   There are $s,r \geq p$ with $\tfrac{1}{s} + \tfrac{n-1}{2 r} < 1 - \tfrac{1}{2p}$ such that
    \[
     C
      \in \WW_{s,r}^{(1,2) \cdot (1-\nicefrac{1}{2p})}(J \times \Sigma; \R^{N \times N})
      := \WW_s^{1-\nicefrac{1}{2p}}(J;\LL_r(\Sigma;\R^{N \times N})) \cap \LL_s(J; \WW_r^{2-1/p}(\Sigma; \R^{N \times N})),
    \]
   hence the reference function $\vec c^\ast$ should be \emph{uniformly positive} and
    \[
     \vec c^\ast
      \in \WW_{s,r}^{(1,2) \cdot (1-\nicefrac{1}{2p})}(J \times \Sigma; \R^N)
      := \WW_s^{1-\nicefrac{1}{2p}}(J;\LL_r(\Sigma;\R^N)) \cap \LL_s(J; \WW_r^{2-1/p}(\Sigma; \R^N)).
    \]
  \end{quote}
 \begin{remark}
  Note that for regularity of $\bb C$, one has to consider the regularity of the functions $(c_k^{\ast})^{-1}$, hence of
   \[
    \frac{\partial}{\partial t} \frac{1}{c_k^\ast}
     = - \frac{\partial_t c_k^\ast}{(c_k^{\ast})^2},
     \quad
    \frac{\partial}{\partial z_i} \frac{1}{c_k^\ast}
     = - \frac{\partial_{z_i} c_k^\ast}{(c_k^{\ast})^2},
     \quad
    \frac{\partial^2}{\partial z_i \partial z_j} \frac{1}{c_k^\ast}
     = - \frac{\partial_{z_i} \partial_{z_j} c_k^\ast}{(c_k^{\ast})^2} - 2 \frac{\partial_{z_i} c_k^\ast \partial_{z_j} c_k^\ast}{(c_k^\ast)^3}.
   \]
 \end{remark}
 Since the reference function should generally lie in the function class $\vec c^\ast \in \WW_p^{(1,2) \cdot (1-\nicefrac{1}{2}p)}(J \times \Sigma; \R^N)$, one naturally should take $s = r = p$, and then the condition on the values of $s$ and $r$ reads
  \[
   \frac{n+1}{2p} < 1 - \frac{1}{2p}
    \quad \Leftrightarrow \quad
   p > \frac{n+2}{2}.
  \]
 Note that in this case the embeddings $\WW_p^{(1,2)}(J \times \Omega) \hookrightarrow \CC(\overline{J} \times \overline{\Omega})$ and $\BB_{pp}^{2-2/p}(\Omega) \hookrightarrow \CC(\overline{\Omega})$ are continuous.
 
 \begin{theorem}
  Let $p > \frac{n+2}{2}$ and $\vec c^\ast \in \WW_p^{(1,2)}(J \times \Omega; (0,\infty)^N)$, where $J = (0,T)$ for some $T > 0$ and $\Omega \subseteq \R^n$ is a bounded $\CC^2$-domain.
  Then, the linearised problem
   \begin{align*}
    \partial_t \vec v - \bb D \Delta \vec v
     &= \vec f,
     &&(t, \vec z) \in (0, \infty) \times \Omega,
     \\
    \ip{\vec e^k}{\bb D \frac{\partial}{\partial \vec n} \vec v}
     &= g_k,
     &&k = 1, \ldots, n^\Sigma, \, (t, \vec z) \in (0, \infty) \times \Sigma,
     \\
    \bb C \vec \nu^{\Sigma,a} \cdot \vec v
     &= h_a,
     &&a = 1, \ldots, m^\Sigma, \, (t, \vec z) \in (0, \infty) \times \Sigma
     \\
    \vec v(0,\cdot)
     &= \vec v^0,
     &&\vec z \in \Omega
   \end{align*}
  has a unique solution in the class $\vec v \in \WW_p^{(1,2)}(J \times \Omega; \R^N)$ if and only if the data are subject to the following regularity and compatibility conditions:
  \begin{enumerate}
   \item
    $\vec f \in \LL_p(J \times \Omega;\R^N) = \LL_p(J; \LL_p(\Omega; \R^N))$,
   \item
    $\vec g \in \WW_p^{(1,2) \cdot (\nicefrac{1}{2}-\nicefrac{1}{2p})} (J \times \Sigma; \R^{n^\Sigma}) = \WW_p^{\nicefrac{1}{2} - \nicefrac{1}{2p}}(J;\LL_p(\Sigma; \R^{n^\Sigma})) \cap \LL_p(J; \WW_p^{1-1/p}(\Sigma; \R^{n^\Sigma})$,
   \item
    $\vec h \in \WW_p^{(1,2) \cdot (1 - \nicefrac{1}{2p}}(J \times \Sigma) = \WW_p^{1-\nicefrac{1}{2p}}(J;\LL_p(\Sigma; \R^{m^\Sigma})) \cap \LL_p(J; \WW_p^{2-1/p}(\Sigma;\R^{m^\Sigma}))$,
   \item
    $\vec v^0 \in \BB_{pp}^{2-2/p}(\Omega; \R^N)$ (which is $\WW^{2-2/p}_p(\Omega;\R^N)$, unless $n=1$ and $p = 2$),
   \item
    $\bb C_a^0 \vec \nu^{\Sigma,a} \cdot \vec v^0|_\Sigma = h_a|_{t=0}$,
   \item
    $\ip{\vec e^k}{\bb D \frac{\partial}{\partial \vec n} \vec v^0|_\Sigma} = g_k|_{t=0}$ (if $p > n + 2$).
  \end{enumerate}
  In this case, the solution depends continuously on the data, i.e.\ for some $C = C(n,\Omega) > 0$ independent of the data it holds that
   \[
    \norm{\vec v}_{\WW_p^{(1,2)}}
     \leq C \left( \norm{\vec f}_{\LL_p} + \norm{\vec g}_{\WW_p^{(1,2) \cdot (\nicefrac{1}{2}-\nicefrac{1}{2p})}} + \norm{\vec h}_{\WW_p^{(1,2) \cdot (1-\nicefrac{1}{2p})}} + \norm{\vec c^0}_{\BB_{pp}^{2-2/p}(\Omega)} \right).
   \]
 \end{theorem}
 
As for many semi-linear systems, $\LL_p$-maximal regularity, and in this case $\LL_p$-$\LL_q$-estimates can be employed to find a (unique) strong solution of the quasi-linear problem. This will be the aim of the next subsection.
   
 \subsection{Local-in time existence of strong solutions, blow-up criteria and a-priori bounds}
 \label{subsection:existence_blowup_aprioribounds}
 
 Having maximal regularity for the linearised problem at hand, we may now come back to the non-linear problem. $\LL_p$-maximal regularity will play the typically crucial role in the construction of strong solutions via the contraction mapping principle. Thereby, rather as a by-product, a condition for continuation of the solution to a maximal solution will be established, where in general the solution may be global in time (which might be expected for a large subclass of the systems considered here), or one may observe a blow-up in finite time. Whereas the boundedness of the $\BB_{pp}^{2-2/p}$-norm cannot be guaranteed in general, or more precisely, it is unclear whether boundedness holds true without any restriction on the bulk and surface chemistry, for slightly weaker norms a-priori bounds are possible, indeed. The latter will be considered in the second part of this subsection.

 \subsubsection{Local-in time existence and maximal continuation of solutions}
 
 Recall the form of the fast-sorption--fast-surface-chemistry limit \eqref{eqn:FSFSC-limit}, and additionally consider a given initial datum $\vec c^0: \overline{\Omega} \rightarrow \R^N$ which should be regular enough (in a sense to be made precise later on):
  \begin{align*}
   \partial_t \vec c - \bb D \Delta \vec c
    &= \sum_a \left( k_a^f \vec c^{\vec \alpha^a} - k_a^b \vec c^{\vec \beta^a} \right) \vec \nu^a,
    &&(t, \vec z) \in (0, \infty) \times \Omega,
    \\
   \ip{\vec e^k}{\bb D \partial_{\vec n} \vec c}
    &= 0,
    &&k = 1, \ldots, n^\Sigma, \, (t, \vec z) \in (0, \infty) \times \Sigma,
    \tag{\ref{eqn:FSFSC-limit}'}
    \label{eqn:FSFSC-limit'}
    \\
   \vec c^{\vec \nu^{\Sigma,a}}
    &= \exp(- \vec \nu^{\Sigma,a} \cdot \vec g(T)),
    &&a = 1, \ldots, m^\Sigma, \, (t, \vec z) \in (0, \infty) \times \Sigma,
    \\
   \vec c|_{t=0}
    &= \vec c^0,
    &&\vec z \in \Omega.
  \end{align*}
 Introducing $\vec v(t,\vec z) := \vec c(t,\vec z) - \vec c^0(\vec z)$ leads to the reaction-diffusion-sorption system for $\vec v$ as follows:
  \begin{align*}
   \partial_t \vec v - \bb D \Delta \vec v
    &= \sum_a \left( k_a^f \left[ (\vec v + \vec c^0)^{\vec \alpha^a}  - (\vec c^0)^{\vec \alpha^a} \right] - k_a^b \left[ (\vec v + \vec c^0)^{\vec \beta^a} - (\vec c^0)^{\vec \beta^a} \right] \right) \vec \nu^a,
    &&(0, \infty) \times \Omega,
    \\
   \ip{\vec e^k}{\bb D \partial_{\vec n} \vec v}
    &= - \ip{\vec e^k}{\bb D \partial_{\vec n} \vec c^0},
    &&(0, \infty) \times \Sigma, \, k = 1, \ldots, n^\Sigma
    \\
   \bb C_a^0 \vec \nu^{\Sigma,a} \cdot \vec v
    &= [\bb C_a(\vec c^0 + \vec v) - \bb C_a(\vec c^0)] \vec \nu^{\Sigma,a} \cdot (\vec v + \vec c^0)
    &&(0, \infty) \times \Sigma, \, a = 1, \ldots, m^\Sigma
    \\
   \vec v|_{t=0}
    &= \vec 0,
    &&\Omega,
  \end{align*}
 where $\bb C_a^0 = (\vec c^0)^{- \vec \nu^{\Sigma,a}} \diag (\vec c^0)^{-1}: \Sigma \rightarrow \R^{N \times N}$.
 Next, assume that $p > \tfrac{n + 2}{2}$, and $\vec v^0 \in \BB_{pp}^{2-2/p}(\Omega)$ with
  \[
   \ip{\vec e^k}{\bb D \partial_{\vec n} \vec c^0|_\Sigma} = 0, \, k = 1, \ldots, n^\Sigma,
    \quad \text{and} \quad
   \vec v^0|_\Sigma^{\vec \nu^{\Sigma,a}} = \exp(- \vec \nu^{\Sigma,a} \cdot \vec g(T)), \, a = 1, \ldots, m^\Sigma
  \]
 and for $\tau > 0$ consider the linear solution operator
  \[
   \Phi_\tau: \DD_0 \subseteq \WW_p^{(1,2)}((0,\tau) \times \Omega; \R^N) \rightarrow \WW_p^{(1,2)}((0,\tau) \times \Omega; \R^N)
  \]
 given by $\vec v \mapsto \vec w$, where
  \[
   \DD_0 = \{\vec v \in \WW_p^{(1,2)}((0,\tau) \times \Omega; \R^N): \, \vec v|_{t=0} \equiv 0 \}
  \] 
 and $\vec w$ is the unique strong solution to the linear problem
  \begin{align*}
   \partial_t \vec w - \bb D \Delta \vec w
    &= \sum_a \left( k_a^f \left[ (\vec v + \vec c^0)^{\vec \alpha^a}  - (\vec v^0)^{\vec \alpha^a} \right] - k_a^b \left[ (\vec v + \vec c^0)^{\vec \beta^a} - (\vec v^0)^{\vec \beta^a} \right] \right) \vec \nu^a,
    &&(0, \tau) \times \Omega,
    \\
   \ip{\vec e^k}{\bb D \partial_{\vec n} \vec w|_\Sigma}
    &= 0,
    &&(0, \tau) \times \Sigma, \, k = 1, \ldots, n^\Sigma
    \\
   \bb C \vec \nu^{\Sigma,a} \cdot \vec w
    &= \bb C \vec \nu^{\Sigma,a} \cdot \vec v - (\vec v + \vec c^0)|_\Sigma^{\vec \nu^{\Sigma,a}} + (\vec c^0)|_\Sigma^{\vec \nu^{\Sigma,a}},
    &&(0, \tau) \times \Sigma, \, a = 1, \ldots, m^\Sigma
    \\
   \vec w|_{t=0}
    &= \vec 0,
    &&\Omega.
  \end{align*}
 This problem can now be handled in the way typical for semi-linear parabolic problems, employing the maximal regularity of the linearised problem and the regularity of the nonlinear maps on the right hand side, which allows for a fixed point argument via the contraction mapping principle. To establish the regularity properties which are needed, one first needs the following auxiliary result on algebraic properties of $\WW_p^{(1,2)}(J \times \Omega)$ for bounded intervals $J$ and bounded $\CC^2$-domains $\Omega$.

 \begin{lemma}
  Let $p \in (\frac{n+2}{2}, \infty)$ and $\Omega \subseteq \R^n$ be a bounded $\CC^2$-domain. Fix $T_0 > 0$. For $T \in (0,T_0]$ and
   \[
    \DD_0(T)
     = \{ u \in \WW_p^{(1,2)}((0,T) \times \Omega): \, u(0) = 0 \},
   \]
  the embedding constant constants for the continuous embeddings
   \[
    \DD_0(T) \hookrightarrow \CC(\overline{J} \times \overline{\Omega})
   \]
  can be chosen independently of $T \in (0, T_0)$, e.g.\ $C_p = 2^{1/p} C_p(T_0)$ where $C_p(T_0)$ is an embedding constant for $T = T_0$.
 \end{lemma}
 \begin{proof}
  Since for $u \in \DD_0(T)$ one has $u(0) = 0$, it follows that
   \[
    \tilde u(t,\cdot)
     := \begin{cases}
      u(t,\cdot),
      &t \in [0, T],
      \\
     u(T-t,\cdot),
      &t \in (T,2T],
      \\
     0,
      &t > 2T
     \end{cases}
   \]
  defines a function $\tilde u \in \WW_p^{(1,2)}(\R_+ \times \Omega)$ and for its restriction to $[0,T_0] \times \Omega$ it holds that
   \begin{align*}
    \norm{\tilde u}_{\WW_p^{(1,2)}((0,T_0) \times \Omega)}
     &\leq 2^{1/p} \norm{u}_{\WW_p^{(1,2)}((0,T) \times \Omega)},
     \\
    \norm{\tilde u}_\infty
     &= \norm{u}_\infty.
   \end{align*}
  From here the assertion follows easily.
 \end{proof}

% \begin{remark}
%  Since $\vec v|_{t=0}$ for $\vec v \in \DD_0$, the compatibility conditions at $t=0$ are satisfied.
% \end{remark}
%
% \begin{lemma}
%  Let $J \subseteq \R$ be a perfect, bounded interval and $\Omega \subseteq \R^n$ a bounded domain with $\CC^2$-boundary.
%  Then, for every $k \in \N$, the maps
%   \[
%    \BB_{pp}^{2-2/p}(\Omega)^k \rightarrow \BB_{pp}^{2-2/p}(\Omega),
%     \quad \vec v \mapsto \vec v^{\vec e} = \prod_{i=1}^k v_i 
%   \]
%  is continuous for $p > \frac{n+2}{2}$, and in that case it is Lipschitz continuous on every bounded subset of $\BB_{pp}^{2-2/p}(\Omega)$.
% \end{lemma}
% \begin{proof}
%  Since, for fixed $T > 0$, there is a continuous extension operator $E: \BB_{pp}^{2-2/p}(\Omega) \rightarrow \WW_p^{(1,2)}((0,T) \times \Omega)$, and $\WW_p^{(1,2)}((0,T) \times \Omega) \hookrightarrow \CC([0,T]; \BB_{pp}^{2-2/p}(\Omega))$, one may apply Lemma \ref{lem:product} to deduce the assertion.
% \end{proof}

 \begin{theorem}[Local-in-time existence of strong solutions]
  Let $p > \frac{n + 2}{2}$ and assume that $\Omega \subseteq \R^n$ is a bounded domain of class $\partial \Omega \in \CC^2$. Then the fast-sorption--fast-surface-chemistry limit problem \eqref{eqn:FSFSC-limit'} admits a unique local-in-time strong solution, if
   \[
    \vec c_0
     \in \II_p^+(\Omega)
      := \{ \vec c_0 \in \BB_{pp}^{2-2/p}(\Omega; (0,\infty)^3): \,
         \ip{\vec e^k}{\bb D \partial_{\vec n} \vec c^0|_\Sigma} = 0,
    \quad
   \vec v^0|_\Sigma^{\vec \nu^{\Sigma,a}} = \exp(- \vec \nu^{\Sigma,a} \cdot \vec g(T)) \}.
   \]
  More precisely, for every reference initial datum $\vec c_0^\ast \in \II_p^+(\Omega)$, there are $T > 0$, $\varepsilon > 0$ and $C > 0$ such that the following statements hold true:
   \begin{enumerate}
    \item
     For all $\vec c_0 \in \II_p^+(\Omega)$ with $\norm{\vec c_0 - \vec c_0^\ast}_{\II_p(\Omega)} < \varepsilon$, there is a unique strong solution $\vec c \in \WW_p^{(1,2)}(J \times \Omega; (0, \infty)^N)$ of \eqref{eqn:FSFSC-limit'} for $J = [0,T]$.
    \item
     For any two initial data $\vec c_0, \tilde{\vec c}_0 \in \II_p^+(\Omega)$ with $\norm{\vec c_0 - \vec c_0^\ast}_{\BB_{pp}^{2-2/p}}, \norm{\tilde{\vec c}_0 - \vec c_0^\ast}_{\BB_{pp}^{2-2/p}} < \varepsilon$ and corresponding strong solutions $\vec c, \tilde {\vec c} \in \WW_p^{(1,2)}(J \times \Omega; \R^3)$ one has
      \[
       \norm{\vec c - \tilde {\vec c}}_{\WW_p^{(1,2)}(J \times \Omega)}
        \leq C \norm{\vec c_0 - \tilde {\vec c}_0}_{\WW_p^{2-2/p}(\Omega)}.
      \]
    \item
     Any strong solution $\vec c \in \WW_p^{(1,2)}(J \times \Omega)$ can be extended in a unique way to a maximal strong solution $\vec c: [0, T_\mathrm{max}) \times \Omega \rightarrow (0,\infty)^N$ (with $T_\mathrm{max} \in (0, \infty]$) with $\vec c \in \WW_p^{(1,2)}((0,T) \times \Omega; (0,\infty)^N)$ for every $T \in (0, T_\mathrm{max})$.
   \end{enumerate}
 \end{theorem}
 
 \begin{proof} %new version
  Let $\varepsilon > 0$ and initial data
   \[
    \vec c^0
     \in \II_p^\varepsilon(\Omega)
      := \{ \vec c^0: \, c_i^0 \geq \varepsilon \, (i = 1, \ldots, N) \, \text{ on } \overline{\Omega} \}
   \]
  be given.
  Let $\rho_0, T_0 > 0$ be such that
   \[
    \norm{\vec v}_\infty \leq \frac{\varepsilon}{3}
    \quad
    \text{for all }
    \vec v \in \DD_{\rho,T} := \{ \vec v \in \WW^{(1,2)}_p((0,T) \times \Omega;\R^N): \, \vec v(0,\cdot) = 0, \, \norm{\vec v}_{\WW^{(1,p)}} \leq \rho \},
    \,
    \rho \in (0, \rho_0], \, T \in (0, T_0].
   \]
  Moreover, let $\mathcal{E}: \, \BB^{2-2/p}_{pp}(\Omega;\R^N) \rightarrow \WW^{(1,2)}_p((0,T_0) \times \Omega; \R^N)$ be a linear continuous extension operator and w.l.o.g.\ assume that $\mathcal{E} \vec v \geq \frac{\varepsilon}{3}$ on $(0,T_0) \times \Omega$ whenever $\vec v \geq \varepsilon$ on $\Omega$.
  Then, in particular, $\mathcal{E} \vec c^0 + \vec v \geq \frac{\varepsilon}{3}$ for all $\rho \in (0, \rho_0]$, $T \in (0, T_0]$ and $\vec v \in \DD_{\rho,T}$ on $(0,T) \times \Omega$.
  \newline
  We are looking for a solution $\vec c \in \WW^{(1,2)}_p((0,T) \times \Omega; \R^N)$ of
   \begin{align*}
    \partial_t \vec c - \DD \Delta \vec c
     &= \vec r(\vec c)
     &&\text{in } (0,T) \times \Omega,
     \\
    - \vec e^k \cdot \bb D \partial_{\vec n} \vec c
     &= \vec 0
     &&\text{on } (0,T) \times \Sigma, \, k = 1, \ldots, n^\Sigma,
     \\
    \vec c^{\vec \nu^{\Sigma,a}}
     &= \kappa_a (> 0)
     &&\text{on } (0,T) \times \Sigma, \, a = 1, \ldots, m^\Sigma,
     \\
    \vec c(0,\cdot)
     &= \vec c^0
     &&\text{in } \Omega.
   \end{align*}
  We set $\vec c^\ast := \mathcal{E} \vec c^0 \in \WW^{(1,2)}_p((0,T_0) \times \Omega; \R^N)$ and $\vec v := \vec c - \vec c^\ast$.
  Then, $\vec v$ should be solution to the following initial-boundary value-problem:
   \begin{align*}
    \partial_t \vec v -  \bb D \Delta \vec v
     &= \vec r(\vec c^\ast + \vec v) - (\partial_t - \bb D \Delta) \vec c^\ast
     &&\text{in } (0,T) \times \Omega,
     \\
    - \vec e^k \cdot \bb D \partial_{\vec n} \vec v
     &= \vec e^k \cdot \cdot \bb D \partial_{\vec n} \vec c^\ast
     &&\text{on } (0,T) \times \Sigma, \, k = 1, \ldots, n^\Sigma,
     \\
    (\vec c^\ast + \vec v)^{\vec \nu^{\Sigma,a}}
     &= \kappa_a
     &&\text{on } (0,T) \times \Sigma, \, a = 1, \ldots, m^\Sigma,
     \tag{$\ast$}
     \label{ast}
     \\
    \vec v(0,\cdot)
     &= \vec 0
     &&\text{in } \Omega.
   \end{align*}
  The nonlinear boundary conditions can be expressed (using that $(\vec c^0)^{\vec \nu^{\Sigma,a}} = \kappa_a$) as
   \begin{align*}
    \bb C_a^0 \vec \nu^{\Sigma,a} \cdot \vec v
     &:= \sum_{i=1}^N \nu_i^{\Sigma,a} \frac{1}{c_i^0} v_i
     \\
     &= \sum_{i=1}^N \nu_i^{\Sigma,a} \big( \frac{1}{c_i^0} - \frac{1}{c_i^\ast + v_i} \big) v_i
      + \sum_{i=1}^N \nu_i^{\Sigma,a} \frac{1}{c_i^\ast + v_i} v_i
      \\
     &= \sum_{i=1}^N \nu_i^{\Sigma,a} \big( \frac{1}{c_i^0} - \frac{1}{c_i^\ast + v_i} \big) v_i
      + \frac{1}{(\vec c^\ast + \vec v)^{\vec \nu^{\Sigma,a}}} \sum_{i=1}^N \nu_i^{\Sigma,a} (\vec c^\ast + \vec v)^{\vec \nu^{\Sigma,a} - \vec e_i} v_i
      \\
     &= \sum_{i=1}^N \nu_i^{\Sigma,a} \big( \frac{1}{c_i^0} - \frac{1}{c_i^\ast + v_i} \big) v_i
      + \frac{1}{(\vec c^\ast + \vec v)^{\nu^{\Sigma,a}}} \big( (\vec c^\ast + \vec v)^{\vec \nu^{\Sigma,a}} - (\vec c^0)^{\vec \nu^{\Sigma,a}} - \sum_{i=1}^N \nu_i^{\Sigma,a} (\vec c^\ast + \vec v)^{\vec \nu^{\Sigma,a} - \vec e_i} v_i \big)
      \\
     &=: \delta \bb C_a(\vec c^\ast, \vec v) \vec \nu^{\Sigma,a} \cdot \vec v
      + \frac{1}{(\vec c^\ast + \vec v)^{\nu^{\Sigma,a}}} \big( (\vec c^\ast + \vec v)^{\vec \nu^{\Sigma,a}} - (\vec c^0)^{\vec \nu^{\Sigma,a}} - \bb C_a(\vec c^\ast + \vec v) \vec \nu^{\Sigma,a} \cdot \vec v \big).
   \end{align*}
  Then, $\vec v \in \DD_{\rho,T}$ is a solution $\vec v \in \WW^{(1,2)}_p((0,T) \times \Omega; \R^N)$ of \eqref{ast} if and only if
   \begin{align*}
    \partial_t \vec v - \bb D \Delta \vec v
     &= \vec r(\vec c^\ast + \vec v) - (\partial_t - \bb D \Delta) \vec c^\ast
     &&\text{in } (0,T) \times \Omega,
     \\
    - \vec e^k \cdot \bb D \partial_{\vec n} \vec v
     &= \vec e^k \cdot \bb D \partial_{\vec n} \vec c^\ast
     &&\text{on } (0,T) \times \Sigma, \, k = 1, \ldots, n^\Sigma
     \\
    \label{astast}
    \tag{$\ast\ast$}
    \bb C_a^0 \vec \nu^{\Sigma,a} \vec v
     &= \delta \bb C_a(\vec c^\ast, \vec v) \vec \nu^{\Sigma,a} \cdot \vec v
      \\ &\qquad
      + \frac{1}{(\vec c^\ast + \vec v)^{\vec \nu^{\Sigma,a}}} \big( (\vec c^\ast + \vec v)^{\vec \nu^{\Sigma}} - (\vec c^0)^{\vec \nu^{\Sigma,a}} - \bb C_a(\vec c^\ast + \vec v) \vec \nu^{\Sigma,a} \cdot \vec v \big)
     &&\text{on } (0,T) \times \Sigma, \, a = 1, \ldots, m^\Sigma,
     \\
    \vec v(0,\cdot)
     &= \vec 0
     &&\text{in } \Omega.
   \end{align*}
  Therefore, $\vec v \in \DD_{\rho,T}$ is a solution to \eqref{ast}, if and only if $\vec v \in \DD_{\rho,T}$ is a fix point of the map $\Phi: \DD_{\rho,T} \rightarrow \WW^{(1,2)}_p((0,T) \times \Omega; \R^N)$ defined as follows:
  For $\vec v \in \DD_{\rho,T}$ let $\Phi(\vec v) := \vec w$ be the unique solution to the inhomogeneous initial-boundary value-problem
   \begin{align*}
    \partial_t \vec w - \bb D \Delta \vec w
     &= \vec r(\vec c^\ast + \vec v) - (\partial_t - \bb D \Delta) \vec c^\ast
     =: \vec f(\vec c^\ast, \vec v)
     &&\text{in } (0,T) \times \Omega,
     \\
    - \vec e^k \cdot \bb D \partial_{\vec n} \vec w
     &= \vec e^k \cdot \bb D \partial_{\vec n} \vec c^\ast
     =: g_k(\vec c^\ast)
     &&\text{on } (0,T) \times \Sigma, \, k = 1, \ldots, n^\Sigma
     \\
    \bb C_a^0 \vec \nu^{\Sigma,a} \vec w
     &= \delta \bb C_a(\vec c^\ast, \vec v) \vec \nu^{\Sigma,a} \cdot \vec v
      \\ &\qquad
      + \frac{1}{(\vec c^\ast + \vec v)^{\vec \nu^{\Sigma,a}}} \big( (\vec c^\ast + \vec v)^{\vec \nu^{\Sigma}} - (\vec c^0)^{\vec \nu^{\Sigma,a}} - \bb C_a(\vec c^\ast + \vec v) \vec \nu^{\Sigma,a} \cdot \vec v \big)
      \\
      &=: h_a(\vec c^\ast, \vec v)
     &&\text{on } (0,T) \times \Sigma, \, a = 1, \ldots, m^\Sigma,
     \\
    \vec w(0,\cdot)
     &= \vec 0
     &&\text{in } \Omega.
   \end{align*}
  By $\LL_p$-maximal regularity of the linearised problem, the solution $\vec w \in \WW^{(1,2)}_p((0,T) \times \Omega; \R^N)$ exists and is uniquely determined, for every $\vec v \in \DD_{\rho,T}$.
  I.e.\ $\Phi$ is well-defined.
  Since the initial datum $\vec w(0,\cdot) = 0$ is zero for all the functions constructed in this way, the constant $C_T = C_{T_0} > 0$ in the maximal regularity estimate
   \[
    \norm{\Phi(v)}_{\WW^{(1,2)}_p((0,T) \times \Omega)}
     \leq C_{T_0} \big( \norm{\vec f(\vec c^\ast, \vec v)}_{\LL^p((0,T) \times \Omega)} + \norm{\vec g(\vec c^\ast)}_{\WW^{(1,2) \cdot (\nicefrac{1}{2} - \nicefrac{1}{2p})}_p((0,T) \times \Sigma)} + \norm{\vec h(\vec c^\ast, \vec v)}_{\WW^{(1,2) \cdot (1 - \nicefrac{1}{2p})}_p((0,T) \times \Sigma)}  \big)
   \]
  for $\vec v \in \DD_{\rho,T}$ can be chosen independently of $T \in (0, T_0]$ (using, e.g.\ a mirroring argument).
  \newline
  We will demonstrate that $\rho \in (0, \rho_0]$ and $T \in (0, T_0]$ can be chosen such that $\Phi$ is a contractive self-mapping on $\DD_{\rho,T}$, and hence attains a unique fixed point by the contraction mapping principle.
  To this end, we first show that $\Phi$ is for $\rho \in (0, \rho_0]$ and suitable $T \in (0, T_1]$ a mapping from $D_{\rho,T}$ into $D_{\rho,T}$.
  We will employ the notation $a \lesssim_{T_0} b$, if there is a constant $C = C(T_0)$ depending on $T_0$ such that $a \leq C(T_0) b$.
    \newline
     The term $\norm{\vec r(\vec c^\ast + \vec v)}_{\LL_p}$ can be handled analogously to reaction-diffusion-systems with linear boundary conditions as
      \[
       \norm{\vec r(\vec c^\ast + \vec v)}_{\LL_p}
        \leq C(\rho,T) \norm{\vec v}_{\WW_p^{(1,2)}},
        \quad
        \vec v \in \DD_{\rho,T},
      \]
     where $C(\rho,T) \rightarrow 0$ as $T \rightarrow 0+$ for $\rho \in (0, \rho_0]$.
    Next, 
      \[
       \norm{(\partial_t - \bb D \Delta) \vec c^\ast}_{\LL_p}
        \lesssim_{T_0} \norm{\vec c^\ast}_{\WW^{(1,2)}_p},
      \]
     which tends to zero as $T \rightarrow 0+$.
    Moreover, 
     \begin{align*}
      \norm{\vec e^k \cdot \bb D \partial_{\vec n} \vec c^\ast}_{\WW^{(1,2) \cdot (\nicefrac{1}{2} - \nicefrac{1}{2p})}_p}
       &\lesssim_{T_0} \norm{c^\ast}_{\WW^{(1,2)}_p},
       \\
      \norm{\delta \bb C_a(\vec c^\ast, \vec v) \vec \nu^{\Sigma,a} \cdot \vec v}_{\WW^{(1,2) \cdot (1 - \nicefrac{1}{2p})}}
       &\lesssim_{T_0} \norm{\delta \bb C_a(\vec c^\ast, \vec v)}_{\WW^{(1,2) \cdot (1 - \nicefrac{1}{2p})}_p} \norm{\vec v}_{\WW^{(1,2)}_p},
     \end{align*}
     where -- using the definition $\delta \bb C_a(\vec c^\ast, \vec v) = \diag (\nicefrac{1}{c_i^\ast + v_i})_i - \diag (\nicefrac{1}{c_i^0})_i \in \WW^{(1,2) \cdot (1 - \nicefrac{1}{2p})}( (0,T) \times \Sigma; \R^{N \times N})$ -- the former factor can be estimated as
      \[
       \norm{\delta \bb C_a(\vec c^\ast, \vec v)}_{\WW^{(1,2) \cdot (1 - \nicefrac{1}{2p})}_p}
        \lesssim_{\rho_0, T_0} \norm{(\vec c^\ast + \vec v) - \vec c^0}_{\WW^{(1,2) \cdot (1 - \nicefrac{1}{2p})}_p}
        \xrightarrow{T \rightarrow 0+} 0
      \]
     for $\rho \in (0, \rho_0]$, as $(\vec c^\ast + v)(0, \cdot) = \vec c^0$.
    \newline
     Splitting
      \[
       \frac{1}{(\vec c^\ast + \vec v)^{\vec \nu^{\Sigma,a}}} \big( (\vec c^\ast + \vec v)^{\vec \nu^{\Sigma}} - (\vec c^0)^{\vec \nu^{\Sigma,a}} - \bb C_a(\vec c^\ast + \vec v) \vec \nu^{\Sigma,a} \cdot \vec v \big)
        = \frac{1}{(\vec c^\ast + \vec v)^{\vec \nu^{\Sigma,a}}} \big( (\vec c^\ast + \vec v)^{\vec \nu^{\Sigma}} - (\vec c^0)^{\vec \nu^{\Sigma,a}} - \diag (\nicefrac{1}{c_i^\ast + v_i})_i \vec \nu^{\Sigma,a} \cdot \vec v \big)
      \]
     into a term
      \begin{align*}
       &\frac{1}{(\vec c^\ast + \vec v)^{\nu^{\Sigma,a}}} \big( (\vec c^\ast + \vec v)^{\vec \nu^{\Sigma,a}} - (\vec c^\ast)^{\vec \nu^{\Sigma,a}} \big)
        \\
        &= 1 - \frac{(\vec c^\ast)^{\vec \nu^{\Sigma,a}}}{(\vec c^\ast + \vec v)^{\vec \nu^{\Sigma,a}}}
        = 1 - \prod_{i=1}^N \frac{(c_i^\ast)^{\nu_i^{\Sigma,a}}}{(c_i^\ast + v_i)^{\nu_i^{\Sigma,a}}}
        \\
        &= 1 - \prod_{i: \alpha_i^{\Sigma,a} > 0} \frac{(c_i^\ast + v_i)^{\alpha_i^{\Sigma,a}}}{(c_i^\ast)^{\alpha_i^{\Sigma,a}}} \cdot \prod_{i: \beta_i^{\Sigma,a} > 0} \frac{(c_i^\ast)^{\beta_i^{\Sigma,a}}}{(c_i^\ast + v_i)^{\beta_i^{\Sigma,a}}}
        \\
        &= \sum_{k: \alpha_k^{\Sigma,a} > 0} \big( \prod_{i < k: \alpha_i^{\Sigma,a} > 0} \frac{(c_i^\ast + v_i)^{\alpha_i^{\Sigma,a}}}{(c_i^\ast)^{\alpha_i^{\Sigma,a}}} \big) \cdot \big( 1 - \frac{(c_k^\ast + v_k^\ast)^{\alpha_k^{\Sigma,a}}}{(c_k^\ast)^{\alpha_k^{\Sigma,a}}} \big)
         \\ &\quad
         + \sum_{k: \beta_k^{\Sigma,a} > 0} \big( \prod_{i: \alpha_i^{\Sigma,a} > 0} \frac{(c_i^\ast + v_i)^{\alpha_i^{\Sigma,a}}}{(c_i)^{\alpha_i^{\Sigma,a}}} \prod_{j < k: \beta_j^{\Sigma,a} > 0} \frac{(c_j^\ast)^{\beta_j^{\Sigma,a}}}{(c_j^\ast + v_j)^{\beta_j^{\Sigma,a}}} \big) \cdot \big( 1 - \frac{(c_k^\ast)^{\beta_k^{\Sigma,a}}}{(c_k^\ast + v_k)^{\beta_k^{\Sigma,a}}} \big),
      \end{align*}
     which can be estimated as
      \begin{align*}
       &\norm{\frac{1}{(\vec c^\ast + \vec v)^{\nu^{\Sigma,a}}} \big( (\vec c^\ast + \vec v)^{\vec \nu^{\Sigma,a}} - (\vec c^\ast)^{\vec \nu^{\Sigma,a}} \big)}_{\WW^{(1,2) \cdot (1 - \nicefrac{1}{2p})}_p((0,T) \times \Sigma)}
        \\
        &\lesssim_{T_0} \norm{\frac{1}{(\vec c^\ast + \vec v)^{\nu^{\Sigma,a}}} \big( (\vec c^\ast + \vec v)^{\vec \nu^{\Sigma,a}} - (\vec c^\ast)^{\vec \nu^{\Sigma,a}} \big)}_{\WW^{(1,2)}_p((0,T) \times \Omega)}
        \leq C_{\rho,T} \norm{v}_{\WW^{(1,2)}_p}
      \end{align*}
     for $C_{\rho,T} > 0$ such that $C_{\rho,T} \rightarrow 0+$ as $T \rightarrow 0+$, $\rho \in (0, \rho_0]$.
     Secondly, the term
      \[
       \bb C_a(\vec c^\ast + \vec v) \vec \nu^{\Sigma,a} \cdot \vec v
        = \sum_{i=1}^N \nu_i^{\Sigma,a} \frac{v_i}{c_i^\ast + v_i},
      \]
     can be handled using the estimate
      \begin{align*}
       \norm{\nicefrac{v_i}{c_i^\ast + v_i}}_{\WW^{(1,2) \cdot (1 - \nicefrac{1}{2p})}_p((0,T) \times \Sigma)}
        &\lesssim_{\rho_0, T_0} \norm{\nicefrac{v_i}{c_i^\ast + v_i}}_{\WW^{(1,2)}_p}
        \leq C_{\rho,T} \norm{v_i}_{\WW^{(1,2)}_p},
      \end{align*}
     also with a constant $C_{\rho,T}$ tending to zero as $T \rightarrow 0+$ and $\rho \in (0, \rho_0]$.
    \newline
  These considerations give that $T_1 \in (0,T_0]$ may be chosen such that for all $\rho \in (0,\rho_0]$ and $T \in (0, T_1]$, the map $\Phi$ maps $\DD_{\rho,T}$ into itself.
  Thus, it remains to show that $\Phi$ is contractive for suitable choice of $\rho, T$ from this range.
  \newline
  Let, therefore, $\rho \in (0, \rho_0]$ and $T \in (0, T_1]$ as well as $\vec u, \vec v \in \DD_{\rho,T}$ be given.
  Then $\vec w := \Phi(\vec v) - \Phi(\vec u) \in \WW^{(1,2)}_p((0,T) \times \Omega; \R^N)$ is the solution to the initial-boundary value-problem
   \begin{align*}
    \partial_t \vec w - \bb D \Delta \vec w
     &= \vec r(\vec c^\ast + \vec v) - \vec r(\vec c^\ast + \vec u)
     &&\text{in } (0,T) \times \Omega,
     \\
    - \vec e^k \cdot \bb D \partial_{\vec n} \vec w
     &= \vec 0
     &&\text{on } (0,T) \times \Sigma, \, k = 1, \ldots, n^\Sigma,
     \\
    \bb C_a^0 \vec \nu^{\Sigma,a} \cdot \vec w
     &= \delta \bb C_a(\vec c^\ast, \vec v) \vec \nu^{\Sigma,a} \cdot \vec v
      - \delta \bb C_a(\vec c^\ast, \vec u) \vec \nu^{\Sigma,a} \cdot \vec u
      \\ &\quad
      + \big( (\vec c^\ast + \vec v)^{- \vec \nu^{\Sigma,a}} - (\vec c^\ast + \vec u)^{- \vec \nu^{\Sigma,a}} \big) (\vec c^0)^{\vec \nu^{\Sigma,a}}
      \\
      &=: \hat h_a(\vec c, \vec v, \vec u)
      &&\text{on } (0,T) \times \Sigma, \, a = 1, \ldots, m^\Sigma,
      \\
     \vec w(0,\cdot)
      &= \vec 0
      &&\text{on } \Omega.
   \end{align*}
  With $\LL_p$-maximal regularity and $\vec w(0,\cdot) = \vec 0$ it holds again
   \[
    \norm{\vec w}_{\WW^{(1,2)}_p}
     \lesssim_{T_0} \norm{\vec r(\vec c^\ast + \vec v) - \vec r(\vec c^\ast + \vec u)}_{\LL_p} + \norm{\vec {\hat h}(\vec c^\ast, \vec v, \vec u)}_{\WW^{(1,2) \cdot (1 - \nicefrac{1}{2p}}_p},
     \quad
     T \in (0, T_1].
   \]
  Similar to the self-mapping property, the bulk chemistry term does not pose much problems and can be handled in an analogous way as above.
  For the term $\vec {\hat h}(\vec c^\ast, \vec v, \vec u)$, we derive the following estimates:
   \newline
     It holds
     \[
      \delta \bb C_a(\vec c^\ast, \vec v) \vec \nu^{\Sigma,a} \cdot \vec v
       - \delta \bb C_a(\vec c^\ast, \vec u) \vec \nu^{\Sigma,a} \cdot \vec u
       = [\delta \bb C_a(\vec c^\ast, \vec v) - \delta \bb C_a(\vec c^\ast, \vec u)] \vec \nu^{\Sigma,a} \cdot \vec v
        + \delta \bb C_a(\vec c, \vec u) \vec \nu^{\Sigma,a} \cdot (\vec v - \vec u),
     \]
    where the latter term is
     \begin{align*}
      \delta \bb C_a(\vec c^\ast, \vec u) \vec \nu^{\Sigma,a} \cdot (\vec v - \vec u)
       &= \big( \diag(\nicefrac{1}{c_i^\ast + v_i})_i - \diag(\nicefrac{1}{c_i^\ast}) \big) \vec \nu^{\Sigma,a} \cdot (\vec v - \vec u)
       \\
       &= \big( \diag(\nicefrac{1}{c_i^\ast(c_i^\ast + \vec u)})_i \diag(u_i)_i\vec \nu^{\Sigma,a} \big) \cdot (\vec v - \vec u)
     \end{align*}
    and may be estimated as
     \begin{align*}
      \norm{\delta \bb C_a(\vec c^\ast, \vec u) \vec \nu^{\Sigma,a} \cdot (\vec v - \vec u)}_{\WW^{(1,2) \cdot (1 - \nicefrac{1}{2p})}_p}
       &\leq C_{\rho,T} \norm{\vec v - \vec u}_{\WW^{(1,2)}_p},
     \end{align*}
    where $C_{\rho,T} \rightarrow 0$ as $T \rightarrow 0+$, $\rho \in (0, \rho_0]$. The first term may be estimated writing
     \begin{align*}
      [\delta \bb C_a(\vec c^\ast, \vec v) - \delta \bb C_a(\vec c^\ast, \vec u)] \vec \nu^{\Sigma,a} \cdot \vec v
       &= \big( \diag (\nicefrac{1}{(c_i^\ast + v_i)(c_i^\ast + u_i)})_i \diag(\vec v_i - \vec u_i)_i \vec \nu^{\Sigma,a} \big) \cdot \vec v.
     \end{align*}
    Also, the term
     \begin{align*}
      &\frac{1}{(\vec c^\ast + \vec v)^{\vec \nu^{\Sigma,a}}} - \frac{1}{(\vec c^\ast + \vec u)^{\vec \nu^{\Sigma,a}}}
       \\
       &= \frac{1}{(\vec c^\ast + \vec v)^{\vec \beta^{\Sigma,a}}} (\vec c^\ast + \vec u)^{\vec \beta^{\Sigma,a}} \big( (\vec c^\ast + \vec v)^{\vec \beta^{\Sigma,a}} - (\vec c^\ast + \vec u)^{\vec \beta^{\Sigma,a}} \big) \big( (\vec c^\ast + \vec u)^{\vec \alpha^{\Sigma,a}} - (\vec c^\ast - \vec v)^{\vec \alpha^{\Sigma,a}} \big)
     \end{align*}
    can be estimated as
     \[
      \norm{\frac{1}{(\vec c^\ast + \vec v)^{\vec \nu^{\Sigma,a}}} - \frac{1}{(\vec c^\ast + \vec u)^{\vec \nu^{\Sigma,a}}}}_{\WW^{(1,2) \cdot (1 - \nicefrac{1}{2p}}_p}
       \leq C_{\rho,T} \norm{\vec v - \vec u}_{\WW^{(1,2)}_p}
     \]
    with $C_{\rho,T} \rightarrow 0$ as $T \rightarrow 0+$, $\rho \in (0, \rho_0]$.
   \newline
   The terms
    \begin{align*}
     &\frac{1}{(\vec c^\ast + \vec v)^{\vec \nu^{\Sigma,a}}} \bb C_a(\vec c^\ast + \vec v) \vec \nu^{\Sigma,a} \cdot \vec v - \frac{1}{(\vec c^\ast + \vec u)^{\vec \nu^{\Sigma,a}}} \bb C_a(\vec c^\ast + \vec u) \vec \nu^{\Sigma,a} \cdot \vec u
     \\
     &= \big( \frac{1}{(\vec c^\ast + \vec v)^{\vec \nu^{\Sigma,a}}} - \frac{1}{(\vec c^\ast + \vec u)^{\vec \nu^{\Sigma,a}}} \big)
      \\ &\quad
      + \frac{1}{(\vec c^\ast + \vec u)^{\vec \nu^{\Sigma,a}}} \big( \bb C_a(\vec c^\ast + \vec v) - \bb C_a(\vec c^\ast + \vec u) \big) \vec \nu^{\Sigma,a} \cdot \vec u
      \\ &\quad
      + \frac{1}{(\vec c^\ast + \vec u)^{\vec \nu^{\Sigma,a}}} \bb C_a(\vec c^\ast + \vec u) \vec \nu^{\Sigma,a} \cdot (\vec v - \vec u)
    \end{align*}
   can be handled analogously: Each of the summands can in $\WW^{(1,2) \cdot (1 - \nicefrac{1}{2p})}_p$-norm be estimated by $C_{\rho,T} \norm{\vec v - \vec u}_{\WW^{(1,2)}_p}$ for some $C_{\rho,T}$ tending to zero as $T \rightarrow 0+$.
   \newline
   As a result, choosing $\rho \in (0, \rho_0]$ arbitrary and $T \in (0, T_1]$ sufficiently small, say $T \in (0, T_2]$ for some $T_2 \in (0,T_1]$, $\Phi: \DD_{\rho,T} \rightarrow \DD_{\rho,T}$ is a strict contraction, and, therefore, by the contraction mapping principle has a unique fixed point $\vec v^\ast = \Phi(\vec v^\ast) \in \DD_{\rho,T}$.
   Then, $\vec c := \vec c^\ast + \vec v^\ast$ is the unique solution of the fast sorption and fast surface chemistry limit reaction diffusion system.
   \newline
   Moreover, for $\vec c^{\ast,0} \in \II_p^\varepsilon(\Omega)$, $\eta > 0$ sufficiently small and initial data $\vec c^0 \in \BB_\eta(\vec c^{\ast,0}) \subseteq \II_p^\varepsilon(\Omega)$ close to $\vec c^{\ast,0}$, there is a common choice of parameters $\rho_0$ and $T_2$ to make the respective maps $\Phi = \Phi^{\vec c^0}$ strictly contractive self-mappings for any $\rho \in (0,\rho_0]$ and $T \in (0, T_2]$, so that for all these initial data the solution exists and is unique at least on the time interval $(0, T_2)$.
   Also it can be seen that as the maps $\Phi^{\vec c^0}$ continuously depend on the initial datum, so do the fixed points, hence the solutions to the fast sorption and fast surface chemistry reaction-diffusion-limit system.
 \end{proof}

From the proof one can extract blow-up criteria for solutions which are not global in time.

 \begin{corollary}[Blow-up criterion]
  Either $T_\mathrm{max} = \infty$ (global existence), or $T_\mathrm{max} < \infty$ and $\norm{\vec c}_{\BB_{pp}^{2-2/p}(\Omega;\R^N)} \rightarrow \infty$ (blow-up) or $\min_{\vec z \in \overline{\Omega}} c_i(t,\vec z) \rightarrow 0$ (degeneration) for some $i \in \{1, \ldots, N\}$ as $t \rightarrow T_\mathrm{max}$.
 \end{corollary}
 
 \begin{remark}
  The inclusion of the case $\min c_i(t,\vec z) \rightarrow 0$ for some $i$ seems to be a bit out of place here, but is due to the chosen linearisation around the reference function. To have enough regularity for $\bb C = \diag (\vec c|_\Sigma)^{-1}$ one needs uniform positivity of the solution candidate $\vec c$, thus on the initial datum $\vec c^0$. Therefore, this approach breaks down as $\min c_i(t,\cdot) \rightarrow 0$.
 \end{remark}

 \subsubsection{A-priori bounds on the strong solution of the fast sorption--surface-chemistry--transmission model}

 In the previous subsection, it has been noticed that a bound on the phase space norm $\norm{\cdot}_{\BB_{pp}^{2-2/p}}$ is enough for establishing global existence of a strong solution. To derive such a bound, however, is a delicate matter, and it is not clear whether global existence holds true in all cases.
 On the other hand, for some weaker norms at least a-priori bounds can be established \emph{for free}.
 The derivation of these a-priori bounds is based on the parabolic maximum principle and entropy considerations, highlighting the fruitful interplay between mathematics and physics, and will be presented in this subsection.
 
 \begin{theorem}[A-priori bounds on the strong solution]
  Let $\vec c^0 \in \II_p^+(\Omega) \cap \CC^2(\overline{\Omega}; \R^N)$ and $\vec c \in \CC^{(1,2)}([0, T_\mathrm{max}) \times \overline{\Omega};\R_+^N)$ be a maximal classical solution to the fast-surface-chemistry--fast-sorption--fast-transmission limit problem
   \begin{align*}
    \partial_t \vec c - \bb D \Delta \vec c
     &= \vec r(\vec c),
     &&t \geq 0, \, \vec z \in \Omega
     \\
    \vec r^\Sigma(\vec c^\Sigma)
     &= \vec 0,
     &&t \geq 0, \vec z \in \Sigma
     \\
    \vec s^\Sigma(\vec c, \vec c^\Sigma)
     &= 0,
     &&t \geq 0, \vec z \in \Sigma
     \\
    - \ip{\vec e^k}{\bb D \partial_{\vec n} \vec c}
     &=  \vec 0,
     &&k = 1, \ldots, n^\Sigma, \, t \geq 0, \vec z \in \Sigma
     \\
    \vec c(0,\cdot)
     &= \vec c^0,
     &&\vec z \in \overline{\Omega}.    
   \end{align*}
  Further, assume that there is a conserved quantity with strictly positive entries, i.e.\ there is
   \[
    \vec e \in (0, \infty)^N \cap \{\vec \nu^a: a = 1, \ldots, m\}^\bot \cap \{\vec \nu^{\Sigma,a}: a = 1, \ldots, m^\Sigma\}^\bot.
   \]
  Then, for every $T_0 \in (0, T_\mathrm{max}] \cap \R$ there is $C = C(T_0) > 0$, also depending on the initial data $\vec c^0$, such that the following a-priori bounds hold true:
   \begin{enumerate}
    \item
     $\LL_\infty^t \LL_1^{\vec z}$--a-priori estimate:
      \[
       \sup_{t \in [0, T_0)} \norm{\vec c(t,\cdot)}_{\LL_1(\Omega;\R^N)}
        \leq C \norm{\vec c^0}_{\LL_1(\Omega; \R^N)},
      \]
     where the constant can actually be chosen independent of $\vec c^0$ and $T_0$, but only depends on the ratio between the smallest and largest entry of $\vec e \in (0, \infty)^N$;
    \item
     $\LL_1^t \LL_\infty^{\vec z}$--a-priori estimate:
      \[
       \sup_{\vec z \in \overline{\Omega}} \norm{c(\cdot,\vec z)}_{\LL_1([0,T_0);\R^N)}
        \leq C;
      \]
    \item
     $\LL_2^t \LL_2^{\vec z}$--a-priori estimate:
      \[
       \norm{\vec c}_{\LL_2([0,T_0) \times \Omega; \R^N)}
        \leq C;
      \]
    \item
     Moreover, the following \emph{entropy identity} holds true:
      \begin{align*}
       &\int_\Omega c_i(t,\vec z) (\mu_i^0 + \ln \vec c_i(t,\vec z) - 1) \dd \vec z
        \\ &\quad
        + \int_0^t \int_\Omega \sum_{i=1}^N d_i \frac{\abs{\nabla c_i(s,\vec z)}^2}{c_i(s,\vec z)} \dd \vec z \dd s
        + \sum_{a=1}^m \int_0^t \int_\Omega \big( \sum_{i=1}^N \ln(c_i) \nu_i^a \big) \left(\exp \big( \sum_{i=1}^N \ln(c_i) \nu_i^a \big) - 1 \right) \dd \vec z \dd s
        \\ &\quad
       = \int_\Omega c_i^0(\vec z) (\mu_i^0 \ln c_i^0(\vec z) - 1) \dd \vec z,
        \quad
        t \in [0,T_0).
      \end{align*}
   \end{enumerate}
 \end{theorem}
 \begin{proof}
  \emph{$\LL_\infty^t \LL_1^{\vec z}$--a-priori estimate:}
  Since $\vec e \in (0, \infty)^N$ is a conserved quantity for both the bulk and surface chemistry, $\ip{\vec r(\vec c)}{\vec e} = \ip{\vec r^\Sigma(\vec c)}{\vec e} = 0$ for all values of $\vec c$.
  Thus, using regularity properties of parameter-dependent integrals and the divergence theorem, for every $t \in [0,T_0)$ it holds that 
   \begin{align*}
    \frac{\dd}{\dd t} \int_\Omega \ip{\vec c(s,\vec z)}{\vec e} \dd \vec z
     &= \int_\Omega \ip{\partial_t \vec c(t,\vec z)}{\vec e} \dd \vec z
     = \int_\Omega \ip{\bb D \Delta \vec c(t,\vec z)}{\vec e} \dd \vec z
      + \int_\Omega \ip{\vec r(\vec c(t,\vec z))}{\vec e} \dd \vec z
      \\
     &= \int_\Omega \ip{\bb D \Delta \vec c(t,\vec z)}{\vec e} \dd \vec z
     = \int_{\partial \Omega} \ip{\bb D \partial_{\vec n} \vec c(t,\vec z)}{\vec e} \dd \sigma(\vec z)
     = 0.
   \end{align*} 
  As a result,
   \[
    \int_\Omega \ip{\vec c(t,\vec z)}{\vec e} \dd \vec z
     = \int_\Omega \ip{\vec c^0(\vec z)}{\vec e} \dd \vec z,
     \quad
     t \in [0, T_0)
   \]
  and, since $\vec e \in (0, \infty)^N$, the map $c \mapsto \int_\Omega \abs{\ip{\vec c}{\vec e}} \dd \vec z$ defines a norm which is equivalent to the standard $\LL_1$-norm on the Lebesgue space $\LL_1(\Omega;\R^N)$.
  More precisely,
   \[
    \norm{\vec c(t,\cdot)}_{\LL_1(\Omega;\R^N)}
     \leq \frac{1}{\min_i e_i} \int_\Omega \ip{\vec c^0(\vec z)}{\vec e} \dd \vec z
     \leq \frac{\max_i e_i}{\min_i e_i} \norm{\vec c^0}_{\LL_1(\Omega;\R^N)},
   \]
  establishing the first a-priori estimate for
   \[
    C
     = \frac{\max_i e_i}{\min_i e_i}
   \]
  independent of $T_0 > 0$ and the initial datum $\vec c^0$.
  \newline
  \emph{$\LL_1^t \LL_\infty^{\vec z}$--a-priori estimate:}
  To derive the $\LL_1 \LL_\infty$-a-priori bound, let us consider the function $w: [0, T_0) \times \overline{\Omega} \rightarrow [0, \infty)$ defined by
   \[
    w(t,\vec z)
     = \int_0^t \ip{\bb D \vec c(s,\vec z)}{\vec e} \dd s,
     \quad
     t \in [0,T_0), \, \vec z \in \overline{\Omega}.
   \]
  As a parameter integral of a $\CC^{(1,2)}$-function, $w$ has the regularity $w \in \CC^2([0,T_0) \times \overline{\Omega})$ and using elementary results on parameter-dependent integrals, the evolution equation \eqref{eqn:FSFSC-limit} and the assumption that $\vec e$ is a conserved quantity, we establish the estimate
   \begin{align*}
    \partial_t w(t,\vec z)
     &= \ip{\bb D \vec c(t,\vec z)}{\vec e}
     \leq d_\mathrm{max} \ip{\vec c(t,\vec z)}{\vec e}
     \\
     &= d_\mathrm{max} \left( \int_0^t \ip{\partial_t \vec c(s,\vec z)}{\vec e} \dd s + \ip{\vec c^0(\vec z)}{\vec e} \right)
     \\
     &= d_\mathrm{max} \left( \int_0^t \ip{\bb D \Delta \vec c(s,\vec z)}{\vec e} + \ip{\vec r(\vec c(s,\vec z))}{\vec e} \dd s + \ip{\vec c^0(\vec z)}{\vec e} \right)
     \\
     &= d_\mathrm{max} \left( \int_0^t \ip{\bb D \Delta \vec c(s,\vec z)}{\vec e} \dd s + \ip{\vec c^0(\vec z)}{\vec e} \right)
     \\
     &= d_\mathrm{max} \left( \Delta w(t,\vec z) + d_\mathrm{max} \ip{\vec c^0(\vec z)}{\vec e} \right),
     &&t \in [0,T_0), \, \vec z \in \Omega
     \\
    \partial_{\vec n} w(t,\vec z)
     &= \int_0^t \ip{\bb D \partial_{\vec n} \vec c(s,\vec z)}{\vec e} \dd s
     = 0,
     &&t \in [0, T_0), \, \vec z \in \partial \Omega
     \\
    w(0,\vec z)
     &= 0,
     &&\vec z \in \overline{\Omega}.
   \end{align*}
  Therefore, $w \geq 0$ satisfies the system of differential inequalities
   \begin{align*}
    \partial_t w - d_\mathrm{max} \Delta w
     &\leq d_\mathrm{max} \ip{\vec c^0}{\vec e},
     &&t \in [0, T_0), \, \vec z \in \Omega
     \\
    \partial_{\vec n} w
     &= 0,
     &&t \in [0, T_0), \, \vec z \in \partial \Omega
     \\
    w(0,\vec z)
     &= 0,
     &&\vec z \in \Omega.
   \end{align*}
  From the parabolic maximum principle for differential inequalities, it then follows that there is $C > 0$ (depending on $T_0$ and $\vec c^0$) such that
   \[
    0 \leq w(t,\vec z) \leq C,
     \quad
     t \in [0, T_0), \, \vec z \in \Omega
   \]
  and, consequently, one finds that
   \[
    \norm{\vec c(\cdot,\vec z)}_{\LL_1([0,T_0);\R^N)} \leq C,
     \quad
     \vec z \in \Omega.
   \]
  \emph{$\LL_2^t \LL_2^{\vec z}$--a-priori estimate:} 
  For the $\LL_2$-estimate, let us fix $T \in (0,T_0)$.
  Employing integration by parts, Fubini's theorem, the no-flux boundary conditions on the conserved part and the fundamental theorem of calculus, we find for the integral
   \begin{align*}
    &\int_0^T \int_\Omega (\ip{\bb D \vec c}{\vec e}) (\ip{\vec c}{\vec e}) \dd \vec z \dd t
     \\
%     &= \int_0^T (\ip{\bb D \vec c(t,\vec z)}{\vec e}) \left( \int_0^t \ip{\partial_t c(s,\vec z)}{\vec e} \dd s + \ip{\vec c^0(\vec z)}{\vec e} \right) \dd \vec z \dd t
     \\
%     &= \int_0^T \int_\Omega \ip{\bb D \vec c(t,\vec z)}{\vec e} \left( \int_0^t \ip{\bb D \Delta \vec c(s,\vec z)}{\vec e} + \ip{\vec r(\vec c(s,\vec z))}{\vec e} \dd s + \ip{\vec c^0(\vec z)}{\vec e} \right) \dd \vec z \dd t
%     \\
     &= \int_0^T \int_\Omega (\ip{\bb D \vec c(t,\vec z)}{\vec e}) \left( \Delta \int_0^t \ip{\bb D \vec c(s,\vec z)}{\vec e} \dd s + (\ip{\bb D \vec c(t,\vec z)}{\vec e}) (\ip{\vec c^0(\vec z)}{\vec e}) \right) \dd \vec z \dd t
     \\
     &= \int_0^T \int_\Omega (\ip{\bb D \vec c(t,\vec z)}{\vec e}) \Delta \int_0^t \ip{\bb D \vec c(s,\vec z)}{\vec e} \dd s \dd \vec z \dd t
      \\ &\quad
      + \int_0^T \int_\Omega (\ip{\bb D \vec c(t,\vec z)}{\vec e}) (\ip{\vec c^0(\vec z)}{\vec e}) \dd \vec z \dd t
      \\
     &= - \int_\Omega \int_0^T \nabla (\ip{\bb D \vec c(t,\vec z)}{\vec e}) \cdot \nabla \int_0^t \ip{\bb D \vec c(s,\vec z)}{\vec e} \dd s \dd t \dd \vec z
      \\ &\quad
      + \int_0^T \int_\Sigma (\ip{\bb D \vec c(t, \vec z)}{\vec e}) \int_0^t \ip{\bb D \partial_{\vec n} \vec c(s, \vec z)}{\vec e} \dd s \dd \sigma(\vec z) \dd t
      \\ &\quad
      + \int_0^T \int_\Omega (\ip{\bb D \vec c(t,\vec z)}{\vec e}) (\ip{\vec c^0(\vec z)}{\vec e}) \dd \vec z \dd t
      \\
     &= - \frac{1}{2} \int_\Omega \abs{\int_0^T \nabla \ip{\bb D \vec c(t,\vec z)}{\vec e} \dd t}^2 \dd \vec z
      + \int_0^T \int_\Omega (\ip{\bb D \vec c(t,\vec z)}{\vec e}) (\ip{\vec c^0(\vec z)}{\vec e}) \dd \vec z \dd t
      \\
     &\leq \int_0^T \int_\Omega (\ip{\bb D \vec c(t,\vec z)}{\vec e}) (\ip{\vec c^0(\vec z)}{\vec e}) \dd \vec z \dd t
      \\
     &\leq \frac{d_\mathrm{max}}{d_\mathrm{min}} T \norm{ \ip{\vec c^0}{\vec e}}_{\LL_\infty(\Omega)} \norm{ \ip{\bb D \vec c^0}{\vec e}}_{\LL_1(\Omega)},
     \quad
     T \in (0, T_0),
   \end{align*}
  where in the last step it has been used that $\vec c \geq \vec 0$ and, therefore,
   \begin{align*}
    \int_\Omega \ip{\bb D \vec c(t,\vec z)}{\vec e} \dd \vec z
     &\leq d_\mathrm{max}  \int_\Omega \ip{\vec c(t,\vec z)}{\vec e} \dd \vec z
     = d_\mathrm{max}  \int_\Omega \ip{\vec c^0(\vec z)}{\vec e} \dd \vec z
     \leq \frac{d_\mathrm{max}}{d_\mathrm{min}} \int_\Omega \ip{\bb D \vec c^0}{\vec e} \dd \vec z.
   \end{align*}
  One may thus take
   \[
    C
     = \frac{d_\mathrm{max}}{d_\mathrm{min}} T_0 \norm{\vec c^0 \cdot \vec e}_{\LL_\infty(\Omega)} \, \norm{\bb D \vec c^0 \cdot \vec e}_{\LL^1(\Omega)}.
   \]
  \emph{Entropy identity:}
  By the theorem on derivatives of parameter-integrals, and as the derivative of the function $(0,\infty) \in x \mapsto x (\ln x - 1)$ is $\ln x $ for all $x \in (0, \infty)$, one finds that
   \begin{align*}
    &\frac{\dd}{\dd t} \int_\Omega \sum_i c_i (\mu_i^0 + \ln (c_i) - 1) \dd \vec z
     \\
     &= \int_\Omega \partial_t c_i (\mu_i^0 + \ln (c_i)) \dd \vec z
     = \int_\Omega \sum_i (d_i \Delta c_i + r_i(\vec c)) (\mu_i^0 + \ln (c_i))
     \\
     &= - \int_\Omega \sum_i \frac{d_i \abs{\nabla c_i}^2}{c_i} \dd \vec z
      + \int_\Sigma \sum_i (\mu_i^0 + \ln (c_i)) d_i \partial_{\vec n} c_i \dd \sigma(\vec z)
      + \int_\Omega \sum_i r_i(\vec c) (\mu_i^0 + \ln c_i) \dd \vec z
   \end{align*}
  The assertion will be established if $\sum_i (\mu_i^0 + \ln (c_i)) d_i \partial_{\vec n} c_i \dd \sigma(\vec z) = 0$ can be proved.
  From the boundary conditions $\ip{\vec e^k}{\partial_{\vec n} (\bb D \vec c)} = 0$, there are scalar functions $\eta_a: [0,T_0) \times \Sigma \rightarrow \R$ such that $\partial_{\vec n} (\bb D \vec c)|_\Sigma = \sum_{a=1}^{m^\Sigma} \eta_a \vec \nu^{\Sigma,a}$.
  Hence,
   \begin{align*}
    \sum_i (\mu_i^0 + \ln (c_i)) d_i \partial_{\vec n} c_i
     &= (\vec \mu^0 + \ln (\vec c)) \cdot \partial_{\vec n} (\bb D \vec c)
     \\
     &= \sum_a \eta_a (\vec \mu^0 + \ln(\vec c)) \cdot \vec \nu^{\Sigma,a}
     = \sum_a \eta_a \sum_{i=1}^N (\mu_i^0 + \ln(c_i)) \nu_i^{\Sigma,a}
     \\
     &= \sum_a \eta_a \sum_{i=1}^N \mu^{\Sigma,a} \nu_i^{\Sigma,a}
     = \sum_a \eta_a \mathcal{A}^\Sigma_a
     = 0
   \end{align*}
  since $\mu_i|_\Sigma = \mu^\Sigma$ at all times $t \geq 0$ and positions $\vec z \in \Sigma$ (sorption processes in equilibrium for the fast-sorption--fast-surface-chemistry limit model) and $\mathcal{A}^\Sigma_a = 0$ at all times $t \geq 0$ and all positions $\vec z \in \Sigma$ (chemical reactions on the surface in equilibrium).
 Therefore, this contribution to the sum vanishes, and the entropy identity follows by the fundamental theorem of calculus.
 \end{proof}
 
 \section{Appendix: $\LL_p$-maximal regularity for parabolic and elliptic boundary value problems with boundary conditions of varying differentiability orders}
 \label{sec:appendix}
 
 General $\LL_p$-maximal regularity and $\LL_p$-$\LL_q$-optimality results as in \cite{DeHiPr03} and \cite{DeHiPr07} for linear parabolic systems on bounded domains have been typically formulated for boundary conditions of the same type, so either on the Dirichlet data of the functions or on its normal flux through the boundary.
 These results are heavily based on harmonic analysis and multiplier theory on Banach spaces of class $\mathcal{HT}$ (UMD-spaces), see, e.g., \cite{BerGil94}, \cite{CleDePSukWit00}, \cite{KalWei01}.
 In this manuscript, however, we encounter the situation where the boundary conditions have a combined type, i.e.\ on some parts of the vector of concentrations $\vec c$, or, more precisely, some linear combinations of them corresponding to the stoichiometric vectors (nonlinear, or, in the linearised version, linear and inhomogeneous) Dirichlet data are prescribed, whereas on other parts, more precisely, on linear combinations corresponding to conserved quantities, no-flux boundary conditions have been imposed. For example, the $\LL_p$-maximal regularity results in \cite{DeHiPr03}, Theorem 7.11 (half-space problem) and Theorem 8.2 (domains with compact boundary), rely on a version of the Lopatinskii-Shapiro condition in which all the boundary operators $\mathcal{B}_j$, $j = 1, \ldots, N$, are \emph{homogeneous} of degree $m_j \in \N_0$ and some perturbation argument. In the definition of a principle symbol as used there, e.g.\ in \cite[Section 8.1]{DeHiPr03}, the principle part of the boundary operators considered here would not 'see' the Dirichlet type contribution, as their order of differentiation is $0$, so strictly smaller than the order $1$ for the Neumann type parts. In this appendix, therefore, some comments will be made on a possible extension of the results in \cite{DeHiPr03} and \cite{DeHiPr07} to boundary conditions of combined type. For this purpose, one has to closely follow the lines of \cite[Chapter 6--8]{DeHiPr03} and \cite{DeHiPr07} to observe which adjustments are needed to transfer the results there to this slightly more general setup.
 
 \subsection{Partial Fourier transforms}
 In \cite[Subsection 6.1]{DeHiPr03}, the authors start from the elliptic boundary value problem on the half-space $\R_+^n := \R^{n-1} \times (0,\infty)$ given as
  \begin{align}
   \lambda u + \mathcal{A}(D) u
    &= f
    &&\text{in } \R_+^n,
    \label{DHP03:6.3}
    \\
   \mathcal{B}_j(D) u
    &= g_j
    &&\text{on } \partial \R_+^n, \, j = 1, \ldots, m,
    \label{DHP03:6.4}
  \end{align}   
 where $\mathcal{A}(D)$ and $\mathcal{B}_j(D)$ only consist of principal parts
  \[
   \mathcal{A}(D) = \sum_{\abs{\vec \alpha} = 2m} a_{\mathcal{\vec \alpha}} D^{\vec \alpha},
   \quad
   \mathcal{B}_j(D) = \sum_{\abs{\vec \beta} = m_j} b_{j,\vec \beta} D^{\vec \beta}
  \]
 for $m \in \N$ and $m_j \in \{0, 1, \ldots, 2m-1\}$ and with constant coefficients $a_{\vec \alpha}, b_{j,\vec \beta} \in \B(E)$, and where $E$ is some Banach space (which only later will be assumed to be of class $\mathcal{HT}$, i.e.\ an UMD-space). Also, the standard notation $D^{\vec \alpha} = \prod_{i=1}^n D_{x_i}^{\alpha_i}$ with $D_{x_i} = - \ii \frac{\partial}{\partial x_i}$ and $\abs{\vec \alpha} = \sum_{i=1}^n \alpha_i$ for $\vec \alpha = (\alpha_1, \ldots, \alpha_n)^\mathsf{T} \in \N_0^n$ is used. This makes the symbols
  \begin{equation}
   \mathcal{A}(\vec \xi)
    = \sum_{\abs{\vec \alpha} = 2m} a_{\vec \alpha} \vec \xi^{\vec \alpha},
    \quad
   \mathcal{B}_j(\vec \xi)
    = \sum_{\abs{\vec \beta} = m_j} b_{j,\vec \beta} \vec \xi^{\vec \beta}
    \label{DHP03:6.2}
  \end{equation}
 homogeneous in $\vec \xi$ of order $2m$ for $\mathcal{A}$ and $m_j$ for $\mathcal{B}_j$, respectively.
 In the situation considered here, such a condition on the boundary operators is too restrictive. Therefore, we assume the following adjusted version concerning the operators $\mathcal{B}_j$ instead:
  \begin{assumption}
   The symbol $\mathcal{A} = \sum_{\abs{\vec \alpha} = 2m} a_{\vec \alpha} \vec \xi^{\vec \alpha}$ is homogeneous of degree $2m$, whereas for $j = 1, \ldots, m$ there are projections $\mathcal{P}_{j,k} \in \B(E)$, $k = 0, 1, \ldots, m_j < 2m$ such that $\mathcal{P}_{j,k} \mathcal{P}_{j,k'} = 0$ for $k \neq k'$, and
    \[
     \mathcal{B}_j(\vec \xi)
      = \sum_{k=0}^{m_j} \mathcal{B}_{j,k}(\vec \xi)
      \quad \text{with} \quad
     \mathcal{B}_{j,k}
      =\sum_{\abs{\vec \beta} = k} b_{j,\vec \beta} \vec \xi^{\vec \beta}  \mathcal{P}_{j,k}.
    \] 
  \end{assumption}
 Moreover, we assume that the symbol $\mathcal{A}(D)$ is parameter-elliptic with angle of ellipticity $\phi_{\mathcal{A}} \in [0, \pi)$, i.e.\
  \begin{equation}
   \sigma(\mathcal{A}(\vec \xi)) \subseteq \Sigma_{\pi - \phi},
    \quad
    \vec \xi \in \R^n \text{ with } \abs{\vec \xi} = 1
  \end{equation}
  for some $\phi \in [0, \pi)$ and $\phi_{\mathcal{A}} \in [0, \pi)$ is the infimum of these $\phi$.
 Here, one defines the sector $\Sigma_{\pi - \phi} = \{ z \in \C \setminus \{0\}: \, \abs{\arg(z)} < \pi - \phi \}$ for $\phi \in [0, \pi)$.
 For $p \in (1, \infty)$ and $\phi > \phi_{\mathcal{A}}$ consider the boundary value problem \eqref{DHP03:6.3}--\eqref{DHP03:6.4} for given data $\lambda \in \Sigma_{\pi-\phi}$, $f \in \LL_p(\R_+^{n+1};E)$ and $g_j^k \in \WW_p^{2m-k}(\R_+^{n+1};\ran(\mathcal{P}_{j,k}))$ for $j = 1, \ldots, m$ and $k = 0, 1, \ldots, m_j < 2m$.
 The function $f$ can be extended trivially by zero to $E_0 f \in \LL_p(\R^{n+1};E)$.
 Moreover, we set $g_j = \sum_{k=0}^{m_j} g_j^k = \sum_{k=0}^{m_j} \mathcal{P}_{j,k} g_j$ (by invariance of $\ran(\mathcal{P}_{j,k})$ under the projection $\mathcal{P}_{j,k})$).
 A common strategy is the following:
 First, solve the inhomogeneous problem
  \[
   \lambda v - \mathcal{A}(D) v
    = E_0 f
  \]
 in the full space $\R^{n+1}$, then consider the problem for $w := u - Pv$ which is homogeneous in $\R_{n+1}^+$, but inhomogeneous on the boundary $\partial \R_+^{n+1}$.
 By this, the resulting function $w := u - P (\lambda + A_{\R^{n+1}})^{-1} E_0 f =: u - w_0$, where $A_{\R^{n+1}}$ is the $\LL_p(\R^n)$-realisation of $\mathcal{A}(D)$ (see \cite[Section 5]{DeHiPr03}) and $P$ is the restriction operator, restricting any function on $\R^{n+1}$ to a function on $\R_+^{n+1}$, has to solve the abstract boundary value problem
  \begin{align*}
   \lambda w + \mathcal{A}(D) w
    &= 0
    &&\text{in } \R_+^{n+1}
    \\
   \mathcal{B}_j(D) w
    &= g_j - \mathcal{B}_j(D) w_0
    &&j = 1, \ldots, m, \, \text{on } \partial \R_+^{n+1},
  \end{align*}
 so that w.l.o.g.\ we may and will assume that $f = 0$ in \eqref{DHP03:6.3}--\eqref{DHP03:6.4} in the following.
 For functions $u = u(\vec z',y) \in \LL_1(\R_+^{n+1}) \cap \LL_p(\R_+^{n+1})$ one may now introduce the partial Fourier transform in $\vec z'$ as
  \[
   \mathcal{F} u(\vec \xi',y)
    = \int_{\R^n} \ee^{- \ii \vec z' \cdot \vec \xi'} u(\vec z',y) \dd \vec z'.
  \]
 This leads to the boundary value problem in the Laplace-Fourier space
  \begin{align}
   \lambda \mathcal{F} u (\vec \xi',y) + \sum_{l=0}^{2m} a_l(\vec \xi') D_y^{2m-l} \mathcal{F} u(\vec \xi',y)
    &= \mathcal{F}f (\vec \xi',y),
    &&\text{in } \R_+^{n+1},
    \label{DHP03:6.5}
    \\
   \sum_{k=0}^{m_j} \sum_{l=0}^k \sum_{\abs{\vec \beta} = l} \tilde b_{jkl}(\vec \xi') D_y^{k-l} \mathcal{P}_{j,k} \mathcal{F} u(\vec \xi',0)
    &= \mathcal{F} g_j(\vec \xi',0),
    &&j = 1, \ldots, m, \, \text{on } \partial \R_+^{n+1},
    \label{DHP03:6.6}
  \end{align}
 where $\tilde b_{jkl}(\vec \xi') = \sum_{\abs{\vec \gamma} = l} b_{j,(\vec y,l)}(\vec \xi') (\vec \xi')^{\vec \gamma}$.
 The ODE \eqref{DHP03:6.5} may then be rewritten as a first order system by considering the $E^{2m}$--valued function
  \[
   \mathcal{F} v
    := (\mathcal{F} u, \frac{1}{\rho} D_y \mathcal{F} u, \ldots, \frac{1}{\rho^{2m-1}} D_y^{2m-1} \mathcal{F} u)^\mathsf{T}
  \]
 for some $\rho > 0$. To exploit homogeneity of $\mathcal{A}(\vec \xi)$, $\mathcal{B}_{j,k}(\vec \xi)$ later on, one further introduces the variables
  \[
   \sigma := \frac{\lambda}{\rho^{2m}},
    \quad \text{and} \quad
   \vec b := \frac{\vec \xi'}{\rho}
  \]
 and defines a matrix-valued function $\bb A_0(\vec \xi',\lambda)$ as
  \[
   \bb A_0(\vec \xi',\lambda)
    = \left[ \begin{array}{ccccc} 0 & I & 0 & \cdots & 0 \\ 0 & 0 & I && 0 \\ \vdots & \vdots &\ddots &\ddots & \\ 0 & \cdots && 0 & I \\ c_{2m}(\vec \xi',\lambda) & c_{2m-1}(\vec \xi') & \cdots && c_1(\vec \xi') \end{array} \right],
  \]
 where
  \[
   c_j(\vec \xi') = - a_0^{-1} a_j(\vec \xi') \, (1 \leq j < 2m), \quad c_{2m}(\vec \xi',\lambda) = - a_0^{-1} (a_{2m}(\vec \xi') + \lambda).
  \]
 Note that the parameter-ellipticity of the symbol $\mathcal{A}(\vec \xi)$ implies, in particular, the invertibility of $a_0(\vec \xi') = a_0$.
 Just as in \cite{DeHiPr03}, the coefficients $a_j(\vec \xi)$ are $\vec \xi'$-homogeneous of degree $j$, and \eqref{DHP03:6.5} for the special case $f=0$ is equivalent to the first order in $y$ system of ordinary differential equations
  \begin{equation}
   D_y \mathcal{F} v
    = \rho \bb A_0(\vec b, \sigma) \mathcal{F} v,
    \quad
    \vec \xi' \in \R^n, \, y > 0.
    \label{DHP03:6.7}
  \end{equation}
 In contrast to the situation in \cite{DeHiPr03}, however, the boundary symbols $\mathcal{B}_j(\vec \xi') = \sum_{k=0}^{m_j} \mathcal{B}_{j,k}(\vec \xi)$ are not homogeneous in $\vec \xi'$ of degree $m_j$ in general, but only the symbols $\mathcal{B}_{j,k}(\vec \xi)$ are homogeneous in $\vec \xi$ of order $k$, so that
  \[
   \mathcal{F} \mathcal{B}_j(\rho \vec b,0)
    = \sum_{k=0}^{m_j} \sum_{l=0}^k \tilde b_{jkl}(\rho \vec b) D_y^{k-l} \mathcal{P}_{j,k} \mathcal{F} u(\rho \vec b,0)
    = \sum_{k=0}^{m_j} \sum_{l=0}^k \rho^l \tilde b_{jkl}(\vec b) D_y^{k-l} \mathcal{P}_{j,k} \mathcal{F} u(\rho \vec b,0)
  \]
 and the boundary condition \eqref{DHP03:6.6} is equivalent to
  \begin{equation}
   \mathcal{B}_{j,k}^0(\vec b) \mathcal{F} v(\rho \vec b, 0)
    =\frac{\mathcal{P}_{j,k} \mathcal{F} g_j(\rho \vec b,0)}{\rho^k},
    \quad
    j = 1, \ldots, m, \, k = 0, 1, \ldots, m_j,
  \end{equation}
 where one sets
  \[
   \mathcal{B}_{j,k}^0(\vec b)
    = (\tilde b_{jkk}(\vec b), \ldots, \tilde b_{jk0}, 0, \ldots, 0)^\mathsf{T},
    \quad
    j = 1, \ldots, m, \, k = 0, 1, \ldots, m_j.
  \]
 By \cite[Proposition 6.1]{DeHiPr03}, the spectrum of $\bb A_0(\vec b, \sigma)$ is strictly separated by the real axis, i.e.\ $\sigma(\bb A_0(\vec b, \sigma)) \cap \R = \emptyset$ and $\ii \sigma(\bb A_0(\vec b,\sigma))$ splits into two parts $S_\pm(\vec b,\sigma)$ which are strictly separated by the imaginary axis.
 More precisely, there are $c_1, c_2 > 0$ such that
  \begin{align}
   \inf \{ \Re \mu: \, \mu \in S_+(\vec b, \sigma) \}
    &\geq c_1,
    &&\abs{\vec b} \leq 1, \, \abs{\sigma} \leq 1,
    \label{DHP03:6.11}
    \\
   \sup \{ \Re \mu: \, \mu \in S_-(\vec b, \sigma) \}
    &\leq - c_2,
    &&\abs{\vec b} \leq 1, \, \abs{\sigma} \leq 1.
    \label{DHP03:6.12}
  \end{align}
 Therefore, the spectral projections $P_\pm(\vec b,\sigma)$ onto $S_\pm(\vec b,\sigma)$ exist.
 
 \subsection{The Lopatinskii-Shapiro condition on the half space}
 
 Similar to the situation in \cite{DeHiPr03}, the \emph{Lopatinskii--Shapiro condition} for the half-space problem may be expressed as follows.
 
 \begin{assumption}[Lopatinskii--Shapiro condition]
  For each $\lambda \in \overline{\Sigma}_{\pi - \phi}$ and $\vec \xi' \in \R^n$ such that $(\lambda, \vec \xi') \neq (0, \vec 0)$, the problem
   \begin{align*}
    \lambda v(y) + \mathcal{A}(\vec \xi',D_y) v(y)
     &= 0,
     &&y > 0,
     \\
    \mathcal{B}_j(\vec \xi',D_y) v(0)
     &= g_j,
     &&j = 1, \ldots, m
   \end{align*}
  admits a unique solution $\vec v \in \CC_0(\R_+; E)$ for each $\vec g = (g_1, \ldots, g_m)^\mathsf{T} \in E^m$.
 \end{assumption}

The proofs of \cite[Proposition 6.2--6.4]{DeHiPr03} then directly transfer to this slightly modified situation to give
 \begin{proposition}
  Let $\mathcal{A}(D)$ be a parameter elliptic operator of order $2m$ and angle of ellipticity $\phi_\mathcal{A} \in [0, \pi)$. Let $\phi > \phi_\mathcal{A}$ and $\lambda \in \overline{\Sigma}_{\pi-\phi}$ and assume that the Lopatinskii-Shapiro condition holds true.
  Then there exists a unique solution $\mathcal{F}u$ of the system \eqref{DHP03:6.5}--\eqref{DHP03:6.6}, which is given by the sum
   \[
    \mathcal{F} u
     = (\mathcal{F} v)_1 + \mathcal{F} w_1 + (\mathcal{F} w_2)_1
   \]
  where $(\cdot)_1$ denotes the first component of an element in $E^{2m}$, and
   \begin{align}
    \mathcal{F} v(\vec \xi',y,\lambda)
     &= \ee^{\ii \rho \bb A_0(\vec b,\sigma)} \mathcal{F}v(\vec \xi',0,\lambda),
     \quad
     \vec \xi' \in \R^n, \, y > 0,
     \label{DHP03:6.13}
     \\
    \intertext{where }
   \mathcal{F} v(\vec \xi',0,\lambda)
    &= \bb M(\vec b,\sigma) \mathcal{F} g_\rho(\vec \xi',0)
    \\
    \intertext{for some function $\bb M(\vec b,\sigma): E^m \rightarrow E^{2m}$, which is jointly holomorphic in $\vec b$ and $\sigma$, and with}
   \mathcal{F} g_\rho(\vec \xi',0)
    &= \left( \sum_{k=0}^{m_1} \tfrac{\mathcal{P}_{1,k} \mathcal{F} g_1(\vec \xi',0)}{\rho^k}, \ldots, \sum_{k=0}^{m_m} \tfrac{\mathcal{P}_{m,k} \mathcal{F} g_m(\vec \xi',0)}{\rho^k} \right),
    \\
   \mathcal{F} w_1(\vec \xi',y,\lambda)
    &= (\mathcal{F} P (\lambda + A_{\R^{n+1}})^{-1} E_0 f)(\vec \xi',y),
    \label{DHP03:6.18}
    \\
   \mathcal{F} w_2(\vec \xi',y,\lambda)
    &= - \ee^{\ii \rho \bb A_0(\vec b,y)} \bb M(\vec b,\sigma) (\mathcal{F} g_\rho^B)(\vec \xi',0),
    \label{DHP03:6.19}
    \\
    \intertext{where}
    \mathcal{F} g_\rho^B(\vec \xi',0)
     &= \left( \sum_{k=0}^{m_1} \tfrac{\mathcal{F} g_{1,k}^B(\vec \xi',0)}{\rho^k}, \ldots, \sum_{k=0}^{m_m} \tfrac{\mathcal{F} g_{m,k}^B(\vec \xi',0)}{\rho^k} \right)
    \\
    \intertext{with}
    \mathcal{F} g_{j,k}^B(\vec \xi',0)
     &= \int_0^\infty h_{j,k}(\vec \xi',s) (\mathcal{F} f)(\vec \xi',s) \dd s,
     \quad
     j = 1, \ldots, m, \, k = 0, 1, \ldots, m_j
     \\
     \intertext{and}
    h_{j,k}(\vec \xi',s)
     &= \int_{\R} \mathcal{B}_j(\vec \xi',\eta) \mathcal{P}_{j,k}
     (\lambda + \mathcal{A}(\vec \xi',\eta))^{-1} \ee^{- \ii \eta s} \dd \eta.
   \end{align}
 \end{proposition}
 
 \begin{proof}
  The proposition is established in the following way:
  First, the linear problem is decomposed into three sub-problems, each of which is easier to handle than the original problem.
  The first of these is the problem \eqref{DHP03:6.5}--\eqref{DHP03:6.6} for the special case $f = 0$, i.e.\
   \[
    \begin{cases}
     \lambda \mathcal{F} u - \mathcal{A}(D) \mathcal{F} u
      = 0
      &\text{in } \R_+^{n+1}
      \\
     \mathcal{B}_j(D) \mathcal{F} u
      =g_j
      &\text{on } \partial \R_+^{n+1}, \, j = 1, \ldots, m.
    \end{cases}
   \]
  Secondly, one extends $f$ by zero to $E_0 f$ on the full space $\R^{n+1}$ and solves the elliptic problem on the full space
  \[
   \lambda \mathcal{F} u - \mathcal{A}(D) \mathcal{F} u
    = E_0 f
    \quad \text{in } \R^{n+1},
  \]
 which constitutes the solution $\mathcal{F} w_1$ to the inhomogeneous problem
  \[
    \begin{cases}
     \lambda \mathcal{F} u - \mathcal{A}(D) \mathcal{F} u
      = f
      &\text{in } \R_+^{n+1}
      \\
     \mathcal{B}_j(D) \mathcal{F} u
      =: \mathcal{B}_j(D) \mathcal{F} w_1
      &\text{on } \partial \R_+^{n+1}, \, j = 1, \ldots, m.
    \end{cases}
  \]
 In the last step one returns to the first problem, but replaces $g_j$ by $- \mathcal{B}_j(D) \mathcal{F} w_1$, i.e.\ minus the boundary data for the solution to the second auxiliary problem on the full space:
   \[
    \begin{cases}
     \lambda \mathcal{F} u - \mathcal{A}(D) \mathcal{F} u
      = 0
      &\text{in } \R_+^{n+1}
      \\
     \mathcal{B}_j(D) \mathcal{F} u
      = - \mathcal{B}_j(D) w_1
      &\text{on } \partial \R_+^{n+1}, \, j = 1, \ldots, m.
    \end{cases}
   \]
  The only thing left to do there, is to derive a formula for the boundary values $\mathcal{B}_j(D) w_1$ on $\partial \R_+^{n+1}$. \newline
  As a first step, the case $f = 0$ may be handled analogously to \cite[Proposition 5.3]{DeHiPr03} which provides us with a unique solution $(\mathcal{F} v)_1$ to the problem \eqref{DHP03:6.3}--\eqref{DHP03:6.4} for $f = 0$, and which has the form \eqref{DHP03:6.13}, where $\bb M(\vec b,\sigma)$ is jointly holomorphic in the variables $(\vec b,\sigma)$. Since the projections $\mathcal{P}_{j,k}$ project the problem to subspaces on which the corresponding system for the boundary operators is homogeneous in $\vec \xi'$ of degree $k$, uniqueness and existence of solutions, as well as the representation formula for the solution follow via the equivalent reformulation as a first order system. It, therefore, remains to prove holomorphy of the map $M$. This can be done almost literally as in the proof of \cite[Proposition 6.2]{DeHiPr03}:
  First, employing the closed graph theorem, $M(\vec z)$ is a linear and closed map $E^m \rightarrow E^{2m}$ is uniformly bounded, first on the set $D = \overline{\BB_1(0)}^{\R^n} \times \overline{\Sigma}_{\pi-\phi} \subseteq \R^n \times \C$, and then in a complex neighourhood $\tilde D$ of $D$, which is sufficiently close to $D$. Using the spectral projection $P_-$ and the Lopatinskii-Shapiro condition, one can then proceed by demonstrating continuity of $v(z) = M(z) g$, for any fixed $g \in E^m$, on complex lines, also using that $P_-$ and the operators $\BB_{j,k}^0$ are continuous.
  Thereof, complex differentiability follows rather easily. \newline
  For the solution of the full space problem, one may employ the full space theory developed in \cite[Section 5]{DeHiPr03}, to write
   \[
    \mathcal{F} w_1(\vec \xi', y, \lambda)
     = (\mathcal{F} P (\lambda + A_{\R^{n+1}})^{-1} E_0 f)(\vec \xi',y).
   \]
  Here, as before, we write $E_0$ for the trivial extension from $\R_+^{n+1}$ to $\R^{n+1}$, $P$ for the restriction from $\R^{n+1}$ to $\R_+^{n+1}$, and $A_{\R_{n+1}}$ for the $\LL_p(\R^{n+1})$-realisation of the differential operator $\mathcal{A}(D)$.
  In the last of the three sub-problems, one further employs a following representation of the boundary values $\mathcal{F} g_j^B(\vec \xi',0) = (\mathcal{B}_j(D) \mathcal{F} P (\lambda + A_{\R^{n+1}})^{-1} E_0 f)(\vec \xi',0)$, which is given by $\mathcal{F} g_j^B(\vec \xi',0) = \sum_{k=0}^{m_j} g_{j,k}^B(\vec \xi,0)$ with
   \begin{align*}
    \mathcal{F} g_{j,k}^B(\vec \xi',0)
     &= \int_0^\infty h_{j,k}(\vec \xi',s) (\mathcal{F} f)(\vec \xi',s) \dd s
     \\
    h_{j,k}(\vec \xi',s)
     &= \int_{\R} \mathcal{B}_j(\vec \xi',\eta) \mathcal{P}_{j,k} (\lambda + \mathcal{A}(\vec \xi',\eta))^{-1} \ee^{-\ii \eta s} \dd \eta,
     \quad
     1 \leq j \leq m, \, 0 \leq k \leq m_j < 2m.
   \end{align*}
  Then, considerations as in the first step lead to the desired representation of $(\mathcal{F} w_2)_1$, and the proposition follows.
 \end{proof}
 
 \subsection{Kernel estimates}

 To derive kernel estimates, one defines (analogously to \cite[Section 6.3]{DeHiPr03}) weighted versions of $\mathcal{F} g_\rho$ and the semigroup associated to $\ii \rho \bb A_0(\vec b, \sigma)$ according to
  \begin{align}
   \mathcal{F} g_\rho^{2m}(\vec \xi',0)
    &= \rho^{2m} \mathcal{F} g_\rho(\vec \xi',0),
    \label{DHP03:6.20}
    \\
   \mathcal{F} v_\rho^{2m}(\vec \xi',y,\lambda)
    &= \frac{1}{\rho^{2m}} \ee^{\ii \rho \bb A_0(\vec b,\sigma) y}
    \label{DHP03:6.21},
  \end{align} 
 so that, by formula \eqref{DHP03:6.13} the identity
  \[
   \mathcal{F} v(\vec \xi',y,\lambda)
    = \mathcal{F} v_\rho^{2m}(\vec \xi',y,\lambda) \mathcal{F} g_\rho^{2m}(\vec \xi',0)
  \]
 holds true.
 Let us further denote by
  \begin{equation}
   K_\lambda(\vec z',y)
    = \mathcal{F}^{-1} (\mathcal{F} v_\rho^{2m})(\vec z',y)
    = \frac{1}{(2\pi)^n} \int_{\R^n} \ee^{\ii \vec z' \cdot \vec \xi'} \mathcal{F} v_\rho^{2m}(\vec \xi',y,\lambda) \dd \vec \xi',
    \quad
    \vec z' \in \R^n, \, y > 0
    \label{DHP03:6.22}
  \end{equation}
 the inverse (partial) Fourier transform of the $\mathcal{B}(E)$-valued function $\mathcal{F} v_\rho^{2m}$.
% Next, fix $\vec z' \in \R^n$ and let $\bb Q$ be a rotation in $\R^n$ such that $Q \vec z' = (\abs{\vec z'}, 0, \ldots, 0)^\mathsf{T}$.
 Following along the lines of \cite[Section 6.3]{DeHiPr03}, one introduces
  \[
   p_{m,k}^n: \R_+ \rightarrow \R,
    \quad
   p_{m,k}^n(r)
    := \int_0^\infty \frac{s^{n-2}}{(1+s)^{m-1-k}} \ee^{-(s+1)r} \dd s.
  \]
  
  \begin{lemma}
   For all $c,y > 0$ and $m, n, k \in \N_0$, the following identity is valid:
    \[
     \int_0^\infty p_{m,k}^n (c (y + r)) r^{n-1} \dd r
      = \frac{(n-1)!}{c^n} p_{m+n,k}^n(cy).
    \]
  \end{lemma}
  \begin{proof}
  This identity follows straightforward via integration by parts, see \cite[Corollary 5.3]{DeHiPr03}.
  \end{proof}
  
 From this, we obtain the following kernel estimates.
 
 \begin{proposition}
 \label{DHP03:Prop_6.5}
 \label{DHP03:Prop_6.6}
  Let $\mathcal{A}(D)$ be a parameter elliptic operator of order $2m$ and angle of ellipticity $\phi_\mathcal{A} \in [0,\pi)$.
  Let $\phi > \phi_\mathcal{A}$ and assume that the Lopatinskii-Shapiro condition holds true.
  Then, for every multi-index $\vec \alpha$, there are constants $M,c > 0$ such that
   \begin{align}
    \abs{D^{\vec \alpha} K_\lambda(\vec z',y)}
     &\leq M \cdot \abs{\lambda}^{\tfrac{n-2m+\abs{\vec \alpha}}{2m}} p_{2m,\abs{\vec \alpha}-1}^{n+1} \left( c \abs{\lambda}^{1/(2m)} \left( \abs{\vec z'} + \abs{y} \right) \right),
     \quad
     \vec z' \in \R^n, \, y > 0, \, \lambda \in \Sigma_{\pi-\phi}
     \label{DHP03:6.24}
     \\
     \intertext{and}
    \abs{D^{\vec \alpha} K_{\lambda}^{w,j}(\vec z',y,\tilde y)}
     &\leq M \abs{\lambda}^{\frac{n-2m+1+\abs{\vec \alpha}}{2m}} p_{2m,\abs{\vec \alpha}}^{n+1}\left(c \abs{\lambda}^{1/(2m)} (\abs{\vec z'} + y + \tilde y) \right),
     \quad
     \vec z' \in \R^n, \, y, \tilde y > 0, \, \lambda \in \Sigma_{\pi-\phi}. 
    \label{DHP03:6.26}
   \end{align}
 \end{proposition}
 
 \begin{proof}
  We proceed as in the proof of \cite[Proposition 6.5]{DeHiPr03}.
  The proof of the kernel estimates heavily uses the theory of complex functions, in particular Cauchy's Theorem.
  Also, it is uses that the formula
   \[
    \mathcal{F} v(\vec \xi',y,\lambda)
     = \mathcal{F} v_\rho^{2m}(\vec \xi',y,\lambda) \mathcal{F} g_\rho^{2m}(\vec \xi',0),
     \quad
     \vec \xi' \in \R^n, \, y > 0, \, \lambda \in \Sigma_{\pi-\phi}
   \]
  is valid independently of the particular choice of $\rho > 0$, so that later on, we may choose variables $(a,r)$ and then adjust $\rho = (\abs{\lambda}^{1/m} + a^2 + r^2)^{\nicefrac{1}{2}}$ accordingly.
  Fix $\vec z' \in \R^n$ and let $\bb Q \in \R^{n \times n}$ be a rotation mapping $\vec z'$ to $\bb Q \vec z' = (\abs{\vec z'},0, \ldots, 0)^\mathsf{T}$. Using polar coordinates one writes $\bb Q \vec \xi' = (a,r \vec \varphi)^\mathsf{T}$ where $a \in \R$, $r > 0$ and $\vec \varphi \in \S^{n-2} = \{\vec z \in \R^{n-1}: \, \abs{\vec z} = 1\}$. 
 Since, by definition, $\sigma = \tfrac{\lambda}{\rho^{2m}}$ and $\vec b = \tfrac{\vec \xi'}{\rho}$, one then obtains via the transformation rule (note that $\abs{\det \bb Q} = 1$ and $\bb Q \bb Q^\mathsf{T} = \bb Q^\mathsf{T} \bb Q = \bb I$) that
  \begin{align*}
   K_\lambda(\vec z',\lambda)
    &= \frac{1}{(2\pi)^n} \int_{\R^n} \ee^{\ii \vec z' \cdot \vec \xi'} \mathcal{F} v_\rho^{2m}(\vec \xi',y,\lambda) \dd \vec \xi'
    \\
    &= \frac{1}{(2\pi)^n} \int_{\R^n} \ee^{\ii \bb Q \vec z' \cdot \bb Q \vec \xi'} \ee^{\ii \rho \bb A_0(\vec b,\sigma) y} \bb M(\vec b,\sigma) \frac{1}{\rho^{2m}} \dd \vec \xi'
    \\
    &= \frac{1}{(2\pi)^n} \int_{\S^{n-2}} \int_0^\infty r^{n-2} \int_{-\infty}^\infty \ee^{\ii \abs{\vec z'} a} \ee^{\ii \rho A_0(Q^\mathsf{T} (\tfrac{a}{\rho}, \tfrac{r}{\rho} \vec \varphi), \tfrac{\lambda}{\rho^{2m}}) y} M(Q^\mathsf{T}(\tfrac{a}{\rho}, \tfrac{r}{\rho} \vec \varphi), \tfrac{\lambda}{\rho^{2m}}) \frac{1}{\rho^{2m}} \dd a \dd r \dd \sigma(\vec \varphi)
    \\
    &= \frac{1}{(2\pi)^n} \int_{\S^{n-2}} \int_0^\infty r^{n-2} \int_{-\infty}^\infty \ee^{\ii \abs{\vec z'} a} \frac{1}{(\abs{\lambda}^{1/m} + a^2 + r^2)^m} F_1(\lambda,a,r,\vec \varphi, y) F_2(\lambda,a,r,\vec \varphi) \dd a \dd r \dd \sigma(\vec \varphi)
    \\
    \intertext{for}
    F_1(\lambda,a,r,\vec \varphi, y)
     &= \ee^{\ii (\abs{\lambda}^{1/m} + a^2 + r^2) \bb A_0((\abs{\lambda}^{1/m} + a^2 + r^2)^{-\nicefrac{1}{2}} \bb Q^\mathsf{T}(a,r \vec \varphi), \lambda (\abs{\lambda}^{1/m} + a^2 + r^2)^{-m}) y},
     \\
    F_2(\lambda,a,r,\vec \varphi)
     &= \bb M((\abs{\lambda}^{1/m} + a^2 + r^2)^{\nicefrac{1}{2}} \bb Q^\mathsf{T}(a, r \vec \varphi), \lambda (\abs{\lambda} + a^2 + r^2)^{-\nicefrac{1}{2}}).
   \end{align*}
  Since the integrand is holomorphic in $r$ (this, in particular, uses that $\bb M(\vec b,\sigma)$ depends holomorphic on $(\vec b, \sigma)$), the integral $\int_{-\infty}^\infty \ldots \dd s$ may be replaced by a contour integral $\int_{\Gamma_\varepsilon} \cdot \ldots \dd \omega$ where $\Gamma_\varepsilon$ is parametrised by
   \[
    \gamma_\varepsilon(s)
     = s + \ii \varepsilon \left( \abs{\lambda}^{1/m} + r^2 + s^2 \right)^{1/m},
     \quad
     s \in (-\infty,\infty)
   \]
  for some (small) $\varepsilon > 0$. Hence,
  \begin{align*}
   K_\lambda(\vec z',\lambda)
    &= \frac{1}{(2\pi)^n} \int_{\S^{n-2}} \int_0^\infty r^{n-2} \int_{\Gamma_\varepsilon} \ee^{\ii \abs{\vec z'} a} \frac{1}{(\abs{\lambda}^{1/m} + a^2 + r^2)^m} F_1(\lambda,a,r,\vec \varphi, y) F_2(\lambda,a,r,\vec \varphi) \dd a \dd r \dd \sigma(\vec \varphi).
  \end{align*}
 Now one may proceed just as in the proof of \cite[Proposition 6.5]{DeHiPr03}.
 Showing that for $\varepsilon_0 > 0$, but small, there is $c > 0$ such that for all $\varepsilon \in (0, \varepsilon_0]$:
  \[
   c \abs{a^2 + r^2}
    \leq r^2 + \abs{s}^2
    \leq \frac{1}{c} (a^2 + r^2),
  \]
 and, moreover, $\abs{\arg (r^2 + a^2)} \leq c \varepsilon$, so that $\abs{\ee^{\ii \abs{\vec z'}a}} \leq \ee^{-\varepsilon\abs{\vec z'}(r + \abs{s} + \abs{\lambda})^{\nicefrac{1}{2m}}}$, and $\abs{\abs{\lambda}^{1/m} + a^2 + r^2}^{\nicefrac{1}{2}} \geq \abs{\lambda}^{\nicefrac{1}{2}m} + r + s$. Moreover, due to $\inf \{\Re z: \, z \in S_+(\vec b,\sigma)\} \geq c_1$, it follows $\abs{F_1(\lambda,a,r,\vec \varphi,y)} \leq \ee^{-c_1 y (\abs{\lambda}^{\nicefrac{1}{2}m} + r + s)}$, and, therefore, one obtains
  \begin{align*}
   \abs{D^{\vec \alpha} K_\lambda(\vec z',y)}
    &\leq M \abs{\lambda}^{\frac{n-2m+\abs{\vec \alpha}}{2m}} p_{2m,\abs{\vec \alpha}}^{n+1}(c \abs{\lambda}^{\nicefrac{1}{2}m}(\abs{\vec z'} + \abs{y})),
    \quad
    \vec z' \in \R^n, \, y > 0, \lambda \in \Sigma_{\pi-\phi}.
  \end{align*}
 This proves the first kernel estimate.
 \newline
 Next, consider the inverse Fourier transform of $\mathcal{F} w_2 = - \ee^{\ii \rho \bb A_0(\vec b,y)} \bb M(\vec b,\sigma) (\mathcal{F} g_\rho^B)(\vec \xi',0)$, i.e.\
  \begin{equation}
   \mathcal{P}_{j,k} K_\lambda^{w,j}
    = - \frac{1}{(2\pi)^n} \int_{\R^n} \ee^{\ii \vec z' \cdot \vec \xi} \ee^{\ii \rho \bb A_0(\vec b,\sigma) y} \bb M(\vec b,\sigma) \mathcal{P}_{j,k} h_j^B(\vec \xi',y) \dd \vec \xi',
   \label{DHP03:6.25}
  \end{equation}
 where $h_j^B = (0, \ldots, 0, h_j, 0, \ldots, 0)^\mathsf{T}$. Then, by Cauchy's theorem and homogeneity considerations
  \begin{align*}
   h_j(\vec \xi', \tilde y)
    &= \sum_{k=0}^{m_j} \rho^{k+1-2m} \int_{\R} \mathcal{B}_j \mathcal{P}_{j,k} (\sigma + \mathcal{A}(\vec b,\sigma))^{-1} \ee^{-\ii \tilde y \rho a} \dd a
    \\
    &= \sum_{k=0}^{m_j} \rho^{k+1-2m} \int_{\Gamma^-} \mathcal{B}_j \mathcal{P}_{j,k} (\sigma + \mathcal{A}(\vec b,\sigma))^{-1} \ee^{-\ii \tilde y \rho a} \dd a,
  \end{align*}
 where $\Gamma^-$ is a closed curve in the open lower complex plane which encircles $- \ii P_+(\vec b,\sigma)$. One then obtains, after using a rotation as above and shifting the path of integration as in \cite{DeHiPr03}, that
  \begin{equation}
   \abs{\mathcal{P}_{j,k} h_j(\vec \xi',\tilde y)}
    \leq C \left(\abs{\lambda} + r^{2m} + s^{2m} \right)^{\nicefrac{k+1-2m}{2m}} \ee^{- c \tilde y (\abs{\lambda} + r^{2m} + s^{2m})^{\nicefrac{1}{2m}}},
    \quad
    \vec \xi' \in \R^n, \, y > 0 
  \end{equation}
 for small $\abs{\operatorname{arg} (r^2 + a^2)}$. For a multi-index $\vec \alpha$ one thus has the estimate
  \begin{align*}
   \abs{D^{\vec \alpha} K_{\lambda}^{w,j}(\vec z',y, \tilde y)}
    &\leq \sum_{k=0}^{m_j} \abs{D^{\vec \alpha} \mathcal{P}_{j,k} K_{\lambda}^{w,j}(\vec z',y, \tilde y)}
    \\
    &\leq M \abs{\lambda}^{\frac{n-2m+1+\abs{\vec \alpha}}{2m}} \int_0^\infty \frac{s^{n-1}}{(1+s)^{2m-1-\abs{\vec \alpha}}} \ee^{-c (s+1) \abs{\lambda}^{1/(2m)} (\abs{\vec z'} + y + \tilde y}) \dd s,
  \end{align*}
 cf.\ \cite[p.\ 80]{DeHiPr03}.
 \end{proof}
 
 \subsection{Solution operators}
 As in \cite[Section 6.4]{DeHiPr03}, $\LL_p$-estimates for $v$ and $w_2$ can be derived via the kernel estimates from the last subsection.
 For this purpose, one introduces the $\LL_p$-realisation for the boundary value problem \eqref{DHP03:6.3}--\eqref{DHP03:6.4} as the closure $A = \overline{A_\mathrm{min}}$ where the minimal operator $A_\mathrm{min}$ is defined by
  \[
   A_\mathrm{min}:
    \dom(A_\mathrm{min})
    = \WW_p^{2m}(\R_+^{n+1};E)
    \rightarrow \LL_p(\R_+^{n+1}; E) \times \prod_{j=1}^m \left(\sum_{k=0}^{m_j} \mathcal{P}_{j,k} \WW_p^{2m-k}(\R_+^{n+1};E) \right)
  \]
 with
  \[
   A_\mathrm{min} u
    = \left( \begin{array}{c} \mathcal{A}(D) \\ \mathcal{B}_1(D) \\ \vdots \\ \mathcal{B}_m (D) \end{array} \right) u.
  \]
 For $\lambda \in \Sigma_{\pi-\phi}$ and functions $g_j \in \sum_{k=0}^{m_j} \mathcal{P}_{j,k} \WW_p^{2m-k}(\R_+^{n+1};E)$ one then first considers the boundary value problem
  \begin{equation}
   \lambda J v + A v
    = \left( \begin{array}{c} 0 \\ \vec g \end{array} \right),
    \label{DHP03:6.27}
  \end{equation}
 where $\vec g = (g_1, \ldots, g_m)^\mathsf{T}$ and $J v = (v, 0, \ldots, 0)^\mathsf{T}$. For $K_\lambda^j := (K_\lambda)_{1,j}$ with $K_\lambda$ the inverse Fourier transform of $\mathcal{F} v_\rho^{2m}$ as defined in \eqref{DHP03:6.22}, $j = 1, \ldots, m$, one then gets the following result.
  \begin{proposition}
  \label{DHP03:Prop_6.8}
   Let $\lambda \in \Sigma_{\pi-\phi}$ and $g_j \in \sum_{k=0}^{m_j} \mathcal{P}_{j,k} \WW_p^{2m-k}(\R_+^{n+1};E)$ for $j = 1, \ldots, m$.
   Then there is a unique solution $v \in \WW_p^{2m-1}(\R_+^{n+1};E) \cap \dom(A)$ of \eqref{DHP03:6.27} which is given by $v = \sum_{j=1}^m \sum_{k=0}^{m_j} v^{j,k}$ for
    \begin{align*}
     v^{j,k}(\vec z',y)
      &= S_\lambda^{j,k} g_j(\vec z',y)
      \\
      &= - \int_0^\infty \int_{\R^n} \frac{\partial}{\partial \tilde y} K_\lambda^{j,k}(\vec z' - \tilde {\vec z}',y + \tilde y)(- \Delta_{\R^n} + \abs{\lambda}^{1/m})^{\frac{2m-k}{2m}} \mathcal{P}_{j,k} g_j(\tilde {\vec z}', \tilde y) \dd \tilde {\vec z}' \dd \tilde y
      \\ &\quad
      - \int_0^\infty \int_{\R^n} K_\lambda^{j,k} (\vec z' - \tilde {\vec z}', y + \tilde y) \frac{\partial}{\partial \tilde y} (- \Delta_{\R^n} + \abs{\lambda}^{1/m})^{\frac{2m-k}{2m}} \mathcal{P}_{j,k} g_j(\tilde {\vec z}', \tilde y) \dd \tilde {\vec z}' \tilde y
      \\
      &=: S_\lambda^{I,j,k} g_j(\vec z',y) + S_\lambda^{II,j,k}(\vec z',y),
      \quad
      \vec z' \in \R^n, \, y > 0.
    \end{align*}
   Here, $\Delta_{\R^n}$ denotes the $\LL_p(\R^n)$-realisation of the Laplacian and $K_\lambda$ is defined by \eqref{DHP03:6.22}.
   Moreover, there is $C > 0$ such that for $\lambda \in \Sigma_{\pi-\phi}$ and $j = 1, \ldots, m$ the following estimates hold true:
    \begin{align*}
     \norm{\lambda^{1-\frac{\abs{\vec \alpha}}{2m}} D^{\vec \alpha} S_\lambda^{I,j,k} g_j}_{\LL_p(\R_+^{n+1};E)}
      &\leq C \norm{(- \Delta_{\R^n} + \abs{\lambda}^{1/m})^{\frac{2m-k}{2}} \mathcal{P}_{j,k} g_j}_{\LL_p(\R_+^{n+1};E)},
      &&0 \leq \abs{\vec \alpha} \leq 2m-1,
      \\
     \norm{\lambda^{1-\frac{\abs{\vec \alpha}}{2m}} D^{\vec \alpha} S_\lambda^{II,j,k} g_j}_{\LL_p(\R_+^{n+1};E)}
      &\leq C \norm{(- \Delta_{\R^n} + \abs{\lambda}^{1/m})^{\frac{2m-k-1}{2}} \tfrac{\partial}{\partial y} \mathcal{P}_{j,k} g_j}_{\LL_p(\R_+^{n+1};E)},
      &&0 \leq \abs{\vec \alpha} \leq 2m-1,
      \\
     \norm{D^\alpha S_\lambda^{I,j,k} g_j}_{\LL_p(\R_+^{n+1};E)}
      &\leq C \abs{\lambda}^{\frac{\abs{\vec \alpha} - k}{2m}} \norm{\mathcal{P}_{j,k} g_j}_{\LL_p(\R_+^{n+1};E)},
      &&0 \leq \abs{\vec \alpha} \leq k-1 \leq m_j - 1,
      \\
     \norm{D^\alpha S_\lambda^{II,j,k} g_j}_{\LL_p(\R_+^{n+1};E)}
      &\leq C \abs{\lambda}^{\frac{\abs{\vec \alpha} - k-1}{2m}} \norm{\tfrac{\partial}{\partial y} \mathcal{P}_{j,k} g_j}_{\LL_p(\R_+^{n+1};E)},
      &&0 \leq \abs{\vec \alpha} \leq k \leq m_j.
    \end{align*}
  \end{proposition}
  
  \begin{proof}
   Analogously to the proof of \cite[Proposition 6.8]{DeHiPr03}, one sets
    \begin{align*}
     v^{j,k}(\vec z',y)
      &= - \int_0^\infty \int_{\R^n} \frac{\partial}{\partial \tilde y} K_\lambda^{j,k}(\vec z' - \tilde {\vec z}', y + \tilde y) (- \Delta_{\R^n} + \abs{\lambda}^{1/m})^{\frac{2m-k}{2m}} g_{j,k}(\vec z',\tilde y) \dd \tilde {\vec z}' \dd y
       \\ &\quad
       - \int_0^\infty \int_{\R^n} K_\lambda^{j,k}(\vec z' - \tilde {\vec z}', y + \tilde y) \frac{\partial}{\partial \tilde y} (- \Delta_{\R^n} + \abs{\lambda}^{1/m})^{\frac{2m-k}{2m}} g_{j,k}(\tilde {\vec z}', \tilde y) \dd \tilde {\vec z}' \tilde y
       \\
      &= S_\lambda^{I,j,k} g_{j,k}(\vec z',y) + S_\lambda^{II,j,k} g_{j,k}(\vec z',y).
    \end{align*}
   As in the proof of \cite[Proposition 6.8]{DeHiPr03}, by using Proposition \ref{DHP03:Prop_6.5}, we then find that
    \begin{align*}
     \norm{D^{\vec \alpha} \frac{\partial}{\partial \tilde y} K_\lambda^{j,k}(\cdot,\tilde y)}_{\LL_1(\R^n;\B(E))}
      &\leq M \abs{\lambda}^{\frac{n-2m+\abs{\vec \alpha}}{2m}} \int_{\R^n} p_{2m,\abs{\vec \alpha}}^{n+1}(c \abs{\lambda}^{\nicefrac{1}{2m}} \tilde y) \dd \vec z'
      \\
      &\leq C \abs{\lambda}^{\frac{n-2m+\abs{\vec \alpha}}{2m}} \int_0^\infty p_{2m,\abs{\vec \alpha}}^{n+1}(c \abs{\lambda}^{\nicefrac{1}{2m}} (r + \tilde y)) r^{n-1} \dd r
      \\
      &= n! C \abs{\lambda}^{\frac{\abs{\vec \alpha} - 2m}{2m}} p_{2m+n,\abs{\vec \alpha}}^{n+1}(c \abs{\lambda}^{\nicefrac{1}{2m}} \tilde y),
      \quad
      \tilde y > 0.
    \end{align*}
   Then, the first estimate claimed follows by Minkowski's inequality and \cite[Lemma 6.7]{DeHiPr03}.
  Similarly, for $S_\lambda^{II,j,k}$ we write
   \[
    \mathcal{F} K_\lambda^{II}(\vec \xi',y)
     = (\mathcal{F} K_\lambda)(\vec \xi',y) \left( \abs{\vec \xi'}^2 + \abs{\lambda}^{\nicefrac{1}{m}} \right)^{\nicefrac{1}{2}},
     \quad
     K_{\lambda}^{II,j,k}
     = \mathcal{P}_{j,k} (K_\lambda^{II})_{1,j}
   \]
  for $j = 1, \ldots, m$ and $k = 0, 1, \ldots, m_j$, to obtain the representation
   \[
    S_\lambda^{II,j,k}(\vec z',y)
     = \int_0^\infty \int_{\R^n} K_\lambda^{II,j,k}(\vec z' - \tilde {\vec z}', y + \tilde y) \frac{\partial}{\partial \tilde y} (- \Delta_{\R^n} + \abs{\lambda}^{\nicefrac{1}{m}})^{\frac{2m-k-1}{2m}} g_{j,k}(\tilde {\vec z}', \tilde y) \dd \tilde {\vec z}' \dd \tilde y.
   \]
  Proposition \ref{DHP03:Prop_6.5} then gives the estimate
   \[
    \abs{D^{\vec \beta} K_\lambda^{II}(\vec z',y)}
     \leq M \abs{\lambda}^{\frac{n-2m+1+\abs{\vec \beta}}{2m}} p_{2m,\abs{\vec \beta}}^{n+1} (c \abs{\lambda}^{\nicefrac{1}{2}m} (\abs{\vec z'} + \tilde y)),
     \quad
     0 \leq \abs{\vec \beta} \leq 2m,
   \]
  so that
   \[
    \norm{D^{\vec \beta} K_\lambda^{II,j,k}(\cdot,y)}_{\LL_1(\R^n;E)}
     \leq M \abs{\lambda}^{\frac{\abs{\vec \beta} + 1 - 2m}{2m}} p_{2m+n,\abs{\vec \beta}}^{n+1}(c \abs{\lambda}^{\nicefrac{1}{2}m} \tilde y),
     \quad \tilde y > 0, \, \lambda \in \Sigma_{\pi-\phi}.
   \]
  Then, by \cite[Lemma 6.7]{DeHiPr03},
   \[
    \norm{\lambda^{1-\frac{\abs{\vec \alpha}}{2m}} D^{\vec \alpha} S_\lambda^{II,j,k} g_{j,k}}_{\LL_p(\R^n;E)}
    \leq C \norm{(- \Delta_{\R^n} + \abs{\lambda}^{1/m})^{\frac{2m-k-1}{2m}} \frac{\partial}{\partial y} g_{j,k}}_{\LL_p(\R^n;E)},
   \]
  so that the second estimate has been established.
  By slight modification of the arguments, the remaining estimates can be proved as well.
  The representation and regularity of the solution to \eqref{DHP03:6.27} can then be demonstrated by approximation just as at the end of the proof of \cite[Proposition 6.8]{DeHiPr03}:
  Given $g_{j,k} \in \WW_p^{2m-k}(\R_+^{n+1}; \ran (\mathcal{P}_{j,k}))$, let $g_{j,k}^\nu \in \WW_p^{2m-k+1}(\R_+^{n+1}; \ran(\mathcal{P}_{j,k}))$ with $g_{j,k}^\nu \rightarrow g_{j,k} \in \WW_p^{2m-k}(\R_+^{n+1}; \ran(\mathcal{P}_{j,k}))$ as $\nu \rightarrow \infty$.
  By the estimates above, $v^{j,k,\nu} := (S_\lambda^{I,j,k} + S_\lambda^{II,j,k}) g_{j,k}^\nu \in \WW_p^{2m}(\R_+;\ran(\mathcal{P}_{j,k}))$ with $(\lambda J + A)v^{\nu} = \left( \begin{smallmatrix} 0 \\ \vec g^\nu \end{smallmatrix} \right)$, i.e.\ $v^\nu \in \dom(A)$, and further $v^{j,k,\nu} \rightarrow v^{j,k}$ in $\WW_p^{2m-1}(\R_+^{n+1}; E)$.
  As $A$ is a closed operator, $v \in \dom(A) \cap \WW_p^{2m-1}(\R_+^{n+1}; E)$ with $(\lambda J + A)  v = \left( \begin{smallmatrix} 0 \\ \vec g \end{smallmatrix} \right)$.
  \end{proof}
  
 Similarly, the results for the boundary value problem
  \begin{equation}
   \lambda J v + A v
    = \left( \begin{array}{c} f \\ \vec 0 \end{array} \right)
   \label{DHP03:6.30}
  \end{equation}
 transfer to the situation considered here.
 
 \begin{proposition}
 \label{DHP03:Prop_6.9}
  Let the differential operators $\mathcal{A}(D)$ and $\mathcal{B}_j(D)$ be given as above, $j = 1, \ldots, m$.
  Let $\lambda \in \Sigma_{\pi-\phi}$, and $f \in \LL_p(\R_+^{n+1};E)$.
  Then there exists a unique solution $w \in \WW_p^{2m-1}(\R_+^{n+1};E) \cap \dom(A)$ of \eqref{DHP03:6.30}, which is given by $w = w_0 + \sum_{j=1}^m \sum_{k=0}^{m_j} v^{j,k}$, for $v^{j,k}$ as in the preceeding proposition and $w_0 = P(\lambda + A_{\R^{n+1}})^{-1} E_0 f$.
  Moreover, there is $C > 0$ such that
   \begin{equation}
    \norm{\lambda^{1-\frac{\abs{\vec \alpha}}{2m}} D^{\vec \alpha} w}_{\LL_p(\R_+^{n+1};E)}
     \leq C \norm{f}_{\LL_p(\R_+^{n+1};E)},
     \quad
     0 \leq \abs{\vec \alpha} \leq 2m-1.
    \label{DHP03:6.31}
   \end{equation}
 \end{proposition}
 \begin{proof}
  The proof can be transferred almost \emph{expressis verbis} from the proof of \cite[Proposition 6.9]{DeHiPr03}.
 \end{proof} 
 
 These results can be summarised as follows, cf.\ \cite[Theorem 6.10]{DeHiPr03}.
 
 \begin{theorem}
 \label{DHP03:Thm_6.10}
  Let $\mathcal{A}(D)$ be a parameter elliptic operator of order $2m$ with  angle of ellipticity $\phi_\mathcal{A} \in [0,\pi)$.
  Let $\phi > \phi_\mathcal{A}$.
  For $j = 1, \ldots, m$ let $\mathcal{B}_j(D)$ be a boundary operator of order $m_j < 2m$ and of the form above.
  Assume that the Lopatinskii--Shapiro condition ${\bf (LS)}$ holds true.
  Let $p \in (1, \infty)$, $E$ be a Banach space, $f \in \LL_p(\R_+^{n+1};E)$ and $g_j \in \sum_{k=0}^{m_j} \mathcal{P}_{j,k} \WW_p^{2m-k}(\R_+^{n+1};E)$ for $j = 1, \ldots, m$. Let $\lambda \in \Sigma_{\pi-\phi}$ and let $A$ be defined as above.
  Then there is a unique solution $u \in \WW_p^{2m-1}(\R_+^{n+1};E) \cap \dom(A)$ of
   \[
    \lambda J u + A u
     = \left( \begin{array}{c} f \\ \vec g \end{array} \right).
   \]
  Moreover, $u$ is given by
   \[
    u
     = P(\lambda + A_{\R^{n+1}})^{-1} E_0 f + \sum_{j=1}^m \sum_{k=0}^{m_j} \left( R_\lambda^{j,k} f + S_\lambda^{j,k} g_j \right)
   \]
  and there is a constant $C > 0$ such that for $0 \leq \abs{\vec \alpha} \leq 2m-1$ and $\lambda \in \Sigma_{\pi-\phi}$ one has
   \begin{align}
    \norm{\lambda^{1-\frac{\norm{\vec \alpha}}{2m}} D^{\vec \alpha} u}_{\LL_p(\R_+^{n+1};E)}
     &\leq C \left( \norm{f}_{\LL_p(\R_+^{n+1};E)} + \sum_{j=1}^m \sum_{k=0}^{m_j} \norm{(- \Delta + \abs{\lambda}^{1/m})^{\frac{2m-k}{2}} \mathcal{P}_{j,k} g_j}_{\LL_p(\R_+^{n+1};E)} 
     \right. \label{DHP03:6.33} \\ &\qquad \left.
     + \sum_{j=1}^m \sum_{k=0}^{m_j} \norm{(-\Delta + \abs{\lambda}^{1/m})^{\frac{2m-k-1}{2}} D_y \mathcal{P}_{j,k} g_j}_{\LL_p^{n+1};E)} \right).
     \nonumber
   \end{align}
 \end{theorem}
 
 \subsection{$\mathcal{H}^\infty$-calculus}
 
 In this subsection, the results on the $\mathcal{H}^\infty$-calculus of \cite[Section 7.1]{DeHiPr03} are transferred to the situation considered here.
 Combining Theorem \ref{DHP03:Thm_6.10} with \cite[Lemma 7.1]{DeHiPr03}, one first gets the following theorem.
 
 \begin{theorem}
 \label{DHP03:Thm_7.3}
 If, in the situation of Theorem \ref{DHP03:Thm_6.10}, the Banach space $E$ is of class $\mathcal{HT}$, then $u \in \WW_p^{2m}(\R_+^{n+1};E)$ and \eqref{DHP03:6.33} holds true for all $\vec \alpha$ with $\abs{\vec \alpha} = 2m$ as well.
 Moreover, in this case
  \begin{align}
   \norm{D^{\vec \alpha} S_\lambda^{I,j,k} g_j}_{\LL_p(\R_+^{n+1};E)}
    &\leq C \abs{\lambda}^{\frac{\abs{\vec \alpha} - k}{2m}} \norm{\mathcal{P}_{j,k} g_j}_{\LL_p(\R_+^{n+1};E)},
    &&0 \leq \abs{\vec \alpha} \leq k \leq m_j,
    \label{DHP03:7.6}
    \\
   \norm{D^{\vec \alpha} S_\lambda^{II,j,k} g_j}_{\LL_p(\R_+^{n+1};E)}
    &\leq C \abs{\lambda}^{\frac{\abs{\vec \alpha} - k - 1}{2m}} \norm{\tfrac{\partial}{\partial y} \mathcal{P}_{j,k} g_j}_{\LL_p(\R_+^{n+1};E)},
    &&0 \leq \abs{\vec \alpha} \leq k + 1 \leq m_j + 1
    \label{DHP03:7.7}
 \end{align}
 and
 \[
   \dom(A_B)
    = \{u \in \WW_p^{2m}(\R_+^{n+1};E): \, \mathcal{B}_j(D) u = 0 \quad \text{for all } j = 1, \ldots, m\}
 \]
 is the domain of the closure $A_B = \overline{A_B^0}$ of the $\LL_p$-realisation of the boundary value problem \eqref{DHP03:6.30}:
  \begin{align}
   A_B^0 u
    &= \mathcal{A}(D) u,
    \label{DHP03:6.34}
    \\
   \dom(A_B^0)
    &= \{ u \in \WW_p^{2m}(\R_+^{n+1};E): \, \mathcal{B}_j(D) u = 0 \quad \text{for all } j = 1, \ldots, m \}.
  \end{align}
 \end{theorem}
 
 Using the kernel estimates of Proposition \ref{DHP03:Prop_6.6}, one may then prove that $A_B$ admits a bounded $\mathcal{H}^\infty$-calculus, see the following analogue to \cite[Theorem 7.4]{DeHiPr03}:
 
 \begin{theorem}
  \label{DHP03:Thm_7.4}
  Let $E$ be a Banach space of class $\mathcal{HT}$ and let $A_B: \dom(A_B) \rightarrow \LL_p(\R_+^{n+1};E)$ be defined as in Theorem \ref{DHP03:Thm_7.3}. Then $A_B \in \mathcal{H}^\infty(\LL_p(\R_+^{n+1};E))$ with $\mathcal{H}^\infty$-angle $\phi_{AB}^\infty \leq \phi_\mathcal{A}$.
 \end{theorem}
 
 The proof can be executed similar to the proof of \cite[Theorem 7.4]{DeHiPr03}, using \cite[Lemma 7.1]{DeHiPr03} and an adjusted version of \cite[Remark 7.2c]{DeHiPr03}.
 
 \begin{corollary}
  \label{DHP03:Cor_7.5}
  Let $E$ be a Banach space of class $\mathcal{HT}$.
  Then there is $C > 0$ such that
   \[
    \mathcal{R} \{\lambda^{1-\frac{\abs{\vec \alpha}}{2m}} D^{\vec \alpha} (\lambda + A_B)^{-1}: \, \lambda \in \Sigma_{\pi-\phi}, \, 0 \leq \abs{\vec \alpha} \leq 2m \}
     \leq C,
   \]
  where $\mathcal{R}(M)$ denotes the $\mathcal{R}$-bound of a $\mathcal{R}$-bounded set $M$.
 \end{corollary}

 \subsection{$\mathcal{R}$-bounds for solution operators}
 
 By Corollary \ref{DHP03:Cor_7.5}, the solution map for the boundary value problem \eqref{DHP03:6.30} with $\vec g = \vec 0$ admits $\mathcal{R}$-bounds. Next, one has to consider the boundary value problem \eqref{DHP03:6.30} for $f = 0$ and general $\vec g \neq 0$. By Proposition \ref{DHP03:Prop_6.8}, its solution may be expressed as
  \[
   v = \sum_{j=1}^m \sum_{k=0}^{m_j} S_\lambda^{j,k} \mathcal{P}_{j,k} g_j.
  \]
 Similar to \cite[Section 7.2]{DeHiPr03} one now defines for $\lambda \in \Sigma_{\pi - \phi}$, $j = 1, \ldots, m$ and $k = 0, \ldots, m_j$ the operators
  \begin{align}
   U_\lambda^{j,k}
    &= S_\lambda^{I,j,k} ((-\Delta)^{\frac{2m-k}{2}} + \abs{\lambda}^{\frac{2m-k}{2m}})^{-1},
    \label{DHP03:7.9}
    \\
   V_\lambda^{j,k}
    &= S_\lambda^{II,j,k}((-\Delta)^{\frac{2m-k-1}{2}} + \abs{\lambda}^{\frac{2m-k-1}{2m}})^{-1}.
    \label{DHP03:7.10}
  \end{align}

 These allow for the following $\mathcal{R}$-bounds, cf.\ \cite[Proposition 7.6]{DeHiPr03}.
 \begin{proposition}
 \label{DHP03:Prop_7.6}
 For each $j = 1, \ldots, m$ and $k = 0, 1, \ldots, m_j$, the sets
  \begin{align*}
   \{\lambda^{1-\frac{\abs{\vec \alpha}}{2m}} D^{\vec \alpha} U_\lambda^{j,k}: \, \lambda \in \Sigma_{\pi-\phi}, \, \abs{\vec \alpha} \leq 2m \}
    &\subseteq \B(\LL_p(\R_+^{n+1};E)),
    \\
   \{\lambda^{1-\frac{\abs{\vec \alpha}}{2m}} D^{\vec \alpha} V_\lambda^{j,k}: \, \lambda \in \Sigma_{\pi-\phi}, \, \abs{\vec \alpha} \leq 2m \}
    &\subseteq \B(\LL_p(\R_+^{n+1};E))
  \end{align*}
 are $\mathcal{R}$-bounded.
 \end{proposition} 
 \begin{proof}
  The operators $D^{\vec \alpha} U_\lambda^{j,k}$ and $D^{\vec \alpha} V_\lambda^{j,k}$ may be represented as integral operators on $\LL_p(\R_+^{n+1};E)$ with kernel $D^{\vec \alpha} \frac{\partial}{\partial y} K_\lambda^{U^{j,k}}$, where
   \begin{align*}
    \mathcal{F} K_\lambda^{U^{j,k}}(\vec \xi',\lambda)
     &= \mathcal{F} K_\lambda^{j,k}(\vec \xi',y) (\abs{\vec \xi'}^2 + \abs{\lambda}^{1/m})^{\frac{2m-k}{2m}} (\abs{\vec \xi'}^{2m-k} + \abs{\lambda}^{\frac{2m-k}{2m}})^{-1},
     \\
    \mathcal{F} K_\lambda^{V^{j,k}}(\vec \xi',y)
     &= \mathcal{F} K_\lambda^{II,j,k}(\vec \xi',y) (\abs{\vec \xi'}^2 + \abs{\lambda}^{1/m})^{\frac{2m-k-1}{2}} (\abs{\vec \xi'}^{2m-k} + \abs{\lambda}^{\frac{2m-k-1}{2m}})^{-1},
     \quad
     1 \leq j \leq m, \, 0 \leq k \leq m_j.
   \end{align*}
  By \cite[Proposition 4.1 and Lemma 7.1]{DeHiPr03}, it suffices to prove that there is $C > 0$ such that for all $\vec z' \in \R^n$  and $y, \tilde y > 0$
   \begin{align}
    \mathcal{R} \{ \lambda^{1- \frac{\abs{\vec \alpha}}{2m}} D^{\vec \alpha} \frac{\partial}{\partial \tilde y} K_\lambda^{U^{j,k}}(\vec z', y + \tilde y): \, \lambda \in \Sigma_{\pi-\phi}, \, \abs{\vec \alpha} \leq 2m \}
     &\leq \frac{C}{(\abs{\vec z'} + y + \tilde y)^{n+1}},
     \label{DHP03:7.11}
     \\
    \mathcal{R} \{ \lambda^{1- \frac{\abs{\vec \alpha}}{2m}} D^{\vec \alpha} \frac{\partial}{\partial \tilde y} K_\lambda^{V^{j,k}}(\vec z', y + \tilde y): \, \lambda \in \Sigma_{\pi-\phi}, \, \abs{\vec \alpha} \leq 2m \}
     &\leq \frac{C}{(\abs{\vec z'} + y + \tilde y)^{n+1}}.
     \label{DHP03:7.12}
   \end{align}
  This can be proved as follows, analogously to \cite[Proposition 7.7]{DeHiPr03}.
  For $\vec z' \in \R^n$, $y > 0$, $\lambda \in \Sigma_{\pi-\phi}$, $0 \leq \abs{\vec \alpha} \leq 2m$ and $\abs{\vec \beta} + \gamma = \abs{\vec \alpha}$, one has
   \begin{align}
    &D^{\vec \alpha} \frac{\partial}{\partial y} K_\lambda^{U^{j,k}}(\vec z',y)
     \label{DHP03:7.13}
     \\
     &= \frac{1}{(2\pi)^n} \int_{\R^n} \ee^{\ii \vec z' \cdot \vec \xi} \frac{\ee^{\ii \rho \bb A_0(\vec b,\sigma) y}}{\rho^{2m}}(\vec \xi')^{\vec \beta} (A_0(b,\sigma) \rho)^\gamma \ii \bb A_0(\vec b,\sigma) M(b,\sigma) \rho^{2m-k} J_{\vec \xi',\lambda}^{2m-k} \dd \vec \xi',
     \nonumber
   \end{align}
  where $J_{\vec \xi',\lambda}^{2m-k} = (\abs{\vec \xi'}^{2m-k} + \abs{\lambda}^{\frac{2m-k}{2m}})^{-1}$.
  As in the proof of Proposition \ref{DHP03:Prop_6.5}, one chooses a rotation $\bb Q \in
\R^{n \times n}$ such that $\bb Q \vec z' = (\abs{\vec z'}, 0, \ldots, 0)^\mathsf{T}$ and writes $\bb Q \vec \xi' = (a, r \vec \varphi)$ for some $a \in \R$, $r > 0$ and $\vec \varphi \in \S^{n-2}$. Then
 \begin{align*}
  &D^{\vec \alpha} \frac{\partial}{\partial y} K_\lambda^{U^{j,k}}(\vec z',y)
   \\
   &= \frac{1}{(2\pi)^n} \int_{\S^{n-2}} \int_0^\infty r^{n-2} \int_{-\infty}^\infty \ee^{\ii \abs{\vec z'} a} \frac{\ee^{\ii \rho \bb A_0(\vec b, \sigma) y}}{\rho^{2m-1}} (\vec \xi')^{\vec \beta} (\bb A_0(\vec b,\sigma) \rho)^\gamma \ii \bb A_0(\vec b,\sigma) \bb M(\vec b,\sigma) \rho^{2m-k} J_{\vec \xi',\lambda}^{2m-k} \dd a \dd r \dd \vec \varphi
   \\
   &= \frac{1}{(2\pi)^n} \int_{\S^{n-2}} \int_0^\infty r^{n-2} \int_{\Gamma^-} \ee^{\ii \abs{\vec z'} a} \frac{\ee^{\ii \rho \bb A_0(\vec b, \sigma) y}}{\rho^{2m-1}} (\vec \xi')^{\vec \beta} (\bb A_0(\vec b,\sigma) \rho)^\gamma \ii \bb A_0(\vec b,\sigma) \bb M(\vec b,\sigma) \rho^{2m-k} J_{\vec \xi',\lambda}^{2m-k} \dd a \dd r \dd \vec \varphi,
 \end{align*}
 where as in the proof of Proposition \ref{DHP03:Prop_6.5} $\gamma^-(s) = s + \ii \varepsilon (r + \abs{s} + \abs{\lambda}^{\nicefrac{1}{2}m})$, $s \in \R$ is a parametrisation of $\Gamma^-$, and Cauchy's integral theorem as well as holomorphy of $\bb A_0$ and $M$ have been used.
 Setting $c = \nicefrac{c_1}{2}$ from \eqref{DHP03:6.11}, as in Proposition \ref{DHP03:Prop_6.5} one obtains
  \begin{align*}
   &\abs{\lambda}^{1-\frac{\abs{\vec \alpha}}{2m}} D^{\vec \alpha} \frac{\partial}{\partial y} K_\lambda^{U^{j,k}}(\vec z',y)(\abs{\vec z'} + y)^{n+1}
    \\
    &= \frac{1}{(2\pi)^n} \int_{\S^{n-2}} \int_0^\infty \int_0^\infty F(r,\tau,\vec \varphi, \lambda, \vec z', y) r^{n-2} (\abs{\vec z'} + y)^n \ee^{-\frac{1}{2} (\varepsilon \abs{\vec z'} + cy)(r + \tau)} \dd r \dd \tau \dd \vec \varphi,
  \end{align*}
 where
  \[
   F(r,\tau,\vec \varphi, \lambda, \vec z',y)
    = G(r,\tau, \vec \varphi, \lambda, y) H(\vec z', y, r, \tau, \lambda)
  \]
 with
  \begin{align*}
   G(r,\tau,\vec \varphi, \lambda, y)
    &= \ee^{\ii \rho \bb A_0(\vec b, \sigma) y} \ee^{cy(r + \tau + \abs{\lambda}^{\nicefrac{1}{2}m})} \bb A_0(\vec b,\sigma)^\gamma \bb A_0(\vec b, \sigma) \bb M(\vec b,\sigma),
    \\
   H(\vec z', y, r, \tau, \lambda)
    &= (\abs{\vec z'} + y) \ee^{-(\varepsilon \abs{\vec z'} + c y) \abs{\lambda}^{\nicefrac{1}{2}m}}  \ee^{- \frac{1}{2} (\varepsilon \abs{\vec z'} + cy)(r + \tau)} {\vec \xi'}^{\vec \beta} \rho^\gamma \frac{\lambda^{1-\frac{\abs{\vec \alpha}}{2m}}}{\rho^{2m-1}} J_{\vec \xi',\lambda}^{2m-k}
  \end{align*}
 and where $\rho, \vec b, \sigma$ depend on $r, \tau, \lambda$ and $\abs{\vec \alpha} = \abs{\vec \beta} + \gamma$.
  Setting $M = \S^{n-2} \times \R_+ \times \R_+$, this integral equals $\int_M F(r,\tau,\vec \varphi, \lambda, \vec z', y) \dd \mu(r,\tau, \vec \varphi)$, where $\dd \mu(r,\tau,\vec \varphi) = r^{n-2} (\abs{\vec z'} + y)^n \ee^{-\frac{1}{2}(\varepsilon \abs{\vec z'} + cy)(r + \tau)} \dd r \dd \tau \dd \sigma(\vec \varphi)$. Here, $\mu$ is a finite measure on $M$ and such that $\int_M \dd \mu(r,\tau, \vec \varphi) \leq C$ for some $C > 0$ independent of $\vec z' \in \R^n$ and $y > 0$.
  By \cite[Proposition 3.8]{DeHiPr03} it follows that for $\abs{\vec \beta} + \gamma = \abs{\vec \alpha}$
   \begin{align*}
    &(\abs{\vec z'} + y)^{n+1} \mathcal{R} \{ \lambda^{1-\frac{\abs{\vec \alpha}}{2m}} D^{\vec \alpha} \frac{\partial}{\partial y} K_\lambda^{U^{j,k}}(\vec z',y): \, \lambda \in \Sigma_{\pi-\phi}, \, \abs{\vec \alpha} \leq 2m \}
     \\
     &\leq \mathcal{R} \{ (\abs{\vec z'} + y)^{n+1} \lambda^{1-\frac{\abs{\vec \alpha}}{2m}} D^{\vec \alpha} \frac{\partial}{\partial y} K_\lambda^{U^{j,k}}(\vec z',y): \, \lambda \in \Sigma_{\pi-\phi}, \, \abs{\vec \alpha} \leq 2m, \, \vec z' \in \R^n, \, y > 0 \}
     \\
     &= \mathcal{R} \big\{ \int_M F(r,\tau,\vec \varphi, \lambda, \vec z',y) \dd \mu(r,\tau, \vec \varphi): \, \lambda \geq 0, \, \vec z' \in \R^n, \, \vec \varphi \in \S^{n-2} \big\}.
   \end{align*}
  Since $F = G H$ and $H$ is a bounded scalar function, it is enough to prove $\mathcal{R}$-boundedness of the sets
   \begin{align*}
    M_1
     &= \{ \ee^{\ii \rho \bb A_0(\vec b, \sigma) y}: \lambda \in \Sigma_{\pi-\phi}, \, y,r,\tau > 0 \}],
     \\
    M_2
     &= \{ \bb A_0(\vec b,\sigma): \lambda \in \Sigma_{\pi-\phi}, \, y,r,\tau > 0 \},
     \\
    M_3
     &= \{ \bb M(\vec b,\sigma): \lambda \in \Sigma_{\pi-\phi}, \, y,r,\tau > 0 \}.
   \end{align*}
  Since $M_2$ and $M_3$ are images under holomorphic maps of a compact set, they are $\mathcal{R}$-bounded by \cite[Proposition 3.10]{DeHiPr03}. For $M_1$, setting $u := y (\abs{\lambda}^{\nicefrac{1}{2}m} + r + \tau)$, by Cauchy's integral theorem and for $\Gamma^K$ a compact curve in $\C_0^-$ surrounding $\sigma(\ii \bb A_0(\vec b,\sigma) + c)$ one finds
   \begin{align*}
    \ee^{\ii \rho \bb A_0(\vec b, \sigma) y}  \ee^{(\ii \bb A_0(\vec b,\sigma) + c) u}
     &= \frac{1}{2 \pi \ii} \int_{\Gamma_K} \ee^{\omega u} (\omega - (\ii \bb A_0(\vec b,\sigma + c))^{-1} \dd \omega,
   \end{align*}
  for arbitrary $r, \tau > 0$ and $\lambda \in \Sigma_{\pi-\phi}$.
  By \cite[Propositions 3.8 and 3.10]{DeHiPr03}, $\mathcal{R}$-boundedness of the set $M_1$ follows. The assertion for $V^{j,k}$ can be established in exactly the same way, cf.\ the proof of \cite[Proposition 7.7]{DeHiPr03}.
 \end{proof}
 
 \subsection{Perturbations}
 
 As also situations with spatial dependent parameters of utmost importance, in particular for the localisation procedure by which results on the half-space are transferred to general, but sufficiently regular domains, we consider spatial dependent perturbations of the principle part next, i.e.\ $\mathcal{A}$ and $\mathcal{B}_j$ are assumed to be of the form
  \[
   \mathcal{A}(\vec z, D)
    = \sum_{\abs{\vec \alpha}} a_{\vec \alpha}(\vec z) D^{\vec \alpha}
  \]
 with $a_{\vec \alpha} \in \mathrm{BUC}(\R_+^{n+1};\B(E))$, and
  \[
   \mathcal{B}_j(\vec z,D)
    = \sum_{k=0}^{m_j} \sum_{\abs{\vec \beta} = k} b_{j,\vec \beta}(\vec z) D^{\vec \beta} \mathcal{P}_{j,k},
    \quad
    j = 1, \ldots, m
  \]
 with regularity of the coefficients $b_{j,\vec \beta} \in \mathrm{BUC}^{2m-k}(\R_+^{n+1};E)$ for $\abs{\vec \beta} = k$.
 Throughout this subsection, we assume that $\phi > \phi_\mathcal{A}$, the Banach space $E$ is of class $\mathcal{HT}$ and the Lopatinskii--Shapiro condition is satisfied for every boundary point $\vec z \in \partial \R_+^{n+1}$.
 Now fix $\vec z_0 \in \R_+^{n+1}$ and, as an intermediate step, assume that, for given $\varepsilon > 0$, the coefficients $a_{\vec \alpha}$ and $b_{j,\vec \beta}$ are of uniformly small oscillation, in the sense that
  \begin{align}
   \sup_{\vec z \in \R_+^{n+1}} \sum_{\abs{\vec \alpha} = 2m} \norm{a_{\vec \alpha}(\vec z) - a_{\vec \alpha}(\vec z_0)}
    &< \varepsilon,
    \label{DHP03:7.14}
    \\
   \sup_{\vec z \in \R_+^{n+1}} \sum_{\abs{\vec \beta} = k} \norm{b_{j,\vec \beta}(\vec z) - b_{j,\vec \beta}(\vec z_0)}
    &< \varepsilon,
    \quad
    j = 1, \ldots, m, \, k = 0, \ldots, m_j.
    \label{DHP03:7.15}
  \end{align}
 We may then write $\mathcal{A}$ and $\mathcal{B}_j$ as small perturbations from the constant coefficient case:
  \begin{align*}
   \mathcal{A}(\vec z,D)
    &= \mathcal{A}(\vec z_0,D) + \mathcal{A}^\mathrm{sm}(\vec z,D),
    \\
   \mathcal{B}_j(\vec z,D)
    &= \mathcal{B}_j(\vec z_0,D) + \mathcal{B}_j^\mathrm{sm}(\vec z,D).
  \end{align*}
 In this case, for $\lambda \in \Sigma_{\pi-\phi}$ with $\phi > \phi_\mathcal{A}$ and $f \in \LL_p(\R_+^{n+1};E)$, the variable coefficient problem
  \begin{equation}
   \begin{cases}
    (\lambda + \mathcal{A}(\vec z,D)) u = f
    &\text{in } \R_+^{n+1},
    \\
    (\mathcal{B}_j(\vec z,D) u)(0) = 0
    &\text{on } \partial \R_+^{n+1}, \, j = 1, \ldots, m
   \end{cases}
   \label{DHP03:7.16}
  \end{equation}
  has a unique solution $\vec u \in \WW_p^{2m}(\R_+^{n+1};E)$ if and only if it is the unique solution $\vec u \in \WW_p^{2m}(\R_+^{n+1};E)$ to the problem
  \begin{equation}
   \begin{cases}
    (\lambda + \mathcal{A}(\vec z_0,D)) u = f - \mathcal{A}^\mathrm{sm}(\vec z, D) u
    &\text{in } \R_+^{n+1},
    \\
    \mathcal{B}_j(\vec z_0) u(\vec z,0) = - \mathcal{B}_j^\mathrm{sm}(\vec z,D) u(\vec z,0)
    &\text{on } \partial \R_+^{n+1}, \, j = 1, \ldots, m.
   \end{cases}
   \label{DHP03:7.17}
  \end{equation}
 A procedure as in \cite[Section 7.3]{DeHiPr03}, based on the contraction mapping principle, establishes that for sufficiently small $\varepsilon$ and $\lambda \in \Sigma_{\pi-\phi}$ with large enough modulus, the equations \eqref{DHP03:7.16} and \eqref{DHP03:7.17} admit a unique solution of the form
  \begin{equation}
   u
    = [I + (\lambda + A_{B^0}^0)^{-1} \mathcal{A}^\mathrm{sm} + \sum_{j=1}^m \sum_{k=0}^{m_j} S_\lambda^{j,k} \mathcal{B}_j^\mathrm{sm}]^{-1} (\lambda + A_{B^0}^0)^{-1} f
   \label{DHP03:7.24}
  \end{equation}
 where $A_{B^0}^0 := A(\vec z_0,D)|_{\ker (\mathcal{B}(\vec z_0,D))}$.
 We then set
  \begin{align}
   A_B u
    &= \mathcal{A}(\vec z,D) u,
    \label{DHP03:7.25}
    \\
   \dom(A_B)
    &= \{ u \in \WW_p^{2m}(\R_+^{n+1};E): \, \mathcal{B}_j(\vec z_0, D) u(\vec z',0) = 0, \, \vec z' \in \R^n, \, j = 1, \ldots, m  \}.
   \label{DHP03:7.26}
  \end{align}
 Thus, for $\lambda \in \Sigma_{\pi-\phi}$ large enough, $(\lambda + A_B)$ is invertible with
  \begin{equation}
   (\lambda + A_B)^{-1}
    = (\lambda + A_{B^0}^0)^{-1} - (\lambda + A_{B^0}^0)^{-1} \mathcal{A}^\mathrm{sm}(\lambda + A_B)^{-1} - \sum_{j=1}^m \sum_{k=0}^{m_j} S_\lambda^{j,k} \mathcal{P}_{j,k} \mathcal{B}_j^\mathrm{sm} (\lambda + A_B)^{-1}.
   \label{DHP03:7.27}
  \end{equation}
 
 Concerning $\mathcal{R}$-boundedness, one has the following result.
 
 \begin{lemma}
  There is $\lambda_0 > 0$ such that the set
   \[
    \{ \lambda^{1 - \frac{\abs{\vec \alpha}}{2m}} D^{\vec \alpha} (\lambda + A_B)^{-1}: \, \lambda \in \Sigma_{\pi-\phi}, \, \abs{\lambda} \geq \lambda_0, \, 0 \leq \abs{\vec \alpha} \leq 2m \} \subseteq \B(\LL_p(\R_+^{n+1};E))
   \]
  is $\mathcal{R}$-bounded.
 \end{lemma}
 \begin{proof}
  The proof is essentially the same as in \cite[Section 7.3]{DeHiPr03}. Instead of the operators $\mathcal{B}_j(D)$, consider the operators $\mathcal{B}_{j,k}(D) = \mathcal{B}_j(D) \mathcal{P}_{j,k}$ ($k = 0, 1, \ldots, m_j$) and use that $\mathcal{B}_j(D) = \sum_{k=0}^{m_j} \mathcal{B}_{j,k}(D)$.
 \end{proof}
 
 \subsection{Variable coefficients in a half space}

 To generalise the results of the previous subsection to more general spatial variant coefficients, one may use the localisation procedure as in \cite[Section 5.3]{DeHiPr03}. In comparison with \cite[Section 7.4]{DeHiPr03} one has to adjust the notion of the principal part to a definition suitable for the combined type boundary conditions considered here. Throughout this subsection, assume that there are projections $\mathcal{P}_{j,k}$, $j = 1, \ldots, m$, \, $k = 0, 1, \ldots, m_j < 2m$ such that $\mathcal{P}_{j,k} \mathcal{P}_{j,k'} = 0$ for $(j,k) \neq (j',k')$ and assume that $\mathcal{A}$ and $\mathcal{B}_j$ have the form
  \[
   \mathcal{A}(\vec z,D)
    = \sum_{\abs{\vec \alpha} \leq 2m} a_{\vec \alpha}(\vec z) D^{\vec \alpha},
    \quad
   \mathcal{B}_j(\vec z,D)
    = \sum_{k=0}^{m_j} \sum_{\abs{\vec \beta} \leq k} b_{j,k,\vec \beta}(\vec z) D^{\vec \beta} \mathcal{P}_{j,k}.
  \]
 In the following, we denote by $\mathcal{A}_\#$ and $\mathcal{B}_{j,\#}$ their (combined type) principle parts
  \[
   \mathcal{A}_\#(\vec z,D)
    = \sum_{\abs{\vec \alpha} = 2m} a_{\vec \alpha}(\vec z) D^{\vec \alpha},
    \quad
   \mathcal{B}_{j,\#}(\vec z,D)
    = \sum_{k=0}^{m_j} \sum_{\abs{\vec \beta} = k} b_{j,k,\vec \beta}(\vec z) D^{\vec \beta} \mathcal{P}_{j,k}.
  \]
 Using the notation
  \[
   \CC_l(\overline{\Omega})
    := \{f \in \CC(\overline{\Omega}): \, \lim_{\abs{\vec z} \rightarrow \infty} f(\vec z) =: f(\infty) \text{ exists, if } \Omega \text{ is unbounded} \}
  \]
 for the coefficients, we impose the following assumptions.
 
 \begin{assumption}
  \label{DHP03:Assmpt_Sec_7.4}
  \begin{enumerate}
   \item
    Smoothness assumptions:
     \begin{enumerate}
      \item
       $a_{\vec \alpha} \in \CC_l(\overline{\R_+^{n+1}};\B(E))$ for all $\abs{\vec \alpha} = 2m$;
      \item
       $a_{\vec \alpha} \in [\LL_\infty + \LL_{r_k}](\R_+^{n+1};\B(E))$ for $\abs{\vec \alpha} = k < 2m$ with $r_k \geq p$ and $2m - k > \tfrac{n}{r_k}$;
      \item
       $b_{j,\vec \beta} \in \CC^{2m-k}(\partial \R_+^{n+1}; \B(E))$ for each $j = 1, \ldots, m$, $0 \leq k \leq m_j$, $\abs{\vec \beta} = k$.
     \end{enumerate}
   \item
    ellipticity conditions: There is $\phi_\mathcal{A} \in [0,\pi)$ such that
     \begin{enumerate}
      \item
       The principal symbol $\mathcal{A}_\#(\vec z,\vec \xi) = \sum_{\abs{\vec  \alpha} = 2m} a_{\vec \alpha}(\vec z) \vec \xi^{\vec \alpha}$ is parameter elliptic with angle of ellipticity less than $\phi_\mathcal{A}$ for each $\vec z \in \overline{\R_+^{n+1}}$.
      \item
       Lopatinskii--Shapiro condition: For each $\vec z_0 \in \overline{\R_+^{n+1}}$, $\lambda \in \overline{\Sigma_{\pi-\phi}}$ and $\vec \xi' \in \R^n$ with $(\lambda, \vec \xi') \neq (0,\vec 0)$, the initial value problem
        \begin{align*}
         (\lambda + \mathcal{A}_\#(\vec z_0, \vec \xi', D_{n+1}))v(y)
          &= 0,
          &&y > 0,
          \\
         \mathcal{B}_{j,\#}(\vec z_0, \vec \xi', D_{n+1}) v(0)
          &= h_j,
          &&j = 1, \ldots, m
        \end{align*}
       has a unique solution $u \in \CC_0(\R_+;E)$, for each $(h_1, \ldots, h_m)^\mathsf{T} \in E^m$.    
     \end{enumerate}
  \end{enumerate}
 \end{assumption}
 
 The procedure above then gives the following result.
 
 \begin{theorem}
  \label{DHP03:Thm_7.11}
  Let $E$ be a Banach space of class $\mathcal{HT}$, $n,m \in \N$, $p \in (1, \infty)$ and assume that for some $\phi_\mathcal{A} \in [0,\pi)$ the boundary value problem $(\mathcal{A}, (\mathcal{B}_j)_{j=1}^m)$ satisfies Assumption \ref{DHP03:Assmpt_Sec_7.4}.
  Let $A_B$ the $\LL_p(\R_+^{n+1};E)$-realisation of the homogeneous boundary value problem with domain
   \[
    \dom(A_B)
     = \{u \in \WW_p^{2m}(\R_+^{n+1};E): \, \mathcal{B}_j(\vec z,D) u = 0 \, \text{on } \partial \R_+^{n+1}, \, j = 1, \ldots, m \}.
   \]
  Then, for each $\phi > \phi_\mathcal{A}$, there is $\mu_\phi \geq 0$ such that $\mu_\phi + A_B$ is $\mathcal{R}$-sectorial with $\phi_{\mu_\phi + A_B} \leq \phi$.
  In particular, if $\phi_\mathcal{A} < \tfrac{\pi}{2}$, then the parabolic initial-boundary value problem
   \begin{align*}
    \partial_t u + A_B u + \mu_\phi u
     &= f,
     &&t > 0,
     \\
    u(0)
     &= 0
   \end{align*}
  has the property of maximal regularity in $\LL_q(\R_+;\LL_p(\R_+^{n+1};E))$ for each $q \in (1, \infty)$.
 \end{theorem}  
 
 \subsection{Localisation techniques for domains}
 
 In this subsection the statement of the maximal regularity theorem for domains will be formulated.
 Let $E$ be a Banach space of class $\mathcal{HT}$ and $m, n, m_1, \ldots, m_m$ be natural numbers with $m_j < 2m$.
 Moreover, let $\mathcal{P}_{j,k}$, $j = 1, \ldots, m$, $k = 0, \ldots, m_j$, be projections such that $\mathcal{P}_{j,k} \mathcal{P}_{j,k'} = 0$ for $(j,k) \neq (j',k')$.
 Let $\Omega \subseteq \R^{n+1}$ be a domain with $\CC^{2m}$-boundary, i.e.\ for each $\vec z_0 \in \partial \Omega$ local coordinates exist which are obtained from the original one by rotating and shifting, and such that the positive $x_{n+1}$-axis corresponds to the direction of the outer normal vector $\vec n$ to $\Omega$ at $\vec z_0$. The local coordinates may be chosen such that they depend continuously on  $\vec z_0 \in \partial \Omega$.
 Consider the differential operators 
  \[
   \mathcal{A}(\vec z,D)
    = \sum_{\abs{\vec \alpha} \leq 2m} a_{\vec \alpha}(\vec z) D^{\vec \alpha},
    \quad
   \mathcal{B}_j(\vec z,D)
    = \sum_{k=0}^{m_j} \sum_{\abs{\vec \beta} \leq k} b_{j,\vec \beta}(\vec z) D^{\vec \beta} \mathcal{P}_{j,k}
  \]
 with variable $\B(E)$-valued coefficients which satisfy the following regularity and ellipticity conditions.
 
 \begin{assumption}
  \label{DHP03:Assmpt_Sec_8.1}
  \begin{enumerate}
   \item
    Smoothness conditions:
     \begin{enumerate}
      \item
       $a_{\vec \alpha} \in \CC_l(\overline{\Omega};\B(E))$ for each $\abs{\vec \alpha} = 2m$;
      \item
       $a_{\vec \alpha} \in [\LL_\infty + \LL_{r_k}](\Omega, \B(E))$ for each $\abs{\vec \alpha} = k < 2m$ with $r_k \geq p$ and $2m - k > \tfrac{n}{r_k}$;
      \item
       $b_{j,\vec \beta} \in \CC^{2m-k}(\partial \Omega;\B(E))$ for each $j = 1, \ldots, m$, $0 \leq k \leq m_j$ and $\abs{\vec \beta} = k$.
     \end{enumerate}
   \item
    ellipticity conditions:
    There is $\phi_\mathcal{A} \in [0,\pi)$ such that the following assertions hold true:
     \begin{enumerate}
      \item
       The principal symbol
        \[
         \mathcal{A}_\#(\vec z, \vec \xi)
          = \sum_{\abs{\vec \alpha} = 2m} a_{\vec \alpha}(\vec z) \vec \xi^{\vec \alpha}
        \]
       is parameter elliptic with angle of ellipticity less than $\phi_\mathcal{A}$, for each $\vec z \in \overline{\Omega} \cup \{\infty\}$.
      \item
       Lopatinskii--Shapiro condition:
       For each $\vec z \in \partial \Omega$, and all $\lambda \in \overline{\Sigma}_{\pi -\phi_\mathcal{A}}$, \, $\vec \xi \in \R^{n+1}$ with $\vec \xi \cdot \vec n(\vec z_0) = 0$, $(\lambda, \vec \xi) \neq (0,\vec 0)$, the initial value problem
        \begin{align*}
         \mathcal{A}_\#(\vec z, \vec \xi + \ii n \tfrac{\partial}{\partial y}) v(y \vec n)
          &= 0,
          &&y > 0,
          \\
         \mathcal{B}_j(\vec z, \vec \xi + \ii n \tfrac{\partial}{\partial y}) v(0)
          &= h_j,
          &&j = 1, \ldots, m
        \end{align*}
       has a unique solution $\vec v \in \CC_0(\R_+;E)$, for each $(h_1, \ldots, h_m)^\mathsf{T} \in E^m$.
     \end{enumerate}
  \end{enumerate}
 \end{assumption}
  
  \begin{theorem}
   \label{DHP03:Thm_8.2}
   Let $E$ be a Banach space of class $\mathcal{HT}$, $n, m \in \N$ and $p \in (1, \infty)$.
   Let $\Omega \subseteq \R^{n+1}$ be a domain with compact $\CC^{2m}$-boundary. Suppose that for $\phi_\mathcal{A} \in [0,\pi)$ the boundary value problem $(\mathcal{A}(\vec z, D), (\mathcal{B}_j(\vec z, D))_{j=1}^m)$ satisfies Assumption \ref{DHP03:Assmpt_Sec_8.1}.
   Let $A_B$ denote the $\LL_p(\Omega;E)$-realisation of the boundary value problem with domain
    \[
     \dom(A_B)
      = \{u \in \WW_p^{2m}(\Omega;E): \, \mathcal{B}_j(\vec z,D)u = 0 \, \text{on } \partial \Omega, \, j = 1, \ldots, m \}.
    \]
   Then, for each $\phi > \phi_\mathcal{A}$, there is $\mu_\phi \geq 0$ such that $\mu_\phi + A_B$ is $\mathcal{R}$-sectorial with $\phi_{\mu_\phi + A_B} \leq \phi$.
   In particular, if $\phi_\mathcal{A} < \tfrac{\pi}{2}$, then the parabolic initial-boundary value problem
    \begin{align*}
     \partial_t u + A_B u + \mu_\phi u
      &= f,
      &&t > 0,
      \\
     u(0)
      &= 0
    \end{align*}
   has the property of maximal regularity in $\LL_q(\R_+;\LL_p(\Omega;E))$ for each $q \in (1, \infty)$.
  \end{theorem}
  \begin{proof}
  The localisation procedure as presented in \cite[Section 8.2]{DeHiPr03} carries over almost literally; again one has to consider the operators $\mathcal{B}_j(D) \mathcal{P}_{j,k}$ for each pair $(j,k)$ separately, but the technique stays the same.
  \end{proof}
  
 \subsection{Inhomogeneous boundary data}
 In \cite{DeHiPr07}, the authors extended their results in \cite{DeHiPr03} to the case of inhomogeneous boundary data by identifying the optimal spaces for the data to obtain optimal $\LL_p$-estimates, or, slightly more general, optimal $\LL_p$-$\LL_q$-estimates.
 These results are heavily based on the results and techniques of \cite{DeHiPr03} plus multiplier theorems for operators acting on Banach spaces with property $\mathcal{HT}$ (i.e.\ UMD-spaces), on Gagliardo-Nirenberg type estimates as well as on optimal embedding results for Sobolev, Besov and Triebel-Lizorkin spaces. The methods used in \cite{DeHiPr07} translate more or less directly to the situation considered here; the main difference being adjustments on the optimal regularity spaces since the boundary operators do not have the same order $k$ in all subspaces $\ran (\mathcal{P}_{j,k})$.
 We introduce the following assumptions.
   \begin{enumerate}
    \item[{\bf (D)}]
     Assumptions on the data:
      \begin{enumerate}
       \item[(i)]
        $f \in \LL_p(J \times \Omega; E)$,
       \item[(ii)]
        $g_{j,k} \in \WW_p^{(1,2m) \cdot \kappa_{j,k}} (J \times \partial \Omega;E)$, where $\kappa_{j,k} = \frac{2m - k - 1/p}{2m}$, and we then set $g_j = \sum_{k=0}^{m_j} g_{j,k}$.
       \item[(iii)]
        $u_0 \in \BB_{pp}^{2m(1-1/p)}(\Omega;E)$,
       \item[(iv)]
        if $\kappa_{j,k} > 1/p$, then $\mathcal{B}_j(0,\vec z) \mathcal{P}_{j,k} u_0(\vec z) = g_{j,k}(0,\vec z)$ for $\vec z \in \partial \Omega$.
      \end{enumerate}
    \item[{\bf (E)}]
     Ellipticity of the interior symbol:
     For all $t \in J$, $\vec z \in \overline{\Omega}$ and $\vec \xi \in \S^{n-1}$ it holds that
      \[
       \sigma(\mathcal{A}(t,\vec z, \vec \xi)) \subseteq \C_0^+,
      \]
     i.e.\ $\mathcal{A}(t,\vec z,D)$ is \emph{normally elliptic}.
     If $\Omega$ is unbounded, the same condition is imposed at $\vec z = \infty$.
    \item[{\bf (LS)}]
     Lopatinskii-Shapiro condition:
     For all $t \in J$, $\vec z \in \partial \Omega$ and all $\vec \xi \in \R^n$ with $\vec \xi \cdot \vec n(\vec z) = 0$, and all $\lambda \in \overline{\C_0^+}$ such that $(\lambda, \vec \xi) \neq (0, \vec 0)$, the initial value problem
      \begin{align*}
       \lambda v(y) + \mathcal{A}_\#(t,\vec z,\vec \xi + \ii \vec n(\vec z) \tfrac{\partial}{\partial y})v(y)
        &= 0,
        &&y > 0,
        \\
       \mathcal{B}_{j,\#}(t,\vec z, \vec \xi + \ii \vec n(\vec z) \frac{\partial}{\partial y})v(0)
       &= h_j,
       &&j = 1, \ldots, m,
      \end{align*}
     has a unique solution in the class $v \in \CC_0(\R_+;E)$.
    \item[{\bf (SD)}]
     There are $r_l, s_l \geq p$ with $\frac{1}{s_l} + \frac{n}{2m r_l} < 1 - \frac{l}{2m}$ such that
      \begin{align*}
       a_{\vec \alpha}
        &\in \LL_{s_l}(J; [\LL_{r_l} + \LL_\infty](\Omega;\B(E))),
        &&\abs{\vec \alpha} = l < 2m,
        \\
       a_{\vec \alpha}
        &\in \CC_l(J \times \overline{\Omega}; \B(E)),
        &&\abs{\vec \alpha} = 2m.
      \end{align*}
    \item[{\bf (SB)}]
     There are $s_{j,k,l}, r_{j,k,l} \geq p$ with $\frac{1}{s_{j,k,l}} + \frac{n-1}{2m r_{j,k,l}} < \kappa_{j,k} + \frac{l-k}{2m}$ such that
      \[
       b_{j,k,\vec \beta} \in \WW_{s_{j,k,l},r_{j,k,l}}^{(1,2m) \cdot \kappa_{j,k}}(J \times \partial \Omega; \B(E)),
       \quad
       \abs{\vec \beta} = l \leq k \leq m_j.
      \] 
   \end{enumerate}
 The theorem on $\LL_p$-optimality then reads as follows.

 \begin{theorem}[$\LL_p$-$\LL_p$-optimality]
  \label{DHP07:Thm_2.1}
  Let $E$ be a Banach space of class $\mathcal{HT}$ and $\Omega \subseteq \R^n$ be a domain with compact boundary of class $\partial \Omega \in \CC^{2m}$.
  Let $p \in (1, \infty)$ and suppose that assumptions {\bf (E)}, {\bf (LS)}, {\bf (SD)} and {\bf (SB)} hold true. Then the problem \eqref{DHP07:2.2}
%   \begin{align}
%    \partial_t u + \mathcal{A}(t,\vec z,D)u
%     &= f(t,\vec z),
%     &&t \in J, \, \vec z \in \Omega,
%     \nonumber \\
%    \mathcal{B}_j(t,\vec z,D)u
%     &= g_j(t,\vec z),
%     &&t \in J, \, \vec z \in \partial \Omega, \, j = 1, \ldots, m,
%     \label{DHP07:2.2}
%     \\
%    u(0,\vec z)
%     &= u_0(\vec z),
%     &&\vec z \in \Omega
%     \nonumber
%   \end{align}
  has a unique solution in the class
   \[
    u \in \WW_p^{(1,2)}(J \times \Omega; E)
   \]
  if and only if the data $f$, $\vec g$ and $u_0$ are subject to conditions {\bf (D)}.
 \end{theorem}
 
 For $\LL_p$-$\LL_q$-estimates with $p \neq q$, a similar result holds true, but the smoothness conditions on the coefficients and on the data have to be slightly modified, cf.\ \cite[Theorem 2.3]{DeHiPr07}.
  \begin{enumerate}
   \item[{\bf (D1)}]
    Assumptions on the data in case $p \neq q$:
     \item[(i)]
      $f \in \LL_p(J;\LL_q(\Omega;E))$,
     \item[(ii)]
      $g_{j,k} \in F_{pq}^{\kappa_{j,k}}(J;\LL_q(\partial \Omega;E)) \cap \LL_p(J;\BB_{qq}^{2m\kappa_{j,k}}(\partial \Omega;E))$ with $\kappa_{j,k} = \frac{2m-k-1/q}{2m}$, and we then set $g_j = \sum_{k=0}^{m_j} g_{j,k}$,
     \item[(iii)]
      $u_0 \in \BB_{qp}^{2m(1-1/p)}(\Omega;E)$,
     \item[(iv)]
      if $\kappa_{j,k} > 1/q$, then $\mathcal{B}_j(0,\vec z) \mathcal{P}_{j,k}u_0(\vec z) = g_{j,k}(0,\vec z)$ for $\vec z \in \partial \Omega$.
   \item[{\bf (SD1)}]
    There are $s_l \geq p$ and $r_l \geq q$ with $\frac{1}{s_l} + \frac{n}{2m r_l} < 1 - \frac{l}{2m}$ such that
     \begin{align*}
      a_{\vec \alpha}
       &\in \LL_{s_l}(J; (\LL_{r_l} + \LL_\infty)(\Omega;\B(E))),
       &&\abs{\vec \alpha} = l < 2m,
       \\
      a_{\vec \alpha}
       &\in \CC_l(J \times \overline{\Omega}; \B(E)),
       &&\abs{\vec \alpha} = 2m.
     \end{align*}
   \item[{\bf (SB1)}]
    There are $s_{j,k,l} \geq p$ and $r_{j,k,l} \geq q$ with $\frac{1}{s_{j,k,l}} + \frac{n-1}{2m r_{j,k,l}} < \kappa_{j,k} + \frac{k-l}{2m}$ such that
     \[
      b_{j,k,\vec \beta}
       \in \WW_{s_{j,k,l},r_{j,k,l}}^{(1,2m) \cdot \kappa_{j,k}}(J \times \partial \Omega; \B(E))),
       \quad
       \abs{\vec \beta} = l \leq k \leq m_j.
     \]
  \end{enumerate}
  
  \begin{theorem}[$\LL_p$-$\LL_q$-optimality]
  Let $E$ be a Banach space of class $\mathcal{HT}$ and $\Omega \subseteq \R^n$ be a domain with compact boundary of class $\partial \Omega \in \CC^{2m}$.
  Let $p,q \in (1, \infty)$ and suppose that assumptions {\bf (E)}, {\bf (LS)}, {\bf (D1)}, {\bf (SD1)} and {\bf (SB1)} hold true. Then the problem
   \begin{align}
    \partial_t u + \mathcal{A}(t,\vec z,D)u
     &= f(t,\vec z),
     &&t \in J, \, \vec z \in \Omega,
     \nonumber \\
    \mathcal{B}_j(t,\vec z,D)u
     &= g_j(t,\vec z),
     &&t \in J, \, \vec z \in \partial \Omega, \, j = 1, \ldots, m,
     \label{DHP07:2.2}
     \\
    u(0,\vec z)
     &= u_0(\vec z),
     &&\vec z \in \Omega
     \nonumber
   \end{align}
  has a unique solution in the class
   \[
    u \in \WW_{pq}^{(1,2)}(J \times \Omega; E)
   \]
  if and only if the data $f$, $g$ and $u_0$ are subject to conditions {\bf (D1)}.
  \end{theorem}

\end{document}